\documentclass[a4paper,10pt]{article}
\usepackage{graphicx}        % standard LaTeX graphics tool
\usepackage{multicol}        % used for the two-column index
\usepackage[bottom]{footmisc}% places footnotes at page bottom

\usepackage{epsfig}
\usepackage{amsfonts,amsthm}
\usepackage{authblk}
\usepackage{amssymb}
\usepackage{graphicx}
\usepackage{float}
\usepackage{diagbox}
\usepackage[applemac]{inputenc}
\usepackage[T1]{fontenc}
\usepackage{blkarray}
\usepackage{color}
\usepackage{amsmath,bm}
\usepackage{mathabx}
\usepackage{mathtools}
\usepackage{nicefrac}
\usepackage[hmargin={15mm,15mm},vmargin={17mm,22mm}]{geometry}
\usepackage{amssymb}
\usepackage{array}
\usepackage{multicol}
\usepackage{pdfpages}
\usepackage{subfig}
\usepackage{multirow}
\usepackage{enumitem}
\usepackage{bbold}
\usepackage{xparse}								% for \DeclareDocumentCommand
\usepackage{stmaryrd}							% for symbols $\llbracket a+b\rrbracket$
\usepackage{sidecap}								% for SCfigure
\usepackage{upgreek}
\usepackage{accents}
\allowdisplaybreaks
\usepackage{tcolorbox}
\usepackage{algorithmic}
\usepackage[ocgcolorlinks,allcolors={blue}]{hyperref}

\usepackage[bbgreekl]{mathbbol}
\usepackage{amsfonts}

\DeclareSymbolFontAlphabet{\mathbb}{AMSb}
\DeclareSymbolFontAlphabet{\mathbbl}{bbold}

\usepackage{bbm}

\def\N{\mathbb N}
\def\R{\mathbb R}

\def\P{\mathbb P}

\def\D{\mathcal{D}}
\def\S{\mathcal{S}}

\def\<{\langle}
\def\>{\rangle}

\def\lsim{\lesssim}

\def\Chi{\raise .3ex \hbox{\large $\chi$}}

\def\[{\Bigl [}
\def\]{\Bigr ]}
\def\({\Bigl (}
\def\){\Bigr )}
\def\l{\iota}

\def\dsp{\displaystyle}
\def\x{{\bf x}}

\def\n{{\bf n}}

\def\U{{\bf U}}

\def\q{{\bf q}}
\def\r{{\bf r}}
\def\g{{\mathbf{g}}}

\def\x{{\bf x}}
\def\dsp{{\displaystyle x}}
\def\d{{\rm d}}
\def\div{{\rm div}}
\def\a{\alpha}
\def\q{\mathbf{q}}
\def\K{\mathbb{K}}
\def\n{\mathbf{n}}
\def\dsp{\displaystyle}
\def\bu{\mathbf{u}}
\def\bv{\mathbf{v}}

\def\O{\Omega}
\def\G{\Gamma}

\def\bbsig{\bbsigma}
\def\bbeps{\bbespilon}
\newcommand{\jump}[1]{\llbracket #1 \rrbracket}

\newcommand{\dtn}{\delta t^{n+\frac{1}{2}}}
\newcommand{\dtnmun}{\delta t^{n-\frac{1}{2}}}
\newcommand{\dtnmdeux}{\delta t^{n-\frac{3}{2}}}

\newcommand{\dtzero}{\delta t^{\frac{1}{2}}}

\newcommand{\nf}{\nicefrac}

\newcommand{\email}[1]{\href{mailto:#1}{#1}}

\begingroup
    \makeatletter
    \@for\theoremstyle:=definition,remark,plain\do{%
        \expandafter\g@addto@macro\csname th@\theoremstyle\endcsname{%
            \addtolength\thm@preskip\parskip
            }%
        }
\endgroup

\newtheorem{theorem}{Theorem}
\numberwithin{theorem}{section}
\newtheorem{proposition}[theorem]{Proposition}
\newtheorem{lemma}[theorem]{Lemma}

\theoremstyle{remark}
\newtheorem{remark}[theorem]{Remark}
\theoremstyle{definition}

\newtheorem{definition}[theorem]{Definition}

\newenvironment{acknowledgements}
      {\bigskip\bigskip~\newline\textbf{Acknowledgements}}

\def\g{{\rm nw}}
\def\l{{\rm w}}
\def\rt{{\rm rt}}

\def\trace{\gamma}
\def\grad{\nabla}
\def\del{\partial}
\def\div{{\rm div}}

\def\weakto{\rightharpoonup}

\definecolor{labelkey}{rgb}{0.6,0,1}

\usepackage[normalem]{ulem}
 \newcounter{corr}
 \definecolor{violet}{rgb}{0.580,0.,0.827}
 \newcommand{\corr}[3]{\typeout{Warning : a correction remains in page
 \thepage}
 				\stepcounter{corr}        
 				{\color{blue}\ifmmode\text{\,\sout{\ensuremath{#1}}\,}\else\sout{#1}\fi}
         {\color{red}#2}
         {\color{violet} \ifmmode\text{#3}\else #3\fi }}

\usepackage{parskip}
\usepackage{algorithmic}

\usepackage{diagbox}
\usepackage{booktabs}

\makeatletter
\def\thm@space@setup{%
  \thm@preskip=\parskip \thm@postskip=0pt
}
\makeatother

\begin{document}
\title{Gradient discretization of two-phase flows coupled with mechanical deformation in fractured porous media%\footnote{This work has been supported by SIR Research Grant no.~RBSI14VTOS funded by MIUR -- Italian Ministry of Education, Universities, and Research, and by ``National Group of Scientific Computing'' (GNCS-INdAM).}
}
\author[1]{{Francesco Bonaldi}\footnote{\email{francesco.bonaldi@univ-cotedazur.fr}}}
\author[1]{{Konstantin Brenner}\footnote{\email{konstantin.brenner@univ-cotedazur.fr}}}
\author[2]{{J\'er\^ome Droniou}\footnote{Corresponding author, \email{jerome.droniou@monash.edu}}}
\author[1]{{Roland Masson}\footnote{\email{roland.masson@univ-cotedazur.fr}}}
\affil[1]{Universit\'e C\^ote d'Azur, Inria, CNRS, Laboratoire J.A. Dieudonn\'e, team Coffee, France}%
\affil[2]{School of Mathematics, Monash University, Victoria 3800, Australia}%
\date{}
\maketitle
\begin{abstract}
\noindent
We consider a two-phase Darcy flow in a fractured porous medium consisting in a matrix flow coupled with a tangential flow in the fractures, described as a network of planar surfaces. This flow model is also coupled with the mechanical deformation of the matrix assuming that the fractures are open and filled by the fluids, as well as small deformations and a linear elastic constitutive law. The model is discretized using the gradient discretization method \cite{gdm}, which covers a large class of conforming and non conforming schemes. This framework allows for a generic convergence analysis of the coupled model using a combination of discrete functional tools. Here, we describe the model together with its numerical discretization and, using discrete compactness techniques, we prove a convergence result (up to a subsequence) assuming the non-degeneracy of the phase mobilities and that the discrete solutions remain physical in the sense that, roughly speaking, the porosity does not vanish and the fractures remain open. This is, to our knowledge, the first convergence result for this type of model taking into account two-phase flows in fractured porous media and the non-linear poromechanical coupling. Previous related works consider a linear approximation obtained for a single phase flow by freezing the fracture conductivity \cite{GWGM2015,hanowski2018}. Numerical tests employing the Two-Point Flux Approximation (TPFA) finite volume scheme for the flows and $\P_2$ finite elements for the mechanical deformation are also provided to illustrate the behavior of the solution to the model.
\end{abstract}
%
%%-----------------------------
%%      your text
%%-----------------------------

\section{Introduction}

Many real-life applications in geosciences involve processes like multi-phase flow and hydromechanical coupling in heterogeneous porous media. Such mathematical models are coupled systems of partial differential equations, including non-linear and degenerate parabolic ones. Besides the inherent difficulties posed by such equations, further complexities stem from the heterogeneity of the medium and the presence of discontinuities like fractures. This has a strong impact on the complexity of the models, challenging their mathematical and numerical analysis and the development of efficient simulation tools.

This work focuses on the so called hybrid-dimensional matrix fracture models obtained by averaging both the unknowns and the equations across the fracture width and by imposing appropriate transmission conditions at the matrix fracture interfaces.
Given the high geometrical complexity of real-life fracture networks, the main advantages of these hybrid-dimensional compared to full-dimensional models are to facilitate the mesh generation and the discretization of the model, and to reduce the computational cost of the resulting schemes. 
This type of hybrid-dimensional models has been the object of intensive researches over the last  twenty years due  to the ubiquity of fractures in geology and their large impact on flow, transport and mechanical behavior of rocks.
For the derivation and analysis of such models, let us refer to \cite{MAE02,FNFM03,KDA04,MJE05,ABH09,BGGLM16,BHMS2016,NBFK2019} for single-phase Darcy flows,  \cite{BMTA03,RJBH06,MF07,Jaffre11,BGGM2015,DHM16,BHMS2018,gem.aghili} for two-phase Darcy flows, and  \cite{MK2013,KTJ2013,JJ14,GWGM2015,hanowski2018,GKT16,JZ17,GFSZ17,UKBN2018} for poroelastic models.

In this article, we consider the two-phase Darcy flow in a network of pre-existing fractures represented as
$(d-1)$-dimensional planar surfaces coupled with the surrounding $d$-dimensional matrix.  The fractures are assumed to be open and filled by the fluids.
Both phase pressures are assumed continuous across the fractures. This is a classical assumption for open fractures given the low pressure drop in the width of the fractures \cite{BMTA03,RJBH06,MF07,BGGM2015}.
For single-phase flows, Poiseuille's law is classically used to model the flow along the fractures. This leads to a Darcy-like tangential flow with conductivity equal to ${d_f^3 \over 12}$, where $d_f$ is the fracture aperture \cite{GWGM2015,hanowski2018}. Following \cite{MK2013}, the extension to a two-phase flow is based on the generalized Darcy laws involving appropriate relative permeabilities and the capillary pressure-saturation relation.
This hybrid-dimensional two-phase Darcy flow model is coupled with the matrix mechanical deformation assuming small strains and a linear poroelastic behavior \cite{MK2013,KTJ2013,JJ14}.  The extension of the single-phase poromechanical coupling \cite{GWGM2015,hanowski2018,GKT16,JZ17,UKBN2018} to two-phase Darcy flows is based on the so-called equivalent pressure used both in the matrix for the total stress
and at both sides of the fractures as boundary condition for the mechanics. Typically, the equivalent pressure is defined as a convex combination of the phase pressures and several different combinations have been proposed in the literature \cite{NL08}. Our choice of the equivalent pressure follows the pioneer monograph by Coussy~\cite{coussy}  and involves the capillary energy which, as already noticed in \cite{KTJ2013,JJ14}, plays a key role to obtain energy estimates for the coupled system. From the open fracture assumption, the fracture mechanical behavior reduces to the continuity of the normal stresses at both sides of the fracture matching with the fracture equivalent pressure times the unit normal vector. To our best knowledge, no theoretical or numerical analysis of the complete poromechanical model, with all non-linear coupling, has been carried out so far.

In this work, the hybrid-dimensional coupled model is discretized using the gradient discretization method (GDM) \cite{gdm}.
This framework is based on abstract vector spaces of discrete unknowns combined with reconstruction operators.
The gradient scheme is then obtained by substitution of the continuous operators by their discrete counterparts in the weak formulation of the coupled model. The main asset of this framework is to enable a generic convergence analysis based on general properties of the discrete operators  that hold for a large class of conforming and non conforming discretizations. 
Two essential ingredients to discretize the coupled model are the discretizations of the hybrid-dimensional two-phase Darcy flow  and the discretization of the mechanics. Let us briefly mention, in both cases, a few families of discretizations typically satisfying the gradient discretization properties. 
For the discretization of the Darcy flow, the gradient discretization framework typically covers the case of cell-centered finite volume schemes with Two-Point Flux Approximation on strongly admissible meshes \cite{KDA04,ABH09,gem.aghili}, or some symmetric Multi-Point Flux Approximations \cite{TFGCH12,SBN12,AELHP153D} on tetrahedral or hexahedral meshes. It also accounts for the families of Mixed Hybrid Mimetic and Mixed or Mixed Hybrid Finite Element discretizations such as in \cite{MAE02,MJE05,BGGLM16,BHMS2016,AFSVV16}. The case of vertex-based discretizations such as Control Volume Finite Element approaches (i.e.~conforming finite element with mass lumping) \cite{BMTA03,RJBH06,MF07} or the Vertex Approximate Gradient scheme \cite{BGGLM16,BHMS2016,BGGM2015,DHM16,BHMS2018} is also accounted for.
For the discretization of the elastic mechanical model, the gradient discretization framework covers conforming finite element methods such as in \cite{GWGM2015}, as well as the Crouzeix-Raviart discretization ~\cite{hansbo-larson,dipietro-lemaire}, Discontinuous Galerkin methods~\cite{gdm-dg}, the Hybrid High Order discretization ~\cite{dipietro-ern}, and the Virtual Element Method~\cite{beirao-brezzi-marini}. Note that many of these methods are actually applicable to both the flow and the mechanical component of the model.

Without taking into account the poromechanical coupling, convergence results have been obtained in \cite{ABH09,MAE02,MJE05,BGGLM16,BHMS2016} for hybrid-dimensional single-phase Darcy flow models, and in \cite{BGGM2015,DHM16} for hybrid-dimensional two-phase Darcy flow models. The well-posedness and convergence analysis of single-phase poromechanical models is studied in \cite{GWGM2015,hanowski2018}. Nevertheless those analyses consider a linear approximation of the coupled model obtained by freezing the fracture conductivity ${d_f^3 \over 12}$, and hence eliminating the non-linear coupling between the fracture aperture and the Darcy flow. Let us also mention the related recent work \cite{both2019global} on unsaturated poroelasticity based on the Richards approximation of the two-phase flow model, using partial linearizations, non-degeneracy conditions and Kirchhoff transformation (which is made possible by assuming that the saturation--capillary pressure law is uniform across the domain). Note that fractures are not considered in this work.

Our main result is the proof of convergence, in the GDM setting, of the approximate solutions to the weak solution of the non-linear coupled model with two-phase flows. To our best knowledge, this is the first convergence result for this type of hybrid-dimensional model taking into account the full non-linear poromechanical coupling. Since it is based on discrete compactness techniques, the convergence is that of a subsequence of approximate solutions (precisely, we prove that sequences of approximate solutions are compact, and that any of their limit points is a weak solution of the continuous model). To establish this result, we make the following main assumptions. It is first assumed that the approximate matrix porosity remains bounded below by a strictly positive constant and that the approximate fracture aperture remains larger than some given aperture vanishing only at the tips.  Let us point out that these assumptions are due to the limitations of the model itself rather than to the shortcomings of the numerical analysis. They cannot be avoided since the continuous model does not ensure the positivity of the porosity nor of the fracture aperture, properties needed to guarantee existence of solutions. We note that previous works on similar models circumvent these limitations by linearization processes (complete or partial freezing of the matrix porosity and fracture apertures). Regarding the assumption on the fracture aperture, it could possibly be overcome by introducing contact mechanics in the model \cite{GKT16,Runar19}. This direction will be investigated in a future work. 
It is also assumed in the numerical analysis that the mobility functions are bounded below by strictly positive constants. Independently of the poromechanical coupling, this is a classical assumption to enable the stability and convergence analysis of two-phase Darcy flows with spatial discontinuity of the capillary pressure functions, as it is always the case in the presence of fractures  (see \cite{EGHM13,BGGM2015,DHM16}). To our knowledge, the only convergence analyses covering both the degeneracy of the mobilities and discontinuous capillary pressures are limited to Two-Point Flux Approximations (see \cite{brenner2013finite,quenjel2020}). Extending such analyses, even considering only the TPFA method for the flow, to the poromechanical model considered here is far from straightforward and seems to bring additional challenges; given that our analysis is already quite technical, we postpone this extension to degenerate mobility functions to a future work.

The rest of the article is organized as follows. Section~\ref{sec:modeleCont} introduces the continuous hybrid-dimensional coupled model. Section~\ref{sec:gradientscheme} describes the gradient discretization method for the coupled model including the definition of the reconstruction operators, the discrete variational formulation and the properties of the gradient discretization needed for the subsequent convergence analysis.
Section~\ref{sec:convergence} proceeds with the convergence analysis. The a priori estimates are established in Subsection~\ref{subsec:apriori}, the compactness properties in Subsection~\ref{subsec:compactness} and the convergence to a weak solution is proved in Subsection~\ref{subsec:convergence}. 
In Section~\ref{sec:numerical.example}, numerical experiments based on the Two-Point Flux Approximation finite volume scheme for the flows and second-order finite elements for the mechanical deformation are carried out for a cross-shaped fracture network in a two-dimensional porous medium, and illustrate the numerical convergence of the solution.  
Appendices~\ref{appendix.a1} and~\ref{appendix.a2} state some technical results used in the convergence analysis.

\section{Continuous model}\label{sec:modeleCont}

We consider a bounded polytopal domain $\Omega$ of $\R^d$, $d\in\{2,3\}$, partitioned
into a fracture domain $\Gamma$ and a matrix domain $\O\backslash\overline\G$.
The network of fractures is defined by 
$$
\overline \Gamma = \bigcup_{i\in I} \overline \Gamma_i
$$  
where each fracture $\Gamma_i\subset \Omega$, $i\in I$ is a planar polygonal simply connected open domain with angles strictly lower than $2\pi$. Without restriction of generality, we will assume that the fractures may intersect exclusively at their boundaries, that is for any $i,j \in I, i\neq j$ one has $\Gamma_i\cap  \Gamma_j = \emptyset$, but not necessarily $\overline{\Gamma}_i\cap \overline{\Gamma}_j = \emptyset$. 
Since one can split  a general (non-simply connected) planar polygon into several simply connected pieces intersecting only at their boundaries (see Figure \ref{fig_network}) our assumptions on the fracture network are in fact quite general. Roughly speaking we only exclude the non-planar fractures.

Since the fractures are assumed open with no contact, we also have to assume in the following that the boundary of each connected component of $\Omega\setminus\Gamma$ has a non zero measure intersection with $\partial\Omega$.

\begin{figure}
\begin{center}
\includegraphics[width=0.35\textwidth]{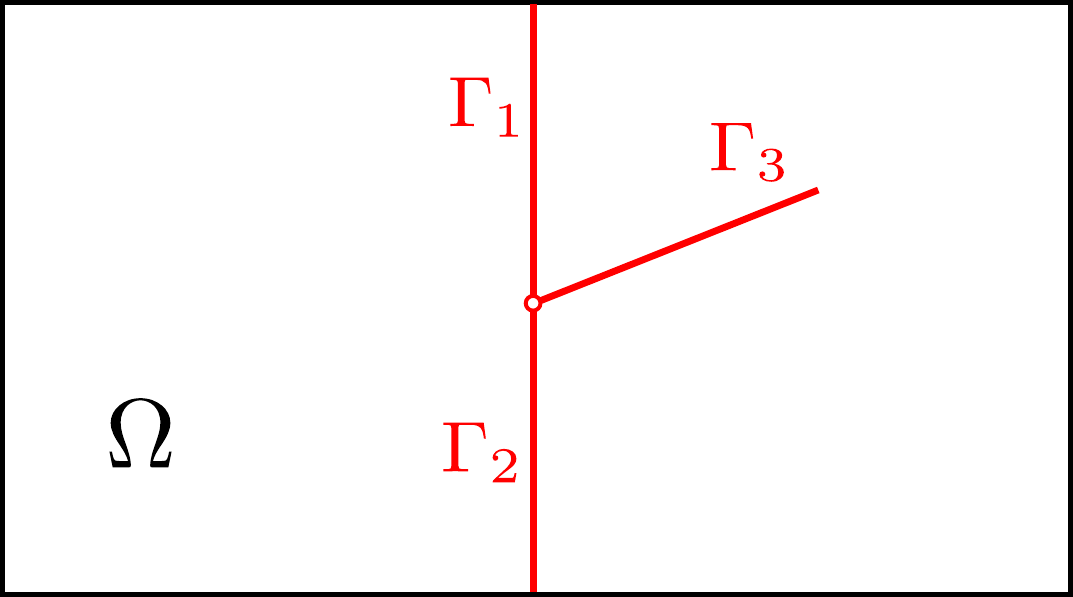}   
\caption{Example of a 2D domain $\Omega$ with three intersecting fractures $\Gamma_i$, $i\in\{1,2,3\}$.}
\label{fig_network}
\end{center}
\end{figure}

The two sides of a given fracture of $\Gamma$ are denoted by $\pm$ in the matrix domain, with unit normal vectors $\n^\pm$ oriented outward of the sides $\pm$. We denote by $\gamma$ the trace operator on $\Gamma$ for functions in $H^1(\Omega)$, by $\gamma_{\del\Omega}$ the trace operator for the same functions on $\del\O$, and by $\jump{\cdot}$ the normal trace jump operator on $\Gamma$ for functions in $H_\div(\O\backslash\overline\Gamma)$, defined by
$$
\jump{\bar\bu} = \bar\bu^+ \cdot \n^+ + \bar\bu^- \cdot \n^-  \, \mbox{ for all } \, \bar\bu \in H_\div(\O\backslash\overline\Gamma). 
$$
We denote by $\nabla_\tau$ the tangential gradient and by $\div_\tau$ the tangential divergence on the fracture network $\Gamma$. The symmetric gradient operator $\bbeps$ is defined such that $\bbeps(\bar\bv) = {1\over 2} (\nabla \bar\bv +^t\!(\nabla \bar\bv))$ for a given vector field $\bar\bv\in H^1(\O\backslash\overline\Gamma)^d$.

The fracture aperture, denoted by $\bar d_f$, is defined by $\bar d_f = - \jump{\bar \bu}$ for a displacement field $\bar\bu \in H^1(\O\backslash\overline\Gamma)^d$. 

Let us fix a continuous function $d_0: \Gamma \to (0,+\infty)$  vanishing at
$\partial \Gamma \setminus (\partial\Gamma\cap \partial\Omega)$ (i.e.~at the tips of $\Gamma$) and taking strictly positive values at $\partial\Gamma\cap \partial\Omega$.
The discrete fracture aperture will be assumed to be greater than or equal to $d_0$ almost everywhere (by the established convergence result, the same will hold for its limit). We note that the assumptions on $d_0$ are minimal, allowing for very general behavior of the fracture aperture at the tips.

Let us introduce some relevant function spaces:  
\begin{equation}\label{displacement}
\U_0 =\{ \bar \bv\in (H^1(\O\backslash\overline\Gamma))^d \mid \trace_{\del\O} \bar \bv = 0\}
\end{equation}
for the displacement vector, and
\begin{equation}\label{pressures}
V_0 = \{\bar v\in H^1_0(\O) \mid \gamma \bar v \in H^1_{d_0}(\G)\}
\end{equation}
for each phase pressure, where the space $H_{d_0}^1(\Gamma)$ is made of functions $v_\Gamma$ in $L^2(\Gamma)$,
such that $d_0^{\nf 3 2} \nabla_\tau v_\Gamma$ is in  $L^2(\Gamma)^{d-1}$, 
and whose traces are continuous at fracture intersections $\del\G_i\cap
\del\G_j$, $(i,j)\in I\times I$ ($i\neq j$) and vanish on the boundary $\partial \Gamma\cap \partial\Omega$.

The matrix and fracture rock types are denoted by the indices ${\rm rt} = m$ and ${\rm rt} = f$, {respectively}, and the non-wetting and wetting phases by the superscripts $\alpha =\g$ and $\alpha=\l$, {respectively}. 
Each rock type ${\rm rt} \in \left\{ m, f \right\}$ is characterized by its own set of mobility functions $\left( \eta^\alpha_{\rm rt}\right)_{\alpha \in \{ \g,\l\}}$ and capillary pressure-saturation relation $\left( S^\alpha_{\rm rt} \right)_{\alpha \in\{ \g,\l\}}$.

\begin{figure}
  \begin{center}
	\includegraphics[width=6cm]{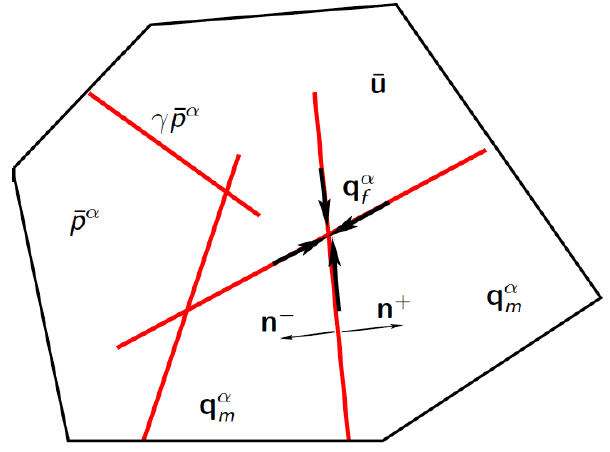}
	\caption{Example of a 2D domain $\Omega$  with its fracture network $\Gamma$, the unit normal vectors $\n^\pm$ to $\Gamma$, the phase pressures $\bar p^\alpha$ in the matrix and $\gamma\bar p^\alpha$ in the fracture network, the displacement vector {field} $\bar \bu$, the matrix Darcy velocities $\q^\alpha_m$ and the fracture tangential Darcy velocities $\q^\alpha_f$ integrated along the fracture width. }
        \end{center}
        \end{figure}

The PDEs model reads: find the phase pressures $\bar p^\alpha$, $\alpha\in\{\g,\l\}$, and the displacement vector {field} $\bar \bu$, both satisfying homogeneous Dirichlet boundary conditions on $\partial\Omega$, such that $\bar p_c = \bar p^\g - \bar p^\l$ and, for $\alpha\in\{\g,\l\}$, 
\begin{equation}
\label{eq_edp_hydromeca} 
\left\{
\begin{array}{lll}
&\partial_t\left(  \bar\phi_m S^{\a}_m(\bar p_c) \right) + \div \left(  \q^\a_m \right) = h_m^\a & \mbox{ on } (0,T)\times \Omega\setminus\overline\Gamma,\\[1ex]
& \q^\a_m = \dsp - \eta^\a_m ( S^{\a}_m(\bar p_c)) \K_m  \nabla \bar p^\a & \mbox{ on } (0,T)\times \Omega\setminus\overline\Gamma,\\[1ex]
&\partial_t\left(  \bar d_f S^{\a}_f(\gamma \bar p_c) \right)
+ \div_\tau (  \q^\a_f ) -  \jump{\q^\a_m} = h_f^\a & \mbox{ on } (0,T)\times \Gamma,\\[1ex]
& \q^\a_f = \dsp -  \eta^\a_f (S^{\a}_f(\gamma \bar p_c)) ({1\over 12} \bar d_f^3) \nabla_\tau \gamma \bar p^\a  & \mbox{ on } (0,T)\times \Gamma,\\[1ex]
& -\div \(\bbsig(\bar\bu) - b ~ \bar p^E_m{\mathbb I}\)= \mathbf{f} & \mbox{ on } (0,T)\times \Omega\setminus\overline\Gamma\\[1ex]
& \bbsig(\bar\bu) =  2\mu ~\bbeps(\bar\bu) + \lambda ~\div(\bar\bu) ~\mathbb{I} & \mbox{ on } (0,T)\times \Omega\setminus\overline\Gamma,  
\end{array}
\right.
\end{equation}
with 
\begin{equation}
\label{closure_laws}
\left\{
\begin{array}{lll}
& \partial_t \bar\phi_m = \dsp b~\div \partial_t \bar\bu + \frac{1}{M} \partial_t \bar p^E_m & \mbox{ on } (0,T)\times \Omega\setminus\overline\Gamma,\\  [1ex]
& {(\bbsig(\bar\bu) - b ~ \bar p^E_m \mathbb I)\n^\pm = - \bar p^E_f \n^\pm}  & \mbox{ on } (0,T)\times \Gamma,\\[1ex]
& \bar d_f = -\jump{\bar\bu}   & \mbox{ on } (0,T)\times \Gamma,  
\end{array}
\right.
\end{equation}
and the initial conditions
$$
\bar p^\a|_{t=0} = \bar p^{\a}_0,  \quad \bar \phi_m|_{t=0}= \bar\phi_m^0. 
$$
Here, we have denoted by $\bar p_c$ the capillary pressure, and the equivalent pressures $\bar p^E_m$ and $\bar p^E_f$ are defined, following~\cite{coussy}, by
$$
\bar p^E_{m} = \dsp \sum_{\alpha \in \{\g,\l\} } \bar p^\alpha~S^\alpha_m(\bar p_c)  - U_{m}(\bar p_c), \quad 
\bar p^E_{f} = \dsp \sum_{\alpha \in \{\g,\l\} } \gamma \bar p^\alpha~S^\alpha_{f}(\gamma \bar p_c)  - U_{f}(\gamma \bar p_c),
$$
where
\begin{equation}\label{eq.capillary_energy}
U_{\rm rt}(\bar p_c) =  \int_0^{\bar p_c} z  \left( S^{\g}_{{\rm rt}} \right)'(z) \,\d z
\end{equation}
is the capillary energy density function of the  rock type ${\rm rt}\in \{m,f\}$. As already noticed in \cite{KTJ2013,JJ14}, this is a key choice to obtain the energy estimates that are the starting point for the convergence analysis.

We make the following main assumptions on the data:
\begin{enumerate}[label=(H\arabic*),leftmargin=*]
\item For each phase $\a \in\{\g,\l\}$ and rock type ${\rm rt}\in\{m,f\}$, the mobility function $\eta^\a_{\rm rt}$ is continuous, non-decreasing, and there exist
  $0< \eta_{\rm rt,{\rm min}}^\a \leq \eta^\a_{\rm rt,{\rm max}}< +\infty$  such that 
  $\eta^\a_{\rm rt,{\rm min}} \leq \eta^\a_{\rm rt}(s) \leq \eta^\a_{\rm rt,{\rm max}}$ for all $s\in [0,1]$.
  \label{first.hyp}
\item\label{H2} For each rock type ${\rm rt}\in\{m,f\}$, the non-wetting phase saturation function $S^\g_{\rm rt}$ is a non-decreasing {Lipschitz}
continuous function with values in $[0,1]$, and   $S^\l_{\rm rt} = 1 - S^\g_{\rm rt}$.
\item $b\in [0,1]$ is the Biot coefficient, $M> 0$ is the Biot modulus, and $\lambda >0$, $\mu>0$ are the Lam\'e coefficients. These coefficients are assumed to be constant for simplicity.
\item The initial pressures are such that $\bar p^\alpha_0 \in V_0\cap L^\infty(\Omega)$ and $\gamma\bar p^\alpha_0\in L^\infty(\Gamma)$, $\alpha\in \{\g,\l\}$; the initial porosity is such that $\bar \phi^0_m \in L^\infty(\Omega)$.
\item The source terms satisfy $\mathbf f \in L^2(\Omega)^d$, $h_m^\alpha\in L^2((0,T)\times \Omega),$ and $h_f^\alpha \in L^2((0,T)\times \Gamma)$.
 \item The matrix permeability tensor $\K_{m}$ is symmetric and uniformly elliptic on $\Omega$. Note that the variation of the matrix permeability with the porosity is neglected. 
 \label{last.hyp}
\end{enumerate}

The notion of weak solution for \eqref{eq_edp_hydromeca}--\eqref{closure_laws} is classically obtained multiplying each flow equation and the mechanical equation by a separate test function, integrating by parts and, for each phase, adding together the equations resulting from the flows in the matrix and the fractures. When the capillary pressure has continuous first temporal and second spatial derivatives in $(0,T)\times (\Omega\backslash\Gamma)$, its trace has continuous first temporal and second tangential derivatives in $(0,T)\times \Gamma$, and the displacement has continuous second spatial derivatives, the following weak formulation is equivalent to the PDE model.

\begin{definition}[Weak solution of the model]
  A \emph{weak solution} of the model is given by $\bar p^\alpha \in L^2(0,T;V_0)$, $\alpha\in \{\g,\l\}$, and $\bar \bu \in L^\infty(0,T;\U_0)$, such that, for any $\alpha\in \{\g,\l\}$, $\bar{d}_f^{\;\nf 3 2}\nabla_\tau\gamma\bar p^\alpha\in L^2((0,T)\times\G)^{d-1}$ and, for all $\bar \varphi^\alpha\in C_c^\infty([0,T)\times\Omega)$ and all smooth functions $\bar \bv:[0,T]\times (\O\setminus\overline\G)\to\R^d$ vanishing on $\partial\Omega$ and whose derivatives of any order admit finite limits on each side of $\Gamma$,
\begin{equation}
\label{eq_var_hydro}      
\left.
\begin{array}{ll}
&\dsp \int_0^T \int_\O \(-\bar \phi_m S^{\a}_m(\bar p_c) \partial_t \bar \varphi^\a + \eta^\a_m ( S^{\a}_m(\bar p_c)) \K_m  \nabla \bar p^\a \cdot \nabla \bar \varphi^\a\) \d\x \d t \\[2ex]
  &   + \dsp \int_0^T \int_\G \( -\bar d_f S^{\a}_f(\gamma \bar p_c) \partial_t \gamma \bar \varphi^\a  + \eta^\a_f ( S^{\a}_f(\gamma \bar p_c)) {\bar d_f^{\;3}\over 12}  \nabla_\tau \gamma \bar p^\a \cdot \nabla_\tau \gamma \bar \varphi^\a \) \d\sigma(\x) \d t\\[2ex]
  &  - \dsp \int_\O \bar \phi_{m}^0 S^{\a}_m(\bar p_c^0) \bar \varphi^\a(0,\cdot)\d\x
  - \int_\G \bar d_f^0 S^{\a}_f(\gamma \bar p_c^0) \gamma \bar \varphi^\a(0,\cdot)\d\sigma(\x) \\[2ex]
  & = \dsp \int_0^T \int_\O h_m^\a \bar \varphi^\a \d\x \d t + \int_0^T \int_\G h_f^\a \gamma \bar \varphi^\a  ~\d\sigma(\x) \d t,
\end{array}
\right.
\end{equation}
\begin{equation}
  \label{eq_var_meca}
\begin{array}{ll}
  & \dsp \int_0^T \int_\O \( \bbsig(\bar \bu): \bbeps(\bar \bv) - b ~\bar p_m^E \div(\bar \bv)\) \d\x \d t + \int_0^T \int_\G \bar p_f^E ~\jump{\bar \bv}  ~\d\sigma(\x) \d t \\[2ex]
&  \dsp = \int_0^T \int_\Omega \mathbf{f}\cdot\bar\bv ~\d\x \d t, 
\end{array}
\end{equation}
with $\bar p_c = \bar p^\g - \bar p^\l$, $\bar d_f = -\jump{\bar\bu}$, $\bar\phi_m -\bar \phi_m^0= \dsp b~\div(\bar\bu-\bar\bu^0) + \frac{1}{M} (\bar p^E_m -\bar p_m^{E,0})$, $\bar d_f^0 = -\jump{\bar \bu^0}$, where $\bar \bu^0$ is the solution of \eqref{eq_var_meca} without the time integral and using the initial equivalent pressures 
$\bar p_m^{E,0}$ and $\bar p_f^{E,0}$ obtained from the initial pressures {$\bar p^\alpha_0$ and $\gamma\bar p^\alpha_0$, $\alpha\in \{\g,\l\}$.}
\label{def:weak.sol}
\end{definition}
\begin{remark}[Regularity of the fracture aperture]
Notice that, by the Sobolev--trace embeddings \cite[Theorem 4.12]{AF03}, $\bar \bu \in L^\infty(0,T;\U_0)$ implies that $\bar d_f = -\jump{\bar\bu}\in L^\infty(0,T;L^4(\G))$. All the integrals above are thus well-defined.
\end{remark}

\section{The gradient discretization method}\label{sec:gradientscheme}

The gradient discretization (GD) for the Darcy continuous pressure model, introduced in \cite{BGGLM16}, is defined by a finite-dimensional vector space of discrete unknowns $X^0_{\D_p}$ and 
\begin{itemize}
\item  two discrete gradient linear operators on the matrix and fracture domains
$$
 \nabla_{\D_p}^m : X^0_{\D_p} \rightarrow L^\infty(\Omega)^d, \quad \quad 
   \nabla_{\D_p}^f : X^0_{\D_p} \rightarrow L^\infty(\Gamma)^{d-1};
$$ 
\item two function reconstruction linear operators on the matrix and fracture domains 
$$
\Pi_{\D_p}^m : X^0_{\D_p} \rightarrow L^\infty(\Omega),\quad\quad \Pi_{\D_p}^f : X^0_{\D_p} \rightarrow L^\infty(\Gamma),
$$
which are \emph{piecewise constant} \cite[Definition 2.12]{gdm}.
\end{itemize}
A consequence of the piecewise-constant property is the following: there is a basis $(\mathbf{e}_i)_{i\in I}$ of $X^0_{\D_p}$ such that, if $v=\sum_{i\in I} v_i\mathbf{e}_i$ and if, for a mapping $g:\R\to\R$ with $g(0)=0$, we define $g(v)=\sum_{i\in I}g(v_i)\mathbf{e}_i\in X^0_{\D_p}$ by applying $g$ component-wise, then $\Pi_{\D_p}^{\rm rt} g(v)=g(\Pi_{\D_p}^{\rm rt} v)$ for $\mbox{\rm rt} \in \{m,f\}$. Note that the basis $(\mathbf{e}_i)_{i\in I}$ is usually canonical and chosen in the design of $X^0_{\D_p}$.
The vector space $X^0_{\D_p}$ is endowed with 
$$
\|v\|_{\D_p} \coloneq \displaystyle  \|\nabla_{\D_p}^m v\|_{L^2(\Omega)} 
+ \|d_{0}^{\nf 3 2}\nabla_{\D_p}^f v\|_{ L^2(\Gamma) },
$$
assumed to define a norm on $X^0_{\D_p}$.

The gradient discretization for the mechanics is defined by a finite-dimensional vector space of discrete unknowns $X^0_{\D_\bu}$ and
\begin{itemize}
\item  a discrete symmetric gradient linear operator  $\bbeps_{\D_\bu} : X^0_{\D_{\bu}} \rightarrow   L^2(\Omega, \S_{d}(\R))$,   
\item a displacement function reconstruction linear operator  
$\Pi_{\D_\bu} : X^0_{\D_{\bu}} \rightarrow L^2(\Omega)^d$,
\item a normal jump function reconstruction linear operator 
  $\jump{\cdot}_{\D_\bu} : X^0_{\D_{\bu}} \rightarrow L^4(\Gamma)$,   
\end{itemize}
where $\S_{d}(\R)$ is the vector space of real symmetric matrices of size $d$. 
Let us define the divergence and stress tensor operators by
$$
\div_{\D_\bu}(\bv) = \mbox{\rm Trace}(\bbeps_{\D_\bu}(\bv))\quad\mbox{ and }\quad
\bbsig_{\D_\bu}(\bv) =  2\mu \bbeps_{\D_\bu}(\bv) + \lambda \, \div_{\D_\bu}(\bv) \mathbb{I}, 
$$
and the fracture width $ d_{f,\D_\bu} = -\jump{\bu}_{\D_\bu}$. 
It is assumed that the following quantity defines a norm on $X^0_{\D_{\bu}}$:
$$
\|\bv\|_{\D_\bu} \coloneq \|\bbeps_{\D_\bu}(\bv)\|_{L^2(\Omega,\mathcal S_d(\mathbb R))}.
$$

\begin{remark}[On the boundary conditions]
The exponent $0$ in the spaces means that homogeneous Dirichlet boundary conditions are encoded in these spaces.
We restrict our analysis to these boundary conditions for simplicity but, as shown in \cite{gdm}, the GDM analysis can easily be adapted to other types of boundary conditions -- in particular to mixed Dirichlet/Neumann boundary conditions (with non-homogeneous Dirichlet values) as used for the flow part of the model in the numerical tests of Section \ref{sec:numerical.example}.
\end{remark}

A spatial GD can be extended into a space-time GD by complementing it with
\begin{itemize}
\item 
a discretization $ 0 = t_0 < t_1 < \dots < t_N = T $ of the time interval $[0,T]$;
\item
interpolators $I_{\D_p} \colon V_0 \rightarrow X^0_{\D_p}$ and $I^m_{\D_p} \colon L^2(\Omega)\rightarrow X^0_{\D_p}$ of initial conditions.
\end{itemize}
For $n\in\{0,\ldots,N\}$, we denote by $\dtn = t_{n+1}-t_n$ the time steps, and by $\Delta t = \max_{n=0,\ldots,N} \dtn$ the maximum time step. 

The spatial operators are extended into space-time operators as follows. Let $\chi$ represent either $p$ or $\bu$. If $w=(w_n)_{n=0}^{N}\in (X^0_{\D_\chi})^{N+1}$, and $\Psi_{\D_\chi}$ is a spatial GD operator, its space-time extension is defined by 
$$
\Psi_{\D_\chi}w(0,\cdot) = \Psi_{\D_\chi}w_0 \mbox{ and, } \forall n\in\{0,\dots,N-1\}\,,\;\forall t\in (t_n,t_{n+1}],\,\;
\Psi_{\D_\chi}w(t,\cdot) = \Psi_{\D_\chi}w_{n+1}.
$$
For convenience, the same notation is kept for the spatial and space-time operators.
Moreover, we define the discrete time derivative as follows: for $f:[0,T]\to L^1(\Omega)$ piecewise constant on the time discretization, with $f_n=f_{|(t_{n-1},t_n]}$ and $f_0=f(0)$, we set $\delta_t f (t) = \frac{f_{n+1} - f_n}{\dtn}$ for all  $t\in (t_n,t_{n+1}]$, $n\in\{0,\ldots,N-1\}$. 
  
Notice that the space of piecewise constant $X^0_{\D_\chi}$-valued functions $f$ on the time discretization together with the initial value $f_0=f(0)$ can be identified with $(X^0_{\D_\chi})^{N+1}$. The same definition of discrete derivative can thus be given for an element $w\in(X^0_{\D_\chi})^{N+1}$.
Namely, \mbox{$\delta_t w \in (X^0_{\D_\chi})^{N}$} is defined by setting, for any $n\in\{0,\ldots,N-1\}$ and $t\in (t_{n},t_{n+1}]$, $\delta_t w(t)=(\delta_t w)_{n+1} \coloneq \frac{w_{n+1} - w_n}{\dtn}$. If $\Psi_{\D_\chi}(t,\cdot)$ is a space-time GD operator, by linearity the following commutativity property holds: $\Psi_{\D_\chi} \delta_t w (t,\cdot) = \delta_t(\Psi_{\D_\chi} w(t,\cdot))$.

%\kb{\sout{
%The gradient scheme for \eqref{eq_edp_hydromeca}  consists in writing the weak formulation \eqref{eq_var_hydro}--\eqref{eq_var_meca} with continuous spaces and operators substituted by their discrete counterparts, after a formal integration by part: 
%}
The gradient scheme for the system consists in replacing the ``continuous'' functional space and differential operators in  \eqref{eq_var_hydro}--\eqref{eq_var_meca} by their discrete counterparts. This results in the following discrete problem:
%}
find $p^\alpha \in (X^0_{\D_p})^{N+1}$, $\alpha\in \{\g,\l\}$, and $\bu \in (X^0_{\D_\bu})^{N+1}$, such that for all $\varphi^\alpha \in (X_{\D_p}^0)^{N+1}$, $\bv \in (X^0_{\D_\bu})^{N+1}$ and $\alpha\in \{\g,\l\}$,
\begin{subequations}\label{eq:GS}
\begin{align}
&\left.\begin{array}{llll}
  && \dsp \int_0^T \int_\Omega \( \delta_t \(\phi_\D \Pi_{\D_p}^m s^\alpha_m \)\Pi_{\D_p}^m \varphi^\alpha 
  +  \eta_m^\alpha(\Pi_{\D_p}^m s_m^\alpha) \K_m \nabla_{\D_p}^m p^\alpha \cdot   \nabla_{\D_p}^m \varphi^\alpha \) \d\x \d t\\[2ex]
  && + \dsp \int_0^T \int_\Gamma \delta_t \(d_{f,\D_\bu} \Pi_{\D_p}^f s^\alpha_f \)\Pi_{\D_p}^f \varphi^\alpha \d\sigma(\x)\d t\\
  && + \dsp \int_0^T \int_\Gamma  \eta_f^\alpha(\Pi_{\D_p}^f s_f^\alpha) {d_{f,\D_\bu}^3 \over 12} \nabla_{\D_p}^f p^\alpha \cdot   \nabla_{\D_p}^f \varphi^\alpha  \d\sigma(\x) \d t\\[2ex]
  && = \dsp \int_0^T \int_\Omega h_m^\alpha \Pi_{\D_p}^m \varphi^\alpha \d\x \d t + \int_0^T \int_\Gamma h_f^\alpha \Pi_{\D_p}^f \varphi^\alpha \d\sigma(\x)\d t,\\[3ex]
  \end{array}\right.
  \label{GD_hydro}
\\
& 
  \left.\begin{array}{llll}
    &&   \dsp \int_0^T \int_\Omega \( \bbsig_{\D_\bu}(\bu) : \bbeps_{\D_\bu}(\bv)  
    - b ~\Pi_{\D_p}^m p_m^E~  \div_{\D_\bu}(\bv)\)  \d\x \d t\\
    && \quad \quad\quad + \dsp \int_0^T \int_\Gamma \Pi_{\D_p}^f p_f^E~  \jump{\bv}_{\D_\bu} \d\sigma(\x)\d t 
    = \int_0^T \int_\Omega \mathbf{f} \cdot \Pi_{\D_\bu} \bv ~\d\x \d t, 
  \end{array}\right.
  \label{GD_meca}
\end{align}
with the closure equations
\begin{equation}
  \left\{\begin{array}{ll}
& p_{c} = p^\g - p^\l,  \quad s^\alpha_m = S^\alpha_m(p_c),  \quad s^\alpha_f = S^\alpha_f(p_c), \\[2ex]
    & \dsp p_m^E = \sum_{\alpha\in\{\g,\l\}}  p^\alpha s^\alpha_m - U_m(p_{c}),\quad \dsp p_f^E = \sum_{\alpha\in\{\g,\l\}}  p^\alpha s^\alpha_f - U_f(p_{c}),\\[4ex]
  &  \phi_{\D} - \Pi_{\D_p}^m  \phi_m^0 = b ~\div_{\D_\bu} (\bu-\bu^0) + {1\over M} \Pi_{\D_p}^m (p_m^E-p_m^{E,0}),\\[2ex]
  & d_{f,\D_\bu} = -\jump{\bu}_{\D_\bu},\\[2ex]
  & \bbsig_{\D_\bu}(\bv) =  2\mu \bbeps_{\D_\bu}(\bv) + \lambda \, \div_{\D_\bu}(\bv) \mathbb{I}.  
  \end{array}\right.
  \label{GD_closures}
\end{equation}
\end{subequations}
The initial conditions are given by $p^\a_0 = I_{\D_p} \bar p^\a_0$ ($\a\in\{\g,\l\}$), $\phi_m^0 = I_{\D_p}^m \bar \phi^0_m$, and the initial displacement $\bu^0$ is the solution of
\eqref{GD_meca} without the time variable and with the equivalent pressures obtained from the initial pressures $(p^\alpha_0)_{\alpha\in\{\g,\l\}}$.

\begin{remark}[Non-homogeneous boundary conditions]
The homogeneous Dirichlet boundary conditions are embedded in the discrete spaces $X^0_{\D_p}$ and $X^0_{\D_\bu}$. Non-homogeneous (or other types of) boundary conditions are equally easy to handle in the GDM setting \cite[Section 2.2 and Chapter 3]{gdm}.
\end{remark}

\begin{remark}[GDM framework]
As shown above, the GDM framework enables a presentation of the schemes in a way that is almost as compact as the weak formulation itself (compare with Definition \ref{def:weak.sol}). This presentation is valid for conforming methods, that already have a compact writing but may not be the best suited in practical applications (especially for the flow component), but also for non-conforming methods of practical interest in engineering; explicitly writing, for example, the TPFA formulation for the flow component of the model would lead to much lengthier equations. Additionally, the GDM analysis is also carried out in a compact way, identifying key properties and manipulating discrete equations almost as their continuous counterparts; notwithstanding the fact that this analysis applies to many different methods at once, developing it for a given specific scheme would not lead to any simplification -- the complexity in the upcoming analysis comes from the poromechanical model we consider, not from the numerical analysis framework we use.
\end{remark}

\subsection{Properties of gradient discretizations}\label{sec:prop}

Let $(\D_p^l)_{l\in\N}$ and $(\D_\bu^l)_{l\in\N}$ be sequences of GDs. We state here the assumptions on these sequences which ensure that the solutions to the corresponding schemes converge. Most of these assumptions are adaptation of classical GD assumptions \cite{gdm}, except for the chain-rule, product rule and cut-off properties used in Subsection \ref{subsec:compactness} to obtain compactness properties; we note that all these assumptions hold for standard discretizations used in porous media flows.

Following  \cite{BGGLM16}, the spatial GD of the Darcy flow $\D_p=\(X_{\D_p}^0,\nabla_{\D_p}^m,\nabla_{\D_p}^f,\Pi_{\D_p}^m,\Pi_{\D_p}^f\)$ is assumed to satisfy the following coercivity, consistency, limit-conformity and compactness properties.

\noindent {\bf Coercivity of $\D_p$}. Let $C_{\D_p} > 0$ be defined by 
\begin{equation}\label{def_CDdarcy}
C_{\D_p} = \max_{0 \neq v \in X_{\D_p}^0} {\|\Pi^m_{\D_p} v\|_{L^2(\Omega)} + \|\Pi^f_{\D_p} v \|_{L^2(\Gamma)} \over \|v \|_{\D_p}}. 
\end{equation}
Then, 
a sequence of spatial GDs $(\D_p^l)_{l\in{\mathbb N}}$
 is said to be {\emph{coercive}} if there exists $\overline C_p >0$ 
such that $C_{\D^l_p} \leq \overline C_p$ for all $l\in \N$.

\noindent {\bf Consistency of $\D_p$}.
Let $r>8$ be given, and for all $w\in V_0$ and $v\in X_{\D_p}^0$ let us define 
\begin{equation}\label{def_sDdarcy}
\begin{aligned}
S_{\D_p}(w,v) ={}&  
\|\nabla^m_{\D_p} v  - \nabla w\|_{L^2(\Omega)} 
+ \|\nabla^f_{\D_p} v  -\nabla_\tau \gamma w\|_{L^r(\Gamma)} \\
&+ \|\Pi^m_{\D_p} v  -w\|_{L^2(\Omega)} + \|\Pi^f_{\D_p} v - \gamma w\|_{L^r(\Gamma)},
\end{aligned}
\end{equation}
and 
${\cal S}_{\D_p}(w) = \min_{v \in X_{\D_p}^0} S_{\D_p}(w,v)$. 
Then, a sequence of spatial GDs $(\D_p^l)_{l\in{\mathbb N}}$ is said to be {\emph{consistent}} if for all $w\in V_0$ one has
$
\lim_{l \rightarrow +\infty} {\cal S}_{\D_p^l}(w) = 0. 
$
Moreover, if $(\D_p^l)_{l\in\N}$ is a sequence of \emph{space-time} GDs, then it is said to be consistent if the underlying sequence of spatial GDs is consistent as above, and if, for any $\varphi\in V_0$ and $\psi\in L^2(\Omega)$,
as $l\to+\infty$,
\begin{equation}\label{eq:cons.stGD}
\Delta\!t^l\to 0\,,\ \|\Pi^m_{\D_p^l} I_{\D_p^l} \varphi-\varphi\|_{L^2(\Omega)}+
\|\Pi^f_{\D_p^l} I_{\D_p^l} \varphi-\varphi\|_{L^2(\Gamma)}\to 0\mbox{ and }
\|\Pi^m_{\D_p^l} I_{\D_p^l}^m \psi-\psi\|_{L^2(\Omega)}\to 0.
\end{equation}

\begin{remark}[Consistency]
In \cite{BGGLM16}, the consistency is only considered for $r=2$. As it will appear clear in the analysis, dealing with the coupling and non-linearity of the model requires us to adopt here a slightly stronger consistency assumption. Under standard mesh regularity assumptions, this stronger consistency property is still satisfied for all classical GDs \cite[Part III]{gdm}.
\end{remark}

\noindent {\bf Limit-conformity of $\D_p$}. For all $(\r_m, \r_f) \in C^\infty(\Omega\setminus\overline\Gamma)^d \times C^\infty(\Gamma)^{d-1} $ and 
$v\in X_{\D_p}^0$, let us define 
\begin{equation}\label{def_wDdarcy}
  \begin{aligned}
  W_{\D_p}(\r_m, \r_f,v) ={}& 
  \int_\Omega \( \r_m \cdot \nabla^m_{\D_p} v +\ \Pi^m_{\D_p}v ~ \div(\r_m) \)\d\x \\
 & + \int_\Gamma \( \r_f \cdot \nabla^f_{\D_p} v +\ \Pi^f_{\D_p}v ~ ( \div_\tau(\r_f) -\jump{\r_m})\)\d\sigma(\x),
  \end{aligned}
\end{equation}
and 
${\cal W}_{\D_p}(\r_m, \r_f) = \displaystyle \max_{0\neq v \in X_{\D_p}^0}
{   |W_{\D_p}(\r_m, \r_f,v)| \over\|v \|_{\D_p} }$. 
Then, a sequence of spatial GDs $(\D_p^l)_{l\in{\mathbb N}}$ 
is said to be {\emph{limit-conforming}} if for all $(\r_m, \r_f) \in C^\infty(\Omega\setminus\overline\Gamma)^d \times C_c^\infty(\Gamma)^{d-1}$ one has 
$\lim_{l \rightarrow +\infty} {\cal W}_{\D_p^l}(\r_m, \r_f) = 0$. Here $C_c^\infty(\Gamma)^{d-1}$ denotes the space of functions whose restriction to each $\Gamma_i$ {is} in $C^\infty(\Gamma_i)^{d-1}$ tangent to $\Gamma_i$, compactly supported away from the tips, and satisfying normal flux conservation at fracture intersections not located at the boundary $\partial\Omega$. 

\begin{remark}[Compactly supported fluxes]
The role of $(\r_m, \r_f)$ is that of test functions (they do not represent the continuous fluxes), to show that the limits of the discrete fluxes are indeed the continuous fluxes, see \cite[Lemma~5.5]{BGGLM16}.
\end{remark}

\noindent {\bf (Local) compactness} of $\D_p$. A sequence of spatial GDs  
$(\D_p^l)_{l\in{\mathbb N}}$ is said to 
be {\emph{locally compact}} if for all sequences $(v^l)_{l\in {\mathbb N}}{\in (X^0_{\D_p^l})_{l\in\N}}$
such that $\sup_{l\in\N}\|v^l\|_{\D_p^l}<+\infty$ and all compact sets $K_m\subset \Omega$ and $K_f\subset\Gamma$, such that $K_f$ is disjoint from the intersections $(\overline{\Gamma}_i\cap\overline{\Gamma}_j)_{i\not=j}$, the sequences $(\Pi^m_{\D_p^l}v^l)_{l\in\N}$ and  $(\Pi^f_{\D_p^l}v^l)_{l\in \N}$ are relatively compact in $L^2(K_m)$ and $L^2(K_f)$, respectively.

\begin{remark}[Local compactness through estimates of space translates]\label{rem:compact.equivalent}
For $K_m,K_f$ as above, set
\begin{equation*}
T_{\D_p^l,K_m,K_f}(\xi,\eta)=\max_{v\in X^0_{\D_p^l}\backslash\{0\}}\frac{\|\Pi^m_{\D_p^l}v(\cdot+\xi)-\Pi^m_{\D_p^l}v\|_{L^2(K_m)}+\sum_{i\in I}\|\Pi^f_{\D_p^l}v(\cdot+\eta_i)-\Pi^f_{\D_p^l}v\|_{L^2(K_f\cap\Gamma_i)}}{\|v\|_{\D_p^l}},
\end{equation*}
where $\xi\in\R^d$, $\eta=(\eta_i)_{i\in I}$ with $\eta_i$ tangent to $\Gamma_i$; for $\xi$ and $\eta$ small enough,
this expression is well defined since $K_m$ and $K_f$ are compact in $\Omega$ and $\Gamma$, respectively. Following \cite[Lemma 2.21]{gdm}, an equivalent formulation of the local compactness property is: for all $K_m,K_f$ as above,
\begin{equation*}
\lim_{\xi,\eta\to 0}\sup_{l\in\N}T_{\D_p^l,K_m,K_f}(\xi,\eta)=0.
\end{equation*}
\end{remark}

\begin{remark}[Usual compactness property for GDs]
The standard compactness property for GD is not \emph{local} but \emph{global}, that is, on the entire domain and not any of its compact subsets (see, e.g., \cite[Definition 2.8]{gdm} and also below for $\D_\bu$). Two reasons pushed us to consider here the weaker notion of local compactness: firstly, for standard GDs, the global compactness does not seem obvious to establish (or even true) in the fractures, because of the weight $d_0$ in the norm $\|\cdot\|_{\D_p}$, which prevents us from estimating the translates of the reconstructed function by the gradient near the fracture tips; secondly, we will only prove compactness on saturations, which are uniformly bounded by 1 and for which local and global compactness are therefore equivalent.

In the following, for brevity we refer to the local compactness of $(\D_p^l)_{l\in\N}$ simply as the \emph{compactness} of this sequence of GDs.
\end{remark}

\noindent\textbf{Chain rule estimate on $(\D_p^l)_{l\in\N}$}: for any Lipschitz-continuous function $F:\R\to\R$, there is $C_F\ge 0$ such that, for all $l\in\N$, $v\in X_{\D_p^l}^0$,  $$\|\nabla_{\D_p^l}^m F(v)\|_{L^2(\Omega)}\le C_F \|\nabla_{\D_p^l}^m v\|_{L^2(\Omega)}.$$

\noindent\textbf{Product rule estimate on $(\D_p^l)_{l\in\N}$}:  there exists $C_P$ such that, for any $l\in\N$ and any $u^l,v^l\in X_{\D_p^l}^0$, it holds
$$
\|\nabla^m_{\D_p} (u^l v^l) \|_{L^2(\Omega)} \leq C_P\( |u^l|_{\infty} \|\nabla^m_{\D_p} v^l \|_{L^2(\Omega)} + |v^l|_{\infty} \|\nabla^m_{\D_p} u^l \|_{L^2(\Omega)} \),
$$
where $|w|_{\infty}\coloneq\max_{i\in I}|w_i|$ whenever $w=\sum_{i\in I}w_i\mathbf{e}_i$ with $(\mathbf{e}_i)_{i\in I}$ the canonical basis of $X^0_{\D_p^l}$.

\noindent\textbf{Cut-off property of $(\D_p^l)_{l\in\N}$}: for any compact set $K\subset \Omega\backslash \Gamma$, there exists $C_K\ge 0$ and $(\psi^l)_{l\in\N}\in (X^0_{\D_p^l})_{l\in\N}$ such that $(|\psi^l|_{\infty})_{l\in\N}$ is bounded and, for $l$ large enough: 
$$
\begin{gathered}
\Pi_{\D_p^l}^m\psi^l\ge 0\text{ on }\Omega; \quad \Pi_{\D_p^l}^m\psi^l = 1\text{ on }K;
\quad  \|\nabla_{\D_p^l}^m\psi^l\|_{L^2(\Omega)}\le C_K\\ \Pi_{\D_p^l}^f (v^l\psi^l)=0 \ \ \text{and} \ \
\nabla_{\D_p^l}^f(v^l\psi^l)=0\quad\text{for all }v^l\in X_{\D_p^l}^0
\end{gathered}
$$

\noindent {\bf Coercivity of $(\D_\bu^l)_{l\in\N}$}.
Let $C_{\D_\bu} > 0$ be defined by 
\begin{equation}\label{def_CDmeca}
C_{\D_\bu} = \max_{\mathbf 0 \neq \bv \in X_{\D_\bu^l}^0} {\|\Pi_{\D_\bu^l} \bv\|_{L^2(\Omega)} + \|\jump{\bv}_{\D_\bu^l}\|_{L^4(\Gamma)} \over \|\bv \|_{\D_\bu^l}}.
\end{equation}
Then, 
the sequence of spatial GDs $(\D_\bu^l)_{l\in{\mathbb N}}$
 is said to be {\emph{coercive}} if there exists $\overline C_{\bu} >0$ 
such that $C_{\D^l_\bu} \leq \overline C_{\bu}$ for all $l\in \N$.

\noindent {\bf Consistency of $(\D_\bu^l)_{l\in\N}$}. For all $\textbf{w}\in \U_0$, it holds $\lim_{l \rightarrow +\infty} {\cal S}_{\D_\bu^l}(\textbf{w}) = 0$ where
\begin{equation}\label{eq:def.SDu}
{\cal S}_{\D_\bu^l}(\textbf{w})=  {\min_{\bv \in X_{\D_\bu^l}^0}}\Big[
\|\bbeps_{\D_\bu^l}(\bv)  - \bbeps(\textbf{w})\|_{L^2(\Omega,\S_{d}(\R))} 
+ \|\Pi_{\D_\bu^l} \bv  -\textbf{w}\|_{L^2(\Omega)} + \left\|\jump{\bv}_{\D_\bu^l}  - \jump{\textbf{w}}\right\|_{L^4(\Gamma)}\Big].
\end{equation}

\noindent {\bf Limit-conformity of $(\D_\bu^l)_{l\in\N}$}.
Let $C_\Gamma^\infty(\Omega\setminus\overline\Gamma,\S_{d}(\R))$ denote the vector space of smooth functions $\bbtau : \Omega\setminus\overline\Gamma \to \S_{d}(\R)$ whose derivatives of any order admit finite limits on each side of $\Gamma$, and such that $\bbtau^+(\x) \n^+ + \bbtau^-(\x) \n^- = {\mathbf 0}$ and $(\bbtau^+(\x) \n^+) {\times} \n^+ = {\mathbf 0}$ for 
a.e.\ $\x\in \Gamma$. For all $\bbtau \in C_\Gamma^\infty(\Omega\setminus\overline\Gamma,\S_{d}(\R))$, it holds
$\lim_{l \rightarrow +\infty} {\cal W}_{\D_\bu^l}(\bbtau) = 0$ where
\begin{equation*}
 {{\cal W}_{\D_\bu^l}(\bbtau)}{} ={\max_{\mathbf 0 \neq \bv \in X_{\D_\bu^l}^0}\frac{1}{\|\bv \|_{\D_\bu^l}}}\left[
  \int_\Omega \( \bbtau : \bbeps_{\D_\bu^l}(\bv) + \Pi_{\D_\bu^l}\bv \cdot \div(\bbtau) \)\d\x 
 - \int_\Gamma  (\bbtau \n^+)\cdot\n^+ \jump{\bv}_{\D_\bu^l} \d\sigma(\x)\right].
\end{equation*}

\noindent {\bf Compactness of $(\D_\bu^l)_{l\in\N}$}. For any sequence $(\bv^l)_{l\in {\mathbb N}}{\in (X^0_{\D_\bu^l})_{l\in\N}}$ such that {$\sup_{l\in\N}\|\bv^l\|_{\D_\bu^l}<+\infty$}, the sequences $(\Pi_{\D_\bu^l}\bv^l)_{l\in\N}$ and $(\jump{\bv^l}_{\D_\bu^l})_{l\in \N}$ 
are relatively compact in $L^2(\Omega)^d$ and in $L^s(\Gamma)$ for all $s<4$, respectively.

\begin{remark}[Compactness through estimates of space translates]\label{rem:compact.equivalent.Du}
Similarly to Remark \ref{rem:compact.equivalent} (see also \cite[Lemma 2.21]{gdm}), the compactness of $(\D_\bu^l)_{l\in\N}$ is equivalent to
\begin{equation*}
\lim_{\xi,\eta\to 0}\sup_{l\in\N}T_{\D_\bu^l,s}(\xi,\eta)=0\quad\forall s<4,
\end{equation*}
where
\begin{equation*}
T_{\D_\bu^l,s}(\xi,\eta)=\max_{\bv\in X^0_{\D_\bu^l}\backslash\{0\}}\frac{\|\Pi_{\D_\bu^l}\bv(\cdot+\xi)-\Pi_{\D_\bu^l}\bv\|_{L^2(\Omega)}+\sum_{i\in I}\left\|\jump{\bv^l}_{\D_\bu^l}(\cdot+\eta_i)-\jump{\bv^l}_{\D_\bu^l}\right\|_{L^s(\Gamma_i)}}{\|\bv\|_{\D_\bu^l}},
\end{equation*}
with $\xi\in\R^d$, $\eta=(\eta_i)_{i\in I}$ with $\eta_i$ tangent to $\Gamma_i$, and the functions extended by $0$ outside their respective domain $\Omega$ or $\Gamma$.
\end{remark}

\section{Convergence analysis}\label{sec:convergence}
The main result of this work is the following theorem stating the convergence of the sequence of discrete solutions to a weak solution up to a subsequence. 

\begin{theorem}[Convergence to a weak solution]\label{thm.conv}
Let $(\D_p^l)_{l\in\N}$, $(\D_\bu^l)_{l\in\N}$, $\{(t^l_n)_{n=0}^{N^l}\}_{l\in\N}$ (where $N^l$ is the number of time steps of $\D_p^l$), be sequences of space time GDs assumed to satisfy the properties described in Section \ref{sec:prop}.
Let $\phi_{m,{\rm min}} > 0$ and assume that, for each $l\in \N$, the gradient scheme \eqref{GD_hydro}--\eqref{GD_meca} has a solution $p^\alpha_l\in (X_{\D_p^l}^0)^{N^l+1}$, $\alpha\in \{\g,\l\}$, $\bu^l \in (X_{\D_\bu^l}^0)^{N^l+1}$ such that
  \begin{itemize}
\item[(i)] $d_{f,\D_\bu^l}(t,\x) \geq d_0(\x)$ for a.e.\ $(t,\x) \in (0,T)\times \Gamma$, 
\item[(ii)] $\phi_{\D^l}(t,\x)\geq \phi_{m,{\rm min}}$ for a.e.\ $(t,\x) \in (0,T)\times \Omega$.
  \end{itemize}
  Then, there exist $\bar p^\alpha\in L^2(0,T;V_0)$, $\alpha\in \{\g,\l\}$, and $\bar \bu\in L^\infty(0,T;\U_0)$
  satisfying the weak formulation \eqref{eq_var_hydro}--\eqref{eq_var_meca} such that for $\alpha\in \{\g,\l\}$ and up to a subsequence 
\begin{equation*}
%\left\{
\begin{array}{lll}
  & \Pi^m_{\D_p^l} p^\alpha_l \weakto \bar p^\alpha & \mbox{weakly in } L^2(0,T;L^2(\Omega)), \\[1ex]
  & \Pi^f_{\D_p^l} p^\alpha_l \weakto \gamma \bar p^\alpha & \mbox{weakly in } L^2(0,T;L^2(\Gamma)), \\[1ex]
  & \Pi_{\D_\bu^l} \bu^l \weakto \bar \bu & \mbox{weakly-$\star$ in } L^\infty(0,T;L^2(\Omega)^d),\\[1ex]
  & \phi_{\D^l} \weakto \bar \phi_m & \mbox{weakly-$\star$ in } L^\infty(0,T;L^2(\Omega)),\\[1ex]
  & d_{f,\D_\bu^l} \rightarrow \bar d_f & \mbox{in } L^\infty(0,T;L^p(\Gamma)) \mbox{ for } 2 \leq p < 4,\\[1ex]
  & \Pi^m_{D_p^l} S^\alpha_m(p_c^l) \rightarrow S^\alpha_m(\bar p_c) & \mbox{in } L^2(0,T;L^2(\Omega)) ,\\[1ex]
  & \Pi^f_{D_p^l} S^\alpha_f(p_c^l) \rightarrow S^\alpha_f(\gamma \bar p_c) & \mbox{in } L^2(0,T;L^2(\Gamma)),
\end{array}
%\right.
\end{equation*}
where $\bar \phi_m = \bar \phi_m^0 + \dsp b~\div(\bar\bu-\bar\bu^0) + \frac{1}{M} (\bar p^E_m -\bar p_m^{E,0})$,
$\bar d_f = -\jump{\bar\bu}$, and $\bar p_c = \bar p^\g - \bar p^\l$. 
\end{theorem}

\begin{remark}[Discrete porosity and fracture aperture]
  As mentioned in the introduction, the assumptions that the discrete porosity and fracture aperture remain bounded below is a requirement coming from the model itself (which does not account for possible contact). It is not a fundamental restriction of the numerical framework and analysis.
\end{remark}

We first present in Subsections \ref{subsec:apriori} and \ref{subsec:compactness} a sequence of intermediate results that will be useful for the proof of Theorem \ref{thm.conv} detailed in Subsection~\ref{subsec:convergence}.

\begin{remark}[Incompressible limit for the solid matrix]
The above convergence result also holds when ${1}/{M}=0$, i.e., in the incompressible limit for the grains of the solid matrix ($M\to+\infty$). Indeed, in this case, Lemma~\ref{lemma_apriori} below does not ensure $L^\infty(L^2)$-boundedness of the reconstructed matrix equivalent pressure. Nevertheless, $L^2(L^2)$-boundedness for this quantity (needed in the proof of the above theorem, cf.~Subsection~\ref{subsec:convergence}) can be readily inferred, based on the $L^2(L^2)$-boundedness of the reconstructed phase pressures (resulting from Lemma~\ref{lemma_apriori}), the fact that reconstructed saturations are bounded, and the definition~\eqref{eq.capillary_energy} of the capillary energy density.
\end{remark}

\subsection{Energy estimates}\label{subsec:apriori}
Using the phase pressures and velocity (time derivative of the displacement field) as test functions, the following \emph{a priori} estimates can be inferred.
\begin{lemma}[A priori estimates]   
  \label{lemma_apriori}
  Let $p^\alpha,\bu$ be a solution to problem~\eqref{eq:GS} such that 
  \begin{itemize}
\item[(i)] $d_{f,\D_\bu}(t,\x) \geq d_0(\x)$ for a.e.\ $(t,\x) \in (0,T)\times \Gamma$, 
\item[(ii)] $\phi_{\D}(t,\x) \geq  \phi_{m,\rm min}$ for a.e.\ $(t,\x) \in (0,T)\times \Omega$,
where $\phi_{m,\rm min}>0$ is a constant.
  \end{itemize}
Under hypotheses \emph{\ref{first.hyp}--\ref{last.hyp}}, there exists a real number $C > 0$ depending on the data, the coercivity constants $C_{\D_p}$, $C_{\D_\bu}$, and $\phi_{m,\rm min}$, such that the following estimates hold: 
\begin{equation}\label{apriori.est}
\begin{aligned}
\|\grad_{\D_p}^m p^\alpha\|_{L^2((0,T)\times\Omega)} \le C, 
&\quad& \| d_{f,\D_\bu}^{\nf 3 2} \grad_{\D_p}^f p^\alpha\|_{L^2((0,T)\times\Gamma)} \le C, \\
\|U_m(\Pi^m_{\D_p}p_c)\|_{L^\infty(0,T;L^1(\Omega))} \le C,
&\quad&
\|d_0 U_f(\Pi^f_{\D_p}p_c)\|_{L^\infty(0,T;L^1(\Gamma))} \le C,\\
\|\Pi^m_{\D_p} p^E_m\|_{L^\infty(0,T;L^2(\Omega))} \le C, 
&\quad&
\| \bbeps_{\D_\bu}(\bu)\|_{L^\infty(0,T;L^2(\Omega,\S_d(\R)))} \le C,\\
\|d_{f,\D_\bu}\|_{L^\infty(0,T;L^4(\Gamma))}\le C.
\end{aligned}
\end{equation}
\end{lemma}
\begin{proof}
For a piecewise constant function $v$ 
on~$[0,T]$ with $v(t)= v_{n+1}$ for all $t \in (t_n,t_{n+1}]$, $n\in\{0,\ldots,N-1\}$, and the initial value $v(0) = v_0$, we define the piecewise constant function $\hat v$ such that $\hat v(t) = v_n$ for all $t \in (t_n,t_{n+1}]$. We notice the following expression for the discrete derivative of the product of two such functions: 
\begin{equation}\label{der.prod}
\delta_t(u v)(t) = \hat u(t)\delta_t v(t) + v(t)\delta_t u(t).
\end{equation}
In \eqref{GD_hydro}, upon choosing $\varphi^\alpha = p^\alpha$ we obtain
$T_1+T_2+T_3+T_4 = T_5 + T_6$, with
\begin{equation}\label{first.eq}
 \hspace{-.1cm}
\begin{array}{llll}
 T_1 = \dsp \int_0^T \int_\Omega \delta_t \(\phi_\D \Pi_{\D_p}^m s^\alpha_m \)\Pi_{\D_p}^m p^\alpha \d\x \d t, \quad & T_2 = \dsp \int_0^T \int_\Omega \eta_m^\alpha(\Pi_{\D_p}^m s_m^\alpha) \K_m \nabla_{\D_p}^m p^\alpha \cdot   \nabla_{\D_p}^m p^\alpha \d\x \d t,\\[2ex]
   T_3 = \dsp \int_0^T \int_\Gamma \delta_t \(d_{f,\D_\bu} \Pi_{\D_p}^f s^\alpha_f \)\Pi_{\D_p}^f p^\alpha \d\sigma(\x)\d t, \quad & T_4 = \dsp \int_0^T \int_\Gamma  \eta_f^\alpha(\Pi_{\D_p}^f s_f^\alpha) {d_{f,\D_\bu}^3 \over 12} \nabla_{\D_p}^f p^\alpha \cdot   \nabla_{\D_p}^f p^\alpha  \d\sigma(\x) \d t,\\[2ex]
   T_5 = \dsp \int_0^T \int_\Omega h_m^\alpha \Pi_{\D_p}^m p^\alpha \d\x \d t,
  \quad & T_6 = \dsp \int_0^T \int_\Gamma h_f^\alpha \Pi_{\D_p}^f p^\alpha \d\sigma(\x)\d t.%\\[3ex]
 \end{array}
\end{equation}

First, we focus on the matrix and fracture accumulation terms $T_1$ and $T_3$, respectively. Using~\eqref{der.prod} and the piecewise constant function reconstruction property of $\Pi^\rt_{\D_p}$, $\rt\in\{m,f\}$, we can write 
$$
\begin{aligned}
\delta_t (\phi_\D S^\alpha_m(\Pi^m_{\D_p} p_c)) & = 
\hat\phi_\D \delta_t S^\alpha_m(\Pi^m_{\D_p} p_c) + S^\alpha_m(\Pi^m_{\D_p} p_c)
\delta_t \phi_\D,\\
\delta_t (d_{f,\D_\bu} S^\alpha_f(\Pi^f_{\D_p} p_c)) & = 
\hat d_{f,\D_\bu} \delta_t S^\alpha_f(\Pi^f_{\D_p} p_c) + S^\alpha_f(\Pi^f_{\D_p} p_c)
\delta_t d_{f,\D_\bu}.
\end{aligned}
$$
Summing on $\alpha\in\{\l,\g\}$, we obtain
$$\begin{aligned}
\sum_\alpha (T_1+T_3) 
= & \sum_\alpha \( \int_0^T \int_\Omega \hat\phi_\D \Pi^m_{\D_p} p^\alpha\,
 \delta_t S^\alpha_m(\Pi^m_{\D_p} p_c) \d\x \d t +
 \int_0^T \int_\Omega S^\alpha_m(\Pi^m_{\D_p} p_c) \Pi^m_{\D_p} p^\alpha \,\delta_t \phi_\D \d\x \d t \\
 &\hspace{-.3cm}+ \int_0^T \int_\Gamma \hat d_{f,\D_\bu} \Pi^f_{\D_p} p^\alpha\,
 \delta_t S^\alpha_f(\Pi^f_{\D_p} p_c) \d\sigma(\x) \d t+
 \int_0^T \int_\Gamma S^\alpha_f(\Pi^f_{\D_p} p_c) \Pi^f_{\D_p} p^\alpha\,
 \delta_t d_{f,\D_\bu} \d\sigma(\x) \d t \).
 \end{aligned}
$$
Now, for $\rt\in\{m,f\}$,
\begin{equation}\label{eq.der_cap_en}
\sum_\alpha \Pi^\rt_{\D_p} p^\alpha\,
 \delta_t S^\alpha_\rt(\Pi^\rt_{\D_p} p_c) = \Pi^\rt_{\D_p} p_c\,
 \delta_t S^\g_\rt(\Pi^\rt_{\D_p} p_c) \ge \delta_t U_\rt(\Pi^\rt_{\D_p} p_c).
 \end{equation}
Indeed, for $n\in\{0,\dots,N-1\}$, by the definition~\eqref{eq.capillary_energy}
of the capillary energy $U_\rt$ and letting \mbox{$\pi^{\rt}_{c,n}=\Pi^\rt_{\D_p}p_{c,n}$}, we have
$$
\begin{aligned}
\pi^\rt_{c,n+1}(S_{\rt}^\g(\pi^\rt_{c,n+1})-S_{\rt}^\g(\pi^\rt_{c,n})) & = U_\rt(\pi^\rt_{c,n+1}) - U_\rt(\pi^\rt_{c,n}) + \int_{\pi^\rt_{c,n}}^{\pi^{\rt}_{c,n+1}} (S_\rt^\g(q) - S_\rt^\g(\pi^\rt_{c,n}))\d q \\
& \ge U_\rt(\pi^\rt_{c,n+1}) - U_\rt(\pi^\rt_{c,n}),
\end{aligned}$$
where the last inequality holds since $S_\rt^\g$ is a non-decreasing function (see \ref{H2}). Thus, we obtain
$$\begin{aligned}
\sum_\alpha (T_1+T_3) 
\ge & \int_0^T \int_\Omega \hat\phi_\D \delta_t U_m(\Pi^m_{\D_p} p_c) \d\x \d t
+\int_0^T \int_\Gamma \hat d_{f,\D_\bu} \delta_t U_f(\Pi^f_{\D_p} p_c) \d\sigma(\x)\d t \\
&\!\! +\sum_\alpha\( \int_0^T \int_\Omega S^\alpha_m(\Pi^m_{\D_p} p_c) \Pi^m_{\D_p} p^\alpha \, \delta_t \phi_\D \d\x \d t + \int_0^T \int_\Gamma S^\alpha_f(\Pi^f_{\D_p} p_c) \Pi^f_{\D_p} p^\alpha \,
 \delta_t d_{f,\D_\bu} \d\sigma(\x)\d t \).
 \end{aligned}
$$
Applying again~\eqref{der.prod}, we have 
$$
\begin{aligned}
\hat\phi_\D \delta_t U_m(\Pi^m_{\D_p} p_c) & = 
\delta_t(\phi_\D U_m(\Pi^m_{\D_p} p_c)) - U_m(\Pi^m_{\D_p} p_c) \delta_t\phi_\D,\\
\hat d_{f,\D_\bu} \delta_t U_f(\Pi^f_{\D_p} p_c) & = \delta_t(d_{f,\D_\bu} U_f(\Pi^f_{\D_p} p_c))
- U_f(\Pi^f_{\D_p} p_c) \delta_t d_{f,\D_\bu}.
\end{aligned}
$$
In the light of the closure equations~\eqref{GD_closures}, this allows us to infer that
\begin{equation}\label{lhs_accumulation}
\begin{aligned}
\sum_\alpha (T_1+T_3) \ge 
& \int_0^T\int_\Omega \delta_t(\phi_\D U_m(\Pi^m_{\D_p} p_c)) 
\d\x \d t + \int_0^T \int_\Gamma \delta_t (d_{f,\D_\bu} U_f(\Pi^f_{\D_p} p_c)) \d\sigma(\mathbf x) \d t \\
& + \int_0^T \int_\Omega \frac{1}{2M} \delta_t\(\Pi^m_{\D_p} p^E_m\)^2 \d\x \d t + 
\int_0^T \int_\Omega b\, \Pi^m_{\D_p} p^E_m \, \div_{\D_\bu} (\delta_t \bu) \d\x \d t \\
& - \int_0^T \int_\Gamma \Pi^f_{\D_p} p^E_f\, \jump{\delta_t \bu}_{\D_\bu} \d\sigma(\x) \d t,
\end{aligned}
\end{equation}
where we have used the fact that 
\begin{equation}\label{eq:vdtv}
v\delta_t v \ge \delta_t\left(\frac{v^2}{2}\right)
\end{equation}
for $v$ piecewise constant
on~$[0,T]$. Then, taking into account assumptions~\ref{first.hyp}--\ref{last.hyp}, there exists a real number $C > 0$ depending only on the data such that
\begin{equation}\label{lhs_diffusion}
\begin{aligned}
\sum_\alpha (T_2+T_4) \ge C \( \int_0^T \int_\Omega \sum_\alpha |\grad_{\D_p}^m p^\alpha|^2 \d\x \d t 
+  \int_0^T \int_\Gamma \sum_\alpha |d_{f,\D_\bu}^{\nf 3 2} \grad^f_{\D_p} p^\alpha|^2 \d\sigma(\x) \d t \).
\end{aligned}
\end{equation}

On the other hand, upon choosing $\bv = \delta_t\bu$ in \eqref{GD_meca}, we get
$T_7+T_8+T_9=T_{10}$, with
\begin{equation}\label{second.eq}
 \hspace{-.1cm}
 \begin{array}{llll}
 T_7 = \dsp \int_0^T \int_\Omega \bbsig_{\D_\bu}(\bu) : \bbeps_{\D_\bu}(\delta_t \bu) \d\x \d t,  \quad 
 & \ \, T_8 = - \dsp \int_0^T \int_\Omega b\, \Pi_{\D_p}^m p_m^E\,  \div_{\D_\bu}(\delta_t \bu) \d\x \d t\,\\[2ex]
   T_9 = \dsp \int_0^T \int_\Gamma \Pi_{\D_p}^f p_f^E \, \jump{\delta_t \bu}_{\D_\bu} \d\sigma(\x) \d t, \quad 
   & T_{10} = \dsp \int_0^T \int_\Omega \mathbf{f} \cdot \Pi_{\D_\bu} (\delta_t \bu) \d\x \d t.
 \end{array}
\end{equation}
Using \eqref{eq:vdtv} and developing the definition of $\bbsig_{\D_\bu}(\bu)$, we see that
\begin{equation}\label{est_elastic_energy}
T_7 \ge  \int_0^T \int_\Omega  \delta_t \( \frac{1}{2} \bbsig_{\D_\bu}(\bu) : \bbeps_{\D_\bu}(\bu) \) \d\x \d t,
\end{equation}
so that, all in all, taking into account that
$\sum_\alpha (T_1+T_2+T_3+T_4) + T_7 + T_8 + T_9 = \sum_\alpha(T_5 + T_6) + T_{10}$
and inequalities \eqref{lhs_accumulation}--\eqref{lhs_diffusion}--\eqref{est_elastic_energy},
we obtain the following estimate
for the solutions of \eqref{eq:GS}: there is a real number $C > 0$ depending on the data such that
\begin{equation}\label{GD_energy_estimate}
\begin{aligned}
&\int_0^T \int_\Omega \delta_t(\phi_\D U_m(\Pi^m_{\D_p} p_c)) \, \d\mathbf x \d t
+ \int_0^T \int_\Gamma \delta_t (d_{f,\D_\bu} U_f(\Pi^f_{\D_p} p_c)) \, \d\sigma(\mathbf x) \d t \\
& + \int_0^T \int_\Omega \delta_t\left(\frac{1}{2}\bbsigma_{\D_\bu}(\bu):\bbeps_{\D_\bu}(\bu) 
+ \frac{1}{2M}(\Pi^m_{\D_p}p^E_m)^2\right)\, \d\x \d t  \\
&+\sum_{\alpha} \int_0^T\int_\Omega |\grad^m_{\D_p}p^\alpha|^2 \, \d\mathbf x \d t + 
\sum_{\alpha} \int_0^T\int_\Gamma | d_{f,\D_\bu}^{\nf 3 2} \grad^f_{\D_p}p^\alpha |^2 \, \d\sigma(\mathbf x) \d t\\
& \le C\left( \int_0^T \int_\Omega \mathbf f \cdot \delta_t \Pi_{\D_\bu} \bu \, \d\mathbf x \d t +
\sum_{\alpha} \int_0^T \int_\Omega h_m^\alpha \Pi^m_{\D_p}p^\alpha\,\d\mathbf x \d t \right. \\ 
& \qquad + \left. \sum_{\alpha} \int_0^T \int_\Gamma h_f^\alpha \Pi^f_{\D_p}p^\alpha\,\d\sigma(\mathbf x) \d t \right).
\end{aligned}
\end{equation}
Now, we have
$$
\begin{alignedat}{1}
 \int_0^T \int_\Omega \mathbf f \cdot \delta_t \Pi_{\D_\bu} \bu \, \d\x \d t
   & = \int_\Omega \mathbf f \cdot (\Pi_{\D_\bu}\bu(T) - \mathbf f\cdot \Pi_{\D_\bu} \bu(0)) \d\x \\
    & \le  C_{\D_\bu} \|\mathbf f\|_{L^2(\Omega)}( \|\bbeps_{\D_\bu}(\bu)(T)\|_{L^2(\Omega,\S_d(\R))} + \|\bbeps_{\D_\bu}(\bu)(0)\|_{L^2(\Omega,\S_d(\R))}),
\end{alignedat}
$$
$$
\begin{alignedat}{1}
& \sum_{\alpha} \( \int_0^T \int_\Omega h_m^\alpha \Pi^m_{\D_p}p^\alpha\,d\mathbf x \d t  +  \int_0^T \int_\Gamma h_f^\alpha \Pi^f_{\D_p}p^\alpha\,d\sigma(\mathbf x) \d t \)
\\
& \le C_{\D_p} \sum_\alpha (\|h_m^\alpha\|_{L^2((0,T)\times\Omega)} + 
\|h_f^\alpha\|_{L^2((0,T)\times\Gamma)})(\|\grad^m_{\D_p} p^\alpha\|_{L^2(0,T;L^2(\Omega))}
+\| d_{f,\D_\bu}^{\nf 3 2} \grad^f_{\D_p} p^\alpha\|_{L^2(0,T;L^2(\Gamma))}),
\end{alignedat}
$$
where we have used the coercivity properties of the two gradient discretizations along
with the Cauchy--Schwarz inequality and $d_0\le d_{f,\D_\bu}$. Using Young's inequality in the last two estimates as well as hypotheses \ref{first.hyp}--\ref{last.hyp} and (ii) in the lemma, it is then possible to infer from~\eqref{GD_energy_estimate} the existence of a real number $C > 0$ depending on the data and on $\phi_{m,\rm{min}}$ such that
$$
\begin{multlined}
\| U_m(\Pi^m_{\D_p} p_c)(T)\|_{L^1(\Omega)} + \| d_0 U_f(\Pi^f_{\D_p} p_c)(T)\|_{L^1(\Omega)}
+ \|(\Pi^m_{\D_p} p^E_m)(T)\|_{L^2(\Omega)}^2 \\
+ \|\bbeps_{\D_\bu}(\bu)(T)\|^2_{L^2(\Omega,\S_d(\R))}
+ \sum_\alpha \( \|\grad^m_{\D_p}p^\alpha\|_{L^2(0,T;L^2(\Omega))}^2 +
\| d_{f,\D_\bu}^{\nf 3 2} \grad^f_{\D_p}p^\alpha\|_{L^2(0,T;L^2(\Gamma))}^2 \) \\
\le C \( \| \mathbf f \|_{L^2(\Omega)}^2 + 
\sum_\alpha \( \|h_m^\alpha\|_{L^2((0,T)\times\Omega)}^2 + 
\|h_f^\alpha\|_{L^2((0,T)\times\Gamma)}^2 \)  \\
+  \| U_m(\Pi^m_{\D_p} p_{c})(0)\|_{L^1(\Omega)} + \| d_{f,\D_\bu}(0) U_f(\Pi^f_{\D_p} p_{c})(0)\|_{L^1(\Gamma)}\\
+ \|(\Pi^m_{\D_p} p^E_m)(0)\|_{L^2(\Omega)}^2 + \|(\Pi^f_{\D_p} p^E_f)(0)\|_{L^2(\Gamma)}^2 \).
\end{multlined}
$$

The consistency property \eqref{eq:cons.stGD} shows that the terms above involving the discrete initial conditions are bounded and thus, together with the fact that $T$ can be replaced by any $t\in (0,T]$ in the left-hand side, this inequality yields the a priori estimates~\eqref{apriori.est} on $p^\alpha$, $p_c$, $p_m^E$ and $\bu$. The estimate on $d_{f,\D_\bu}$ follows from its definition and from the definition \eqref{def_CDmeca} of $C_{\D_\bu}$.
\end{proof}

\subsection{Compactness properties}\label{subsec:compactness}

Throughout the analysis, we write $a\lesssim b$ for $a\le Cb$ with constant $C$ depending only on the coercivity constants $C_{\D_p}$, $C_{\D_\bu}$ of the considered GDs, and on the physical parameters.

\subsubsection{Estimates on time translates}

\begin{proposition}\label{prop_timetranslates}
  Let $\D_p$, $\D_\bu$, $(t_n)_{n=0}^N$ be given space time GDs and $\phi_{m,{\rm min}} > 0$. It is assumed that the gradient scheme \eqref{GD_hydro}--\eqref{GD_meca} has a solution $p^\alpha\in (X_{\D_p}^0)^{N+1}$, $\alpha\in \{\g,\l\}$, $\bu \in (X_{\D_\bu}^0)^{N+1}$ such that
  $\phi_{\D}(t,\x) \geq \phi_{m,{\rm min}}$ for a.e.\ $(t,\x) \in (0,T)\times \Omega$ and $d_{f,\D_\bu}(t,\x) \geq d_0(\x)$ for a.e.\ $(t,\x) \in (0,T)\times \Gamma$. 
  Let $\tau,\tau'  \in (0,T)$ and,  for $s\in (0,T]$, denote by $n_s$ the natural number such that $s\in (t_{n_s},t_{n_s+1}]$. For any $\varphi \in X^0_{\D_p}$, it holds 
  \begin{equation}
    \label{est_timetranslates}
\begin{array}{ll}
& \Big| 
\dsp \< [\phi_\D \Pi^m_{\D_p} s^\alpha_m](\tau) - [\phi_\D \Pi^m_{\D_p} s^\alpha_m](\tau'), \Pi^m_{\D_p} \varphi \>_{L^2(\Omega)} \\\\
& \qquad +\,\,   
\dsp \< [d_{f,\D_{\bu}} \Pi^f_{\D_p} s^\alpha_f](\tau) - [d_{f,\D_{\bu}} \Pi^f_{\D_p}s^\alpha_f](\tau'), \Pi^f_{\D_p} \varphi \>_{L^2(\Gamma)} \Big| \\\\
& \lsim  \dsp 
\sum_{n = n_{\tau}+1}^{n_{\tau'}}
 \dtn  \left( 
 \xi^{(1),\alpha,n+1}_m  \| \nabla^m_{\D_p} \varphi \|_{L^2(\Omega)} +
 \xi^{(1),\alpha,n+1}_f \| \nabla^f_{\D_p} \varphi \|_{L^8(\Gamma)} \right.\\
 & \qquad\qquad\qquad\qquad + \left. \xi^{(2),\alpha,n+1}_m  \| \Pi^m_{\D_p} \varphi \|_{L^2(\Omega)} +
 \xi^{(2),\alpha,n+1}_f \| \Pi^f_{\D_p} \varphi \|_{L^2(\Gamma)} 
 \right),
\end{array}
\end{equation}
with 
$$
\sum^{N-1}_{n = 0} \dtn \left( \xi^{(j),\alpha,n+1}_{\rm rt} \right)^2 \lsim 1 \qquad \mbox{for } \rt \in \{m,f\}, \, j\in \{1,2\},  
$$
and  
\begin{align*}
& \xi^{(1),\alpha,n+1}_m  = \| \nabla^m_{\D_p} p^{\alpha}_{n+1} \|_{L^2(\Omega)} \quad
  \mbox{ and }\quad   \xi^{(1),\alpha,n+1}_f = \| (d^{n+1}_{f,\D_{\bu}})^{\nf 3 2} \nabla^f_{\D_p} p^{\alpha}_{n+1} \|_{L^2(\Gamma)}  \| d^{n+1}_{f,\D_{\bu}} \|^{\nf 3 2}_{L^4(\Gamma)},\\
  & \xi^{(2),\alpha,n+1}_m  = \Big\| {1\over \dtn}\int_{t_n}^{t_{n+1}} h^{\alpha}_m(t,\cdot)\d t \Big\|_{L^2(\Omega)} \quad
  \xi^{(2),\alpha,n+1}_f  = \Big\| {1\over \dtn}\int_{t_n}^{t_{n+1}} h^{\alpha}_f(t,\cdot)\d t \Big\|_{L^2(\Gamma)}. 
\end{align*}
\end{proposition}
\begin{proof} 
For any $\varphi \in X^0_{\D_p}$, writing the difference of piecewise-constant functions at times $\tau$ and $\tau'$ as the sum of their jumps between these two times, one has 
\begin{equation}\label{est_tt1}
  \begin{aligned}
\Big| 
 \< [\phi_\D {}& \Pi^m_{\D_p} s^\alpha_m](\tau) - [\phi_\D \Pi^m_{\D_p} s^\alpha_m](\tau'), \Pi^m_{\D_p} \varphi \>_{L^2(\Omega)} \\
& +
\< [d_{f,\D_{\bu}} \Pi^f_{\D_p} s^\alpha_f](\tau) - [d_{f,\D_{\bu}} \Pi^f_{\D_p}s^\alpha_f](\tau'), \Pi^f_{\D_p} \varphi \>_{L^2(\Gamma)} \Big| \\
 \leq{}& 
\sum_{n = n_{\tau}+1}^{n_{\tau'}}
  \dtn \Big| \< \delta_t [\phi_\D \Pi^m_{\D_p} s^\alpha_m](t_{n+1}), \Pi^m_{\D_p} \varphi \>_{L^2(\Omega)} + 
 \<  \delta_t  [d_{f,\D_{\bu}} \Pi^f_{\D_p} s^\alpha_f](t_{n+1}), \Pi^f_{\D_p} \varphi \>_{L^2(\Gamma)}
\Big|.
\end{aligned}
\end{equation}
From the gradient scheme discrete variational equation \eqref{GD_hydro}, we deduce that 
\begin{equation}\label{est_tt2}
\begin{aligned}
\Big| 
 \< \delta_t  [\phi_\D {}&\Pi^m_{\D_p} s^\alpha_m](t_{n+1}), \Pi^m_{\D_p} \varphi \>_{L^2(\Omega)} + 
\dsp \< \delta_t  [d_{f,\D_{\bu}} \Pi^f_{\D_p} s^\alpha_f](t_{n+1}), \Pi^f_{\D_p} \varphi \>_{L^2(\Gamma)} \Big| \\
 \lsim{}&
 \| \nabla^m_{\D_p} p^{\a}_{n+1} \|_{L^2(\Omega)} ~ \| \nabla^m_{\D_p} \varphi \|_{L^2(\Omega)}
 +
 \| (d^{n+1}_{f,\D_{\bu}})^{\nf 3 2} \nabla^f_{\D_p} p^{\a}_{n+1} \|_{L^2(\Gamma)} ~ \| (d^{n+1}_{f,\D_{\bu}})^{\nf 3 2} \nabla^f_{\D_p} \varphi \|_{L^2(\Gamma)}\\
 & + \Big\|{1\over \dtn} \int_{t_n}^{t_{n+1}} h^{\alpha}_m(t,\cdot)\d t \Big\|_{L^2(\Omega)} ~\| \Pi^m_{\D_p} \varphi \|_{L^2(\Omega)}\\
& + 
 \Big\| {1\over \dtn} \int_{t_n}^{t_{n+1}} h^{\alpha}_f(t,\cdot)\d t \Big\|_{L^2(\Gamma)} ~  \|\Pi^f_{\D_p} \varphi \|_{L^2(\Gamma)}\\
 \lsim{}&
 \xi^{(1),\alpha,n+1}_m  \| \nabla^m_{\D_p} \varphi \|_{L^2(\Omega)} + \xi^{(1),\alpha,n+1}_f \| \nabla^f_{\D_p} \varphi \|_{L^8(\Gamma)}\\
& + \, \xi^{(2),\alpha,n+1}_m  \| \Pi^m_{\D_p} \varphi \|_{L^2(\Omega)} + \xi^{(2),\alpha,n+1}_f \| \Pi^f_{\D_p} \varphi \|_{L^2(\Gamma)},
 \end{aligned}
\end{equation}
 where the term $\| (d^{n+1}_{f,\D_{\bu}})^{\nf 3 2} \nabla^f_{\D_p} \varphi \|_{L^2(\Gamma)}$ has been estimated using the generalized H\"older inequality with exponents $(8,8/3)$, which satisfy $\frac{1}{8}+\frac{3}{8}=\frac{1}{2}$.
Hence the result follows from \eqref{est_tt1}, \eqref{est_tt2}, the a priori estimates of Lemma \ref{lemma_apriori}, and from the assumptions $h_m^\alpha \in L^2((0,T)\times\Omega)$,
$h_f^\alpha \in L^2((0,T)\times\Gamma)$. 
\end{proof}

\begin{remark}
  Summing the estimate \eqref{est_timetranslates} on $\alpha\in \{\g,\l\}$ we obtain the following time translate estimates on $\phi_\D$ and $d_{f,\D_\bu}$:
  \begin{equation}
    \label{est_timetranslates_phi_df}
\begin{array}{ll}
& \Big| 
\dsp \< \phi_\D(\tau) - \phi_\D(\tau'), \Pi^m_{\D_p} \varphi \>_{L^2(\Omega)} + 
\< d_{f,\D_{\bu}}(\tau) - d_{f,\D_{\bu}}(\tau'), \Pi^f_{\D_p} \varphi \>_{L^2(\Gamma)} \Big| \\\\
& \lsim  \dsp 
\sum_{\alpha\in \{\g,\l\}}\sum_{n = n_{\tau}+1}^{n_{\tau'}}
 \dtn  \left( 
 \xi^{(1),\alpha,n+1}_m  \| \nabla^m_{\D_p} \varphi \|_{L^2(\Omega)} +
 \xi^{(1),\alpha,n+1}_f \| \nabla^f_{\D_p} \varphi \|_{L^8(\Gamma)} \right.\\
 & \qquad\qquad\qquad\qquad\qquad\qquad\quad  + \left. \xi^{(2),\alpha,n+1}_m  \| \Pi^m_{\D_p} \varphi \|_{L^2(\Omega)} +
 \xi^{(2),\alpha,n+1}_f \| \Pi^f_{\D_p} \varphi \|_{L^2(\Gamma)} 
 \right). 
\end{array}
\end{equation}
\end{remark}

\subsubsection{Compactness properties of $\Pi^m_{\D_p} s^\alpha_m$}
\label{subsec:compactnessmat}
\begin{proposition}
  \label{prop_compactness_Sm}
  Let $(\D_p^l)_{l\in\N}$, $(\D_\bu^l)_{l\in\N}$, $\{(t^l_n)_{n=0}^{N^l}\}_{l\in\N}$ be sequences of space time GDs assumed to satisfy the properties described in Section \ref{sec:prop}.
  Let $\phi_{m,{\rm min}} > 0$ and assume that, for each $l\in \N$, the gradient scheme \eqref{GD_hydro}--\eqref{GD_meca} has a solution $p^\alpha_l\in (X_{\D_p^l}^0)^{N^l+1}$, $\alpha\in \{\g,\l\}$, $\bu^l \in (X_{\D_\bu^l}^0)^{N^l+1}$ such that
\begin{itemize}
\item[(i)] $d_{f,\D_\bu^l}(t,\x) \geq d_0(\x)$ for a.e.\ $(t,\x) \in (0,T)\times \Gamma$, 
\item[(ii)] $\phi_{\D^l}(t,\x)\geq \phi_{m,{\rm min}}$ for a.e.\ $(t,\x) \in (0,T)\times \Omega$.
  \end{itemize}
  Then, the sequence $(\Pi^m_{\D_p}s_m^{\alpha,l})_{l\in \N}$,  with $s_m^{\alpha,l}= S_m^\alpha(p_c^l)$, is relatively compact in $L^2((0,T)\times \Omega)$. 
\end{proposition}

\begin{proof}
Let $K$ be a fixed compact set of $\Omega\setminus\Gamma$ and let us consider cut-off functions $\psi^l$ as defined in the cut-off property of the sequence of spatial GDs $(\D_p^l)_{l\in\N}$.
The superscript $l\in \N$ will be dropped in the proof, and assumed to be large enough. All hidden constants in the following estimates are independent of $l$. 
Using that $\phi_{\D}(t,\x) \geq \phi_{m,{\rm min}}$ for a.e.\ $(t,\x) \in (0,T)\times \Omega$, the properties of the cut-off functions, and noting that $\Pi^m_{\D_p} s_m^{\alpha,l}=S_m^\alpha(\Pi^m_{\D_p} p_c^l)\in [0,1]$, we obtain
\begin{align*}
& \dsp \int_0^T \| \Pi^m_{\D_p} s^\alpha_m(\cdot + \tau,\cdot ) - \Pi^m_{\D_p} s^\alpha_m\|^2_{L^2(K)} \d t\\
&\lsim \dsp \tau + \int^{T-\tau}_0 \int_\O (\Pi^m_{\D_p} \psi) ~\phi_\D \left( \Pi^m_{\D_p} s^\alpha_m(\cdot + \tau,\cdot ) - \Pi^m_{\D_p} s^\alpha_m\right)^2 \d\x \d t =  \tau + T_1 + T_2, 
\end{align*}
where
\begin{align*}
& \dsp T_1 = \int^{T-\tau}_0 \Big| \<  [\phi_\D \Pi^m_{\D_p} s^\alpha_m]( t + \tau ) - [\phi_\D \Pi^m_{\D_p} s^\alpha_m](t),  \Pi^m_{\D_p} \zeta_m^\alpha(t) \>_{L^2(\Omega)} \Big| \d t, \\
& \dsp T_2 = \int^{T-\tau}_0 \Big| \<  \phi_\D( t + \tau ) - \phi_\D(t),  \Pi^m_{\D_p} \chi_m^\alpha(t) \>_{L^2(\Omega)} \Big| \d t,
\end{align*}
with 
$\zeta_m^\a(t) =  \( s^\alpha_m(t + \tau ) - s^\alpha_m(t) \) \psi$ and $\chi_m^\a(t) =  \zeta_m^\a(t)  ~s^\alpha_m(t+\tau)$. From the cut-off property it results that
$\Pi^f_{\D_p} \zeta_m^\alpha = 0$ and $\nabla^f_{\D_p} \zeta_m^\alpha = 0$.  
Then, in view of the estimates \eqref{est_timetranslates},  we have 
$$
\begin{aligned}
  T_1 &\lsim \dsp \int^{T-\tau}_0 \sum_{n = n_t+1}^{n_{(t+\tau)}} \dtn
  \( \xi^{(1),\alpha,n+1}_m \| \nabla^m_{\D_p} \zeta^\alpha_m (t)\|_{L^2(\Omega)}
 +  \xi^{(2),\alpha,n+1}_m \| \Pi^m_{\D_p} \zeta^\alpha_m (t)\|_{L^2(\Omega)}\) 
  ~\d t \\\\
&\lsim
 \dsp \int^{T-\tau}_0 \sum_{n = n_t+1}^{n_{(t+\tau)}} \dtn  \( (\xi^{(1),\alpha,n+1}_m)^2 + (\xi^{(2),\alpha,n+1}_m)^2 % \\
+\, \| \nabla^m_{\D_p} \zeta^\alpha_m (t)\|^2_{L^2(\Omega)} + \| \Pi^m_{\D_p} \zeta^\alpha_m (t)\|_{L^2(\Omega)}^2
 \)~\d t.  
\end{aligned}
$$
From Proposition \ref{prop_timetranslates}, we have 
$$
\sum_{n = 0}^{N-1} \dtn  \((\xi^{(1),\alpha,n+1}_m)^2 + (\xi^{(2),\alpha,n+1}_m)^2 \) \lsim 1. 
$$
Using the a priori estimates of Lemma \ref{lemma_apriori}, $h_m^\alpha \in L^2((0,T)\times\Omega)$, the Lipschitz property of $S^\alpha_m$, the chain rule and product rule estimates on the sequence of GDs $(\D_p^l)_{l\in\N}$, and the cut-off property, we obtain that
$$
\int^{T-\tau}_0 \(\| \nabla^m_{\D_p} \zeta^\alpha_m (t)\|^2_{L^2(\Omega)} + \| \Pi^m_{\D_p} \zeta^\alpha_m (t)\|_{L^2(\Omega)}^2\)\d t \lsim 1. 
$$
We deduce from \cite[Lemma 4.1]{brenner.hilhorst2013} that $T_1 \lsim \tau + \Delta t$ with a hidden constant depending on $K$ but independent of $l$.
Similarly, using the time translate estimate \eqref{est_timetranslates_phi_df}, one shows that $T_2 \lsim \tau+ \Delta t$, which provides the time translates estimates on $\Pi^m_{\D_p} s^\alpha_m$ in $L^2(0,T;L^2(K))$. 

The space translates  estimates for $\Pi^m_{\D_p} s_m^\alpha$ in $L^2(0,T;L^2(K))$  derive from the a priori estimates of Lemma \ref{lemma_apriori}, the Lipschitz properties of 
$S^{\alpha}_m$ and from the compactness property of the sequence of spatial GDs $(\D_p^l)_{l\in\N}$ (cf.~Remark \ref{rem:compact.equivalent}).
Combined with the time translate estimates,  the Fr\'echet--Kolmogorov theorem implies that $\Pi^m_{\D_p} s_m^\alpha$  is relatively compact in $L^2(0,T;L^2(K))$ for any compact set $K$ of $\Omega\setminus\Gamma$. 
Since $\Pi^m_{\D_p} s^\alpha_m\in [0,1]$, it results that $\Pi^m_{\D_p} s_m^\alpha$  is relatively compact in $L^2((0,T)\times\Omega)$.
\end{proof}

\subsubsection{Uniform-in-time $L^2$-weak convergence of $\phi_\D \Pi^m_{\D_p} s^\alpha_m$ and $\phi_\D$}

\begin{proposition}
  \label{prop_uniftimeweakL2_phimSm}
  Under the assumptions of Proposition \ref{prop_compactness_Sm}, the sequences $(\phi_{\D^l})_{l\in\N}$ and $(\phi_{\D^l} \Pi^m_{\D_p}s_m^{\alpha,l})_{l\in \N}$,  with $s_m^{\alpha,l}= S_m^\alpha(p_c^l)$, converge up to a subsequence uniformly in time  weakly in $L^2(\Omega)$. 
\end{proposition}

\begin{proof}
Let $K$ be a fixed compact set of $\Omega\setminus\Gamma$ and let $\psi^l$ be cut-off functions for this compact set, as defined in the cut-off property of $(\D_p^l)_{l\in\N}$. The superscript $l\in \N$ will be dropped when not required for the clarity of the proof, and assumed to be large enough. 

For $w\in V_0$ we let $P_{\D_p} w\in X^0_{\D_p}$ be the element that realizes the minimum in $\mathcal S_{\D_p}(w)$, so that
\begin{multline}\label{eq:def.PDp}
\|\nabla^m_{\D_p} P_{\D_p} w  - \nabla w\|_{L^2(\Omega)} 
+ \|\nabla^f_{\D_p} P_{\D_p}w  -\nabla_\tau \gamma w\|_{L^r(\Gamma)} \\
+ \|\Pi^m_{\D_p} P_{\D_p}w  -w\|_{L^2(\Omega)} + \|\Pi^f_{\D_p} P_{\D_p}w - \gamma w\|_{L^r(\Gamma)}
=\mathcal S_{\D_p}(w).
\end{multline}

Let $\overline{\varphi} \in C^\infty_c(\Omega)$ and set $\varphi = P_{\D_p}\overline{\varphi}$. It results from the cut-off property that
$\Pi^f_{\D_p} (\psi \varphi) = 0$ and $\nabla^f_{\D_p} (\psi\varphi) = 0$. Using the GD consistency property of $(\D_p^l)_{l\in\N}$ and \eqref{eq:def.PDp}, we see that $\| \nabla^m_{\D_p} (\psi \varphi) \|_{L^2(\Omega)}$ and  
$\| \Pi^m_{\D_p} (\psi \varphi) \|_{L^2(\Omega)}$ are bounded by constants  depending on $K$ and $\overline\varphi$ but independent of $l$. 
Then, from Proposition \ref{prop_timetranslates}, we have with hidden constants independent of $l$ but possibly depending on $K$ and $\overline\varphi$, that 
$$
\begin{array}{ll}
& \Big| 
\dsp \< \Pi^m_{\D_p} \psi  \left( [\phi_\D \Pi^m_{\D_p} s^\alpha_m](\tau) - [\phi_\D \Pi^m_{\D_p} s^\alpha_m](\tau') \right), \Pi^m_{\D_p} \varphi \>_{L^2(\Omega)} 
\Big| \\\\  
& = \dsp \Big| \< [\phi_\D \Pi^m_{\D_p} s^\alpha_m](\tau) - [\phi_\D \Pi^m_{\D_p} s^\alpha_m](\tau'), \Pi^m_{\D_p} (\psi \varphi) \>_{L^2(\Omega)} \Big|
\\\\
& \lsim  \dsp 
\sum_{n = n_\tau+1}^{n_{\tau'}} \dtn 
 \left( 
 \xi^{(1),\alpha,n+1}_m  \| \nabla^m_{\D_p} (\psi \varphi) \|_{L^2(\Omega)} +
 \xi^{(2),\alpha,n+1}_m  \| \Pi^m_{\D_p} (\psi \varphi) \|_{L^2(\Omega)}
 \right) \\\\
 & \lsim  \dsp 
 \left(\sum_{n = n_\tau+1}^{n_{\tau'}}
 \dtn  \left( 
 \left( \xi^{(1),\alpha,n+1}_m \right)^2 +
 \left( \xi^{(2),\alpha,n+1}_m \right)^2 
 \right) \right)^{1\over 2} \left(\sum_{n = n_\tau+1}^{n_{\tau'}}
 \dtn \right)^{1\over 2} \\\\
 & \dsp \lsim   |\tau-\tau'|^{1\over 2} + \Delta t^{1\over 2}. 
 \end{array}
$$
Since $\Pi^m_{\D_p} s^\alpha_m\in [0,1]$, $\phi_\D$ is bounded in $L^\infty(0,T;L^2(\Omega))$ 
(see \eqref{GD_closures} and \eqref{apriori.est}), and $\Pi^m_{\D_p} \psi$ is uniformly bounded, one has
\begin{equation}\label{translates_phi_s}
\begin{array}{lllll}
\Big| 
\dsp \< \Pi^m_{\D_p} \psi  \left( [\phi_\D \Pi^m_{\D_p} s^\alpha_m](\tau) - [\phi_\D \Pi^m_{\D_p} s^\alpha_m](\tau') \right), \overline{\varphi} \>_{L^2(\Omega)} 
\Big| 
\lsim   |\tau-\tau'|^{1\over 2} + \Delta t^{1\over 2} + \omega_{\D_p}, 
 \end{array}
\end{equation}
with $\omega_{\D_p}=\|\overline\varphi - \Pi^m_{\D_p} \varphi\|_{L^2(\Omega)}$ a consistency error term such that
$\lim_{l\rightarrow + \infty}\omega_{\D^l_p} = 0$.
It follows from the discontinuous Ascoli-Arzel\`a  theorem \cite[Theorem C.11]{gdm} that (up to a subsequence) the sequence $(\Pi^m_{\D_p} \psi)  \phi_\D (\Pi^m_{\D_p}  s^\alpha_m)=\phi_\D\Pi^m_{\D_p}  (s^\alpha_m \psi)$ converges 
uniformly in time weakly in $L^2(\Omega)$.

Let us now take $w\in C^\infty_c(\Omega\backslash\Gamma)$ and let $K$ be the support of $w$. For $l$ large enough, by definition of $\psi^l$ we have $\big(\phi_{\D^l} \Pi_{\D^l_p}^m s^{\alpha,l}_m\big)|_K  =  \phi_{\D^l} \Pi_{\D^l_p}^m (\psi^l  s^{\alpha,l}_m)$. Hence, 
\begin{equation}\label{eq:cv.unif}
\<\phi_{\D^l} \Pi_{\D^l_p}^m s^{\alpha,l}_m,w\>_{L^2(\Omega)}\mbox{ converges uniformly with respect to $t\in [0,T]$.}
\end{equation}
Since $(\phi_{\D^l} \Pi_{\D^l_p}^m  s^{\alpha,l}_m)_{l\in\N}$ is bounded in $L^\infty(0,T;L^2(\Omega))$, the density of $C^\infty_c(\Omega\backslash\Gamma)$ in $L^2(\Omega)$ shows that the convergence \eqref{eq:cv.unif} is valid for any $w\in L^2(\Omega)$, which concludes the proof that the sequence $\phi_{\D^l} \Pi_{\D^l_p}^m  s^{\alpha,l}_m$ converges uniformly in time, weakly in $L^2(\Omega)$. 

We deduce that the sequence $\phi_{\D^l}=\sum_{\alpha \in \{\rm nw,\rm w\}}\phi_{\D^l} \Pi_{\D^l_p}^m  s^{\alpha,l}_m$ also converges uniformly in time, weakly in $L^2(\Omega)$.
\end{proof}

\subsubsection{Uniform-in-time $L^2$-weak convergence of $d_{f,\D_\bu} \Pi^f_{\D_p} s^\alpha_f$ and $d_{f,\D_\bu}$}

\begin{proposition}
  \label{prop_uniftimeweakL2_dfsf}
  Under the assumptions of Proposition \ref{prop_compactness_Sm}, the sequences $(d_{f,\D^l_\bu})_{l\in\N}$ and $(d_{f,\D^l_\bu} \Pi^f_{\D_p}s_f^{\alpha,l})_{l\in \N}$,  with $s_f^{\alpha,l}= S_f^\alpha(p_c^l)$, converge up to a subsequence uniformly in time  weakly in $L^2(\Gamma)$. 
\end{proposition}

\begin{proof}
Let $K$ be a fixed compact set of $\Omega\setminus\Gamma$ and let us consider cut-off functions $\psi^l$ as defined in the cut-off property of $(\D_p^l)_{l\in\N}$. In the following, the superscript $l\in \N$ (assumed to be large enough) is dropped when not required for the clarity of the proof, and the hidden constants are independent of $l$.
Let $\overline{\varphi} \in C^\infty_c(\Omega)$ and set $\varphi = P_{\D_p}\overline{\varphi}$, with $P_{\D_p}$ characterised by \eqref{eq:def.PDp}. From Proposition \ref{prop_timetranslates} we have 
\begin{align*}
&\Big| \< [d_{f,\D_\bu} \Pi^f_{\D_p} s^\alpha_f](\tau) - [d_{f,\D_\bu} \Pi^f_{\D_p} s^\alpha_f](\tau'), \Pi^f_{\D_p}  \varphi \>_{L^2(\Gamma)}\Big|\\
& \lsim   
\Big| 
\dsp \<    \left( [\phi_\D \Pi^m_{\D_p} s^\alpha_m](\tau) - [\phi_\D \Pi^m_{\D_p} s^\alpha_m](\tau') \right), \Pi^m_{\D_p} \varphi \>_{L^2(\Omega)} 
\Big| \\
&\quad+ 
 \max\left(  \| \nabla^m_{\D_p}  \varphi \|_{L^2(\Omega)}, \| \nabla^f_{\D_p}  \varphi \|_{L^8(\Gamma)},\| \Pi^m_{\D_p}  \varphi \|_{L^2(\Omega)}, \| \Pi^f_{\D_p}  \varphi \|_{L^2(\Gamma)}  \right)  \\
& \qquad\times 
 \left(\sum_{n = n_\tau+1}^{n_{\tau'}}
 \dtn  \left( 
 \left( \xi^{(1),\alpha,n+1}_m \right)^2 +
 \left( \xi^{(1),\alpha,n+1}_f \right)^2 +
 \left( \xi^{(2),\alpha,n+1}_m \right)^2 +
 \left( \xi^{(2),\alpha,n+1}_f \right)^2
 \right) \right)^{1\over 2} \\
 & \qquad \times\left(\sum_{n = n_\tau+1}^{n_{\tau'}}
 \dtn \right)^{1\over 2} \\
 & \lsim  \( |\tau-\tau'|^{1\over 2} + \Delta t^{1\over 2}\)
 + \Big| 
\dsp \<    \left( [\phi_\D \Pi^m_{\D_p} s^\alpha_m](\tau) - [\phi_\D \Pi^m_{\D_p} s^\alpha_m](\tau') \right), \Pi^m_{\D_p} \varphi \>_{L^2(\Omega)} 
\Big|.
 \end{align*}

Since $\phi_\D \Pi^m_{\D_p} s^\alpha_m$ is bounded in $L^\infty(0,T;L^2(\Omega))$ (see the proof of Proposition \ref{prop_uniftimeweakL2_phimSm}),
we have
\begin{align*}
& \Big| 
\dsp \<   \left( [\phi_\D \Pi^m_{\D_p} s^\alpha_m](\tau) - [\phi_\D \Pi^m_{\D_p} s^\alpha_m](\tau') \right), \Pi^m_{\D_p} \varphi \>_{L^2(\Omega)} 
\Big|\\
&\lsim  \|\bar \varphi - \Pi^m_{\D_p} \varphi \|_{L^2(\Omega)} +
\Big| 
\dsp \< [\phi_\D \Pi^m_{\D_p}  s^\alpha_m](\tau) - [\phi_\D \Pi^m_{\D_p} s^\alpha_m](\tau'), \bar\varphi \>_{L^2(\Omega)}  
\Big|
 \end{align*}
and
$$
\Big| \< [d_{f,\D_\bu} \Pi^f_{\D_p} s^\alpha_f](\tau) - [d_{f,\D_\bu} \Pi^f_{\D_p} s^\alpha_f](\tau'), \bar\varphi-\Pi^f_{\D_p}  \varphi \>_{L^2(\Gamma)}\Big| \lsim
\|d_{f,\D_\bu}\|_{L^\infty(0,T;L^2(\Gamma))} \| \bar\varphi-\Pi^f_{\D_p}  \varphi \|_{L^2(\Gamma)}. 
$$
Using the a priori estimates of Lemma \ref{lemma_apriori}, and Proposition \ref{prop_uniftimeweakL2_phimSm} stating the uniform-in-time $L^2(\Omega)$-weak convergence of $\phi_\D \Pi^m_{\D_p}  s^\alpha_m$ (which implies the equi-continuity of the functions $\tau\mapsto \< [\phi_\D \Pi^m_{\D_p}  s^\alpha_m](\tau),\bar\varphi\>_{L^2(\Omega)}$), we deduce that
$$
\Big| \< [d_{f,\D_\bu} \Pi^f_{\D_p} s^\alpha_f](\tau) - [d_{f,\D_\bu} \Pi^f_{\D_p} s^\alpha_f](\tau'), \bar \varphi \>_{L^2(\Gamma)}\Big| \lsim
  \omega(|\tau-\tau'|) + \Delta t^{1\over 2} + \varpi_{\D_p}, 
$$
with $\lim_{h\rightarrow 0}\omega(h)=0$ and
$\varpi_{\D_p}=\|\overline\varphi - \Pi^m_{\D_p} \varphi\|_{L^2(\Omega)} + \| \bar\varphi-\Pi^f_{\D_p}  \varphi \|_{L^2(\Gamma)}$ a consistency error term such that
$\lim_{l\rightarrow + \infty}\varpi_{\D^l_p} = 0$.
It follows from the discontinuous Ascoli-Arzel\`a theorem \cite[Theorem C.11]{gdm} that (up to a subsequence) the sequence $d_{f,\D_\bu} \Pi^f_{\D_p} s^\alpha_f$ converges uniformly in time weakly in $L^2(\Gamma)$. Summing over $\alpha \in \{\rm nw,\rm w\}$, we also deduce the uniform-in-time $L^2(\Gamma)$-weak convergence of $d_{f,\D_\bu}$.
\end{proof}

\subsubsection{Strong convergence of $d_{f,\D_\bu}$, $d_{f,\D_\bu} \Pi^f_{\D_p} s_{f}^\alpha$, and $\Pi^f_{\D_p} s^\alpha_{f}$}

\begin{proposition}
  \label{prop_compactnessdfsf}
  Under the assumptions of Proposition \ref{prop_compactness_Sm}, the sequence $(d_{f,\D^l_\bu})_{l\in\N}$ converges up to a subsequence in $L^\infty(0,T;L^p(\Gamma))$ for all $2 \leq p < 4$,
  and the sequences $(d_{f,\D^l_\bu} \Pi^f_{\D_p}s_f^{\alpha,l})_{l\in\N}$ and $(\Pi^f_{\D_p}s_f^{\alpha,l})_{l\in \N}$,  with $s_f^{\alpha,l}= S_f^\alpha(p_c^l)$, converge up to a subsequence in $L^4(0,T;L^2(\Gamma))$. 
\end{proposition}

\begin{proof} 
By the characterization in Remark \ref{rem:compact.equivalent.Du} of the compactness of $(\D_\bu^l)_{l\in\N}$ and the
estimate on $\bbeps_{\D_\bu}(\bu)$ in Lemma \ref{lemma_apriori}, we have, for all $i\in I$, all $\eta_i$ tangent to $\Gamma_i$, a.e.\ $t\in (0,T)$ and all $s<4$,
$$
\left\|d_{f,\D_\bu^l}(t,\cdot+\eta_i)-d_{f,\D_\bu^l}(t,\cdot)\right\|_{L^s(\Gamma_i)}\le T_{\D_\bu^l,s}(0,\eta)\|\bbeps_{\D_\bu}(\bu)(t,\cdot)\|_{L^2(\Omega,\S_d(\R))}
\lsim T_{\D_\bu^l,s}(0,\eta),
$$
where $\eta=(0,\ldots,0,\eta_i,0,\ldots,0)$ and $d_{f,\D_\bu^l}$ has been extended by 0 in the hyperplane spanned by $\Gamma_i$. Together with the uniform-in-time $L^2(\Gamma)$-weak convergence of $d_{f,\D_\bu^l}$ from Proposition \ref{prop_uniftimeweakL2_dfsf}, this shows that we can apply Lemma \ref{lemme2} to $d_{f,\D_\bu^l}$ with $p=+\infty$ and get the convergence of this sequence in $L^\infty(0,T;L^2(\Gamma))$.
Since, from the a priori estimates of Lemma \ref{lemma_apriori}, this sequence $d_{f,\D_\bu^l}$ is bounded in $L^\infty(0,T;L^4(\Gamma))$, it follows that it converges in $L^\infty(0,T;L^q(\Gamma))$ for all $2 \leq q < 4$.
 
For any compact set $K_{f}\subset \Gamma$ that is disjoint from the intersections $(\overline{\Gamma}_i\cap\overline{\Gamma}_j)_{i\not=j}$, using that $\Pi^f_{\D_p} s^\alpha_{f}\in [0,1]$, that $\|d_{f,\D_\bu}(t,\cdot) \|_{L^4(\Gamma)}$ is uniformly bounded in $t$, and the Lipschitz properties of $S^\alpha_f$, it follows that, for all $i\in I$ and $\eta_i$ tangent to $\Gamma_i$ small enough, 
\begin{align*}
  & \|[d_{f,\D_\bu} \Pi^f_{\D_p} s^\alpha_{f}](t,\cdot+\eta_i) - [d_{f,\D_\bu} \Pi^f_{\D_p} s^\alpha_{f}](t,\cdot) \|_{L^2(K_f\cap\Gamma_i)}\\
  &\leq
  \|d_{f,\D_\bu}(t,\cdot+\eta_i) - d_{f,\D_\bu}(t,\cdot) \|_{L^2(K_f\cap\Gamma_i)}\\
& \qquad  + \|\Pi^f_{\D_p} s_{f}^\alpha(t,\cdot+\eta_i) - \Pi^f_{\D_p} s^\alpha_{f}(t,\cdot) \|_{L^4(K_f\cap\Gamma_i)} \|d_{f,\D_\bu}(t,\cdot) \|_{L^4(K_f\cap\Gamma_i)}\\
  &\lsim 
  \|d_{f,\D_\bu}(t,\cdot+\eta_i) - d_{f,\D_\bu}(t,\cdot) \|_{L^2(K_f\cap\Gamma_i)}
  + \|\Pi^f_{\D_p} s_{f}^\alpha(t,\cdot+\eta_i) - \Pi^f_{\D_p} s^\alpha_{f}(t,\cdot) \|_{L^2(K_f\cap\Gamma_i)}^{1\over 2} \\
  & \lsim \|d_{f,\D_\bu}(t,\cdot+\eta_i) - d_{f,\D_\bu}(t,\cdot) \|_{L^2(K_f\cap\Gamma_i)}
  + \|\Pi_{\D_p}^f p_{c}(t,\cdot+\eta_i) - \Pi_{\D_p}^f p_{c}(t,\cdot) \|_{L^2(K_f\cap\Gamma_i)}^{1\over 2}. 
\end{align*}
From the compactness properties of $(\D_\bu^l)_{l\in\N}$ and $(\D_p^l)_{l\in\N}$ (see Remarks \ref{rem:compact.equivalent} and \ref{rem:compact.equivalent.Du}) it results that  
\begin{align*}
&  \sum_{i\in I}\Big\| \sup_{|\eta_i|\leq \delta} \|[d_{f,\D_\bu} \Pi^f_{\D_p} s^\alpha_{f}](\cdot,\cdot+\eta_i) - [d_{f,\D_\bu} \Pi^f_{\D_p} s^\alpha_{f}](\cdot,\cdot) \|_{L^2(K_f\cap\Gamma_i)}\Big\|_{L^4(0,T)}\\
  &  \lsim  T_{K_f}(\delta) \(
\|\bbeps_{\D_\bu}(\bu)\|_{L^\infty(0,T;L^2(\Omega))} 
  + \sum_{\alpha \in \{\g,\l\}}  (\|d_0^{\nf 3 2}\nabla_{\D_p}^f p^\alpha\|_{L^2(0,T;L^2(\Gamma))} + \|\nabla_{\D_p}^m p^\alpha\|_{L^2(0,T;L^2(\Omega))} ) \)
\end{align*}
with $\lim_{\delta \rightarrow 0}T_{K_f}(\delta) = 0$. 
From the a priori estimates of Lemma \ref{lemma_apriori}, and the uniform-in-time $L^2(\Gamma)$-weak convergence of $d_{f,\D_\bu} s^\alpha_{f}$ of Proposition \ref{prop_uniftimeweakL2_dfsf}, it follows from  Lemma \ref{lemme2} that $d_{f,\D_\bu}\Pi^f_{\D_p}s^\alpha_{f}$ converges up to a subsequence in $L^4(0,T;L^2(K_f))$.

From the assumption $d_{f,\D_\bu}(t,\x) \geq d_0(\x)$, $d_{f,\D_\bu}$ is bounded below by a strictly positive constant on $K_f$. Writing that $\Pi^f_{\D_p} s^\alpha_{f} = {1 \over d_{f,\D_\bu}} (d_{f,\D_\bu} \Pi^f_{\D_p} s^\alpha_{f})$, it follows that $\Pi^f_{\D_p} s^\alpha_{f}$ converges in $L^4(0,T;L^2(K_f))$. Since this is true for any $K_f$ compact in $\Gamma$ that does not touch the fractures intersections, and since $\Pi^f_{\D_p} s^\alpha_{f}\in [0,1]$, we deduce that $\Pi^f_{\D_p} s^\alpha_{f}$ converges in $L^4(0,T;L^2(\Gamma))$.
\end{proof}

\subsection{Convergence to a weak solution}\label{subsec:convergence}

\begin{proof}[Proof of Theorem~\ref{thm.conv}]
  The superscript $l$ will be dropped in the proof, and all convergences are up to appropriate subsequences. From Lemma \ref{lemma_apriori} and Proposition \ref{prop_compactnessdfsf}, there exist $\bar d_f\in L^\infty(0,T;L^4(\Gamma))$ 
and $\bar s_f^\alpha\in L^\infty((0,T)\times\Gamma)$ such that
\begin{equation}
  \label{conv_dfsf}
\begin{array}{lll}
  & d_{f,\D_\bu} \rightarrow \bar d_f & \mbox{in } L^\infty(0,T;L^p(\Gamma)),\, 2\leq p < 4, \\[1ex]
  & \Pi^f_{D_p} S^\alpha_f(p_c) \rightarrow \bar s^\alpha_f & \mbox{in } L^4(0,T;L^2(\Gamma)).
\end{array} 
\end{equation}
From Proposition \ref{prop_compactness_Sm}, there exists $\bar s_m^\alpha\in L^\infty((0,T)\times\Omega)$ such that
\begin{equation}
  \label{conv_sm}
\begin{array}{lll}
  & \Pi^m_{D_p} S^\alpha_m(p_c) \rightarrow \bar s^\alpha_m & \mbox{in } L^2(0,T;L^2(\Omega)).
\end{array} 
\end{equation}
The identification of the limit \cite[Lemma~5.5]{BGGLM16}, resulting from the limit-conformity property, can easily be adapted to our definition of $V_0$, with weight $d_0^{\nf 3 2}$ and the use in the definition of limit-conformity of fracture flux functions that are compactly supported away from the tips. Using this lemma and the a priori estimates of Lemma \ref{lemma_apriori}, we obtain $\bar p^\alpha\in L^2(0,T;V_0)$ and  ${\bf g}^\alpha_f \in L^2(0,T;L^2(\Gamma)^{d-1})$, such that the following weak limits hold
\begin{equation}
  \label{conv_pnablap}
\left. 
\begin{array}{lll}
  & \Pi^m_{\D_p} p^\alpha \weakto \bar p^\alpha & \mbox{in } L^2(0,T;L^2(\Omega))\mbox{ weak}, \\
  & \Pi^f_{\D_p} p^\alpha \weakto \gamma \bar p^\alpha & \mbox{in } L^2(0,T;L^2(\Gamma))\mbox{ weak}, \\
  & \nabla^m_{\D_p} p^\alpha \weakto  \nabla \bar p^\alpha & \mbox{in } L^2(0,T;L^2(\Omega)^d)\mbox{ weak}, \\
  & d_0^{\nf 3 2}\nabla^f_{\D_p} p^\alpha \weakto d_0^{\nf 3 2} \nabla_\tau \gamma \bar p^\alpha & \mbox{in } L^2(0,T;L^2(\Gamma)^{d-1})\mbox{ weak}, \\
    & d_{f,\D_\bu}^{\nf 3 2}\nabla^f_{\D_p} p^\alpha \weakto {\bf g}^\alpha_f & \mbox{in } L^2(0,T;L^2(\Gamma)^{d-1})\mbox{ weak}. 
\end{array}
\right.
\end{equation}
Let $\boldsymbol{\varphi} \in C_c^0((0,T)\times\Gamma)^{d-1}$ whose support is contained in $(0,T)\times K$, with $K$ compact set not containing the tips of $\Gamma$. We have
$$
\int_{0}^T\int_\Gamma d_{f,\D_\bu}^{\nf 3 2}\nabla^f_{\D_p} p^\alpha \cdot \boldsymbol{\varphi} ~\d\sigma(\x)\d t \rightarrow \int_{0}^T\int_\Gamma {\bf g}^\alpha_f \cdot \boldsymbol{\varphi}~ \d\sigma(\x)\d t. 
$$
On the other hand, it results from  \eqref{conv_pnablap} and the fact that $d_0$ is bounded away from $0$ on $K$ (because $d_0$ is continuous and does not vanish outside the tips of $\Gamma$) that $\nabla^f_{\D_p} p^\alpha \weakto \nabla_\tau \gamma \bar p^\alpha$ in $L^2(0,T;L^2(K)^{d-1})$. 
Combined with the convergence $d_{f,\D_\bu}^{\nf 3 2} \boldsymbol{\varphi} \rightarrow (\bar d_f)^{\nf 3 2}\boldsymbol{\varphi}$ in $L^{\infty}(0,T;L^2(\Gamma)^{d-1})$ given by \eqref{conv_dfsf}, we infer that 
\begin{equation*}
\int_{0}^T\int_\Gamma d_{f,\D_\bu}^{\nf 3 2}\nabla^f_{\D_p} p^\alpha \cdot \boldsymbol{\varphi}~ \d\sigma(\x)\d t
 \rightarrow 
\int_{0}^T\int_\Gamma (\bar d_f)^{\nf 3 2} \nabla_\tau \gamma \bar p^\alpha \cdot \boldsymbol{\varphi} ~\d\sigma(\x)\d t. 
\end{equation*}
This shows that ${\bf g}^\alpha_f = (\bar d_f)^{\nf 3 2} \nabla_\tau \gamma \bar p^\alpha$ on $(0,T)\times\Gamma$.

Combining the strong convergence of $\Pi^m_{D_p} S^\alpha_m(p_c)= S^\alpha_m(\Pi^m_{D_p} p_c)$ (resp.~of $\Pi^f_{D_p} S^\alpha_f(p_c) = S^\alpha_f(\Pi^f_{D_p} p_c)$), 
the weak convergence of  $\Pi^m_{D_p} p_c$ (resp. $\Pi^f_{D_p} p_c$), and the monotonicity of $S^\alpha_m$ (resp. $S^\alpha_f)$, it results from the Minty trick (see e.g.~\cite[Lemma 2.6]{EGHM13}) that $\bar s_m^\alpha = S_m^\alpha(\bar p_c)$ (resp.~$\bar s_f^\alpha = S_f^\alpha(\gamma \bar p_c)$) with $\bar p_c = \bar p^\g-\bar p^\l$. 

From the a priori estimates of Lemma \ref{lemma_apriori} and the limit-conformity property of the sequence of GDs $(\D_\bu^l)_{l\in \N}$ (see Lemma \ref{lemma_limitconformityDu}), there exists $\bar\bu\in L^\infty(0,T;\U_0)$, such that
\begin{equation}
   \label{conv_unablau}
\left.
\begin{array}{lll}  
  & \Pi_{\D_\bu} \bu \weakto \bar \bu & \mbox{in } L^\infty(0,T;L^2(\Omega)^d) \mbox{ weak $\star$},\\
  & \bbeps_{\D_\bu} (\bu) \weakto \bbeps(\bar\bu) & \mbox{in } L^\infty(0,T;L^2(\Omega,\S_d(\R))) \mbox{ weak $\star$},\\
  & \div_{\D_\bu} \bu \weakto \div(\bar\bu) & \mbox{in } L^\infty(0,T;L^2(\Omega)) \mbox{ weak $\star$},\\
  & d_{f,\D_\bu} = -\jump{\bu}_{\D_\bu} \weakto -\jump{\bar\bu} & \mbox{in } L^\infty(0,T;L^2(\Gamma)) \mbox{ weak $\star$},
\end{array}
\right.
\end{equation}
from which we deduce that $\bar d_f = -\jump{\bar\bu}$ and that $\bbsig_{\D_u}(\bu)$ converges to $\bbsig(\bar \bu)$ in $L^\infty(0,T;L^2(\Omega,\S_d(\R)))$ weak $\star$. 

From the a priori estimates and the closure equations
\eqref{GD_closures}, there exist $\bar \phi_m \in L^\infty(0,T;L^2(\Omega)$ and $\bar p^E_m \in L^\infty(0,T;L^2(\Omega)$ such that  
\begin{equation}
   \label{conv_phipEm}
\left.
\begin{array}{lll}  
  & \phi_\D \weakto \bar \phi_m & \mbox{in } L^\infty(0,T;L^2(\Omega)) \mbox{ weak $\star$},\\
  & \Pi^m_{\D_p} p^E_m \weakto \bar p^E_m & \mbox{in } L^\infty(0,T;L^2(\Omega)) \mbox{ weak $\star$}. 
\end{array}
\right.
\end{equation}
Since $0\le  U_{\rm rt}(z) = \int_0^p z (S^{\g}_{\rm rt})'(z) \d z \leq 2 |p|$ for ${\rm rt} \in \{m,f\}$,
it results from the a priori estimates of Lemma \ref{lemma_apriori} that there exist $\bar p^E_f \in L^2(0,T;L^2(\Gamma))$, $\bar U_f \in L^2(0,T;L^2(\Gamma))$ and $\bar U_m \in L^2(0,T;L^2(\Omega))$ such that
\begin{equation}
   \label{conv_pEfUmf}
\left.
\begin{array}{lll}  
  & \Pi^f_{\D_p} p^E_f \weakto \bar p^E_f & \mbox{in } L^2(0,T;L^2(\Gamma)) \mbox{ weak},\\
  & \Pi^f_{\D_p} U_f(p_c) \weakto \bar U_f & \mbox{in } L^2(0,T;L^2(\Gamma))\mbox{ weak},\\
  & \Pi^m_{\D_p} U_m(p_c) \weakto \bar U_m & \mbox{in } L^2(0,T;L^2(\Omega))\mbox{ weak}.   
\end{array}
\right.
\end{equation}
For ${\rm rt}\in \{\g,\l\}$, it is shown in \cite{DHM16}, following ideas from \cite{DE15}, that $U_{\rm rt}(p) = B_{\rm rt}(S^\g_{\rm rt}(p))$ where $B_{\rm rt}: [0,1]\mapsto \R$ is a convex lower semi-continuous function  with finite limits at $s=0$ and $s=1$ (note that $B_{\rm rt}$ is therefore actually continuous on $[0,1]$).
Since $\Pi^{m}_{\D_p} s^\g_{m}$ converges strongly in $L^2((0,T)\times\Omega)$  to $S^\g_m(\bar p_c)$, it converges a.e.\ in $(0,T)\times\Omega$.
It  results that $B_{m}(\Pi^{m}_{\D_p} s^\g_{m})$  converges a.e.\ in $(0,T)\times\Omega$ to $B_{m}(S^\g_{m}(\bar p_c))$, and hence that 
$\bar U_{m} =  B_{m}(S^\g_{m}(\bar p_c)) = U_{m}(\bar p_c)$. Similarly, $\bar U_{f} =  B_{f}(S^\g_{f}(\gamma \bar p_c)) = U_{f}(\gamma \bar p_c)$. We deduce
that
$$
\bar p^E_m = \sum_{\alpha\in \{\g,\l\}} \bar p^\alpha S^\alpha_m(\bar p_c) - U_m(\bar p_c) \quad \mbox{ and } \quad 
\bar p^E_f = \sum_{\alpha\in \{\g,\l\}} \gamma \bar p^\alpha S^\alpha_f(\gamma \bar p_c) - U_f(\gamma\bar p_c). 
$$
Using the estimate
$$
 \left|U_{\rm rt}(p_2)-U_{\rm rt}(p_1)\right|  = \left|\int^{p_2}_{p_1} z (S^{\g}_{\rm rt})'(z) \d z\right| \leq |p_2-p_1| + |p_2 S^{\g}_{\rm rt}(p_2) - p_1 S^{\g}_{\rm rt}(p_1)|,
$$
the Lipschitz property of $S^{\g}_{\rm rt}$, $\bar p^\alpha_0 \in V_0\cap L^\infty(\Omega)$, $\gamma\bar p^\alpha_0\in L^\infty(\Gamma)$, $\alpha\in \{\g,\l\}$, and the consistency of the sequence of GDs $(\D_p^l)_{l\in \N}$, we deduce that
  \begin{equation}
    \label{conv_pE0}
\left.
\begin{array}{lll}  
  & \Pi^m_{\D_p} p^{E,0}_m \rightarrow \bar p^{E,0}_m & \mbox{in } L^2(\Omega),\\
  & \Pi^f_{\D_p} p^{E,0}_f \rightarrow \bar p^{E,0}_f & \mbox{in } L^2(\Gamma). 
\end{array}
\right.
  \end{equation}
  Then, from Proposition \ref{error_estimate_mechanics} it holds that 
  \begin{equation}
    \label{conv_u0}
\left.
\begin{array}{lll}  
  & \div_{\D_\bu}(\bu^0) \rightarrow \div(\bar\bu^0)  & \mbox{in } L^2(\Omega),\\[1ex]
  & \jump{\bu^0}_{\D_\bu} \rightarrow \jump{\bar\bu^0} = - \bar d_f^0  & \mbox{in } L^2(\Gamma).    
\end{array}
\right.
\end{equation}
It results from  \eqref{conv_phipEm}, \eqref{conv_unablau}, \eqref{conv_pE0} and \eqref{conv_u0} and the definition of $\phi_\D$ that   
$$
\bar \phi_m = \bar \phi_m^0 +  b~\div(\bar\bu-\bar\bu^0) + \frac{1}{M} (\bar p^E_m -\bar p_m^{E,0}). 
$$

Let us now prove that the functions $\bar p^\alpha$, $\alpha\in \{\g,\l\}$, and $\bar\bu$ satisfy the variational formulation \eqref{eq_var_hydro}--\eqref{eq_var_meca} by passing to the limit in the gradient scheme \eqref{eq:GS}.

For $\theta\in C^\infty_c([0,T))$ and $\psi\in C_c^\infty(\Omega)$ let us set, with $P_{\D_p}$ characterised by \eqref{eq:def.PDp},
$$
\varphi = (\varphi^1,\ldots,\varphi^{N}) \in (X^0_{\D_p})^{N} \mbox{ with } \varphi^i = \theta(t_{i-1}) (P_{\D_p}\psi).
$$
From the consistency properties of $(\D_p^l)_{l\in \N}$ with given $r> 8$, we deduce that
\begin{equation}
   \label{conv_thetapDpsi}
\left.
\begin{array}{lllll}
  & \Pi^m_{\D_p}  P_{\D_p}\psi \rightarrow \psi & \mbox{in } L^2(\Omega),\quad &
   \Pi^f_{\D_p}  P_{\D_p}\psi \rightarrow \gamma \psi & \mbox{in } L^2(\Gamma),\\  
  & \Pi^m_{\D_p}\varphi  \rightarrow \theta\psi & \mbox{in } L^\infty(0,T;L^2(\Omega)),\quad
  & \Pi^f_{\D_p} \varphi \rightarrow \theta\gamma\psi & \mbox{in } L^\infty(0,T;L^2(\Gamma)),\\
  & \nabla^m_{\D_p}\varphi  \rightarrow \theta \nabla \psi & \mbox{in } L^\infty(0,T;L^2(\Omega)^d),\quad
  & \nabla^f_{\D_p} \varphi \rightarrow \theta\nabla_\tau \gamma\psi & \mbox{in } L^\infty(0,T;L^r(\Gamma)^{d-1}). 
\end{array}
\right.
\end{equation}

Setting
$$
\left.\begin{array}{lll}
 && T_1 = \dsp \int_0^T \int_\Omega  \delta_t \(\phi_\D \Pi_{\D_p}^m s^\alpha_m \)\Pi_{\D_p}^m \varphi ~\d\x \d t \\[2ex]
 && T_2=  \dsp \int_0^T \int_\Omega \eta_m^\alpha(\Pi_{\D_p}^m s_m^\alpha) \K_m \nabla_{\D_p}^m p^\alpha \cdot   \nabla_{\D_p}^m \varphi  ~\d\x \d t\\[2ex]
 && T_3 = \dsp \int_0^T \int_\Gamma \delta_t \(d_{f,\D_\bu} \Pi_{\D_p}^f s^\alpha_f \)\Pi_{\D_p}^f \varphi ~\d\sigma(\x)\d t\\[2ex]
  && T_4= \dsp \int_0^T \int_\Gamma  \eta_f^\alpha(\Pi_{\D_p}^f s_f^\alpha) {d_{f,\D_\bu}^3 \over 12} \nabla_{\D_p}^f p^\alpha \cdot   \nabla_{\D_p}^f \varphi  ~\d\sigma(\x) \d t\\[2ex]
  && T_5 = \dsp \int_0^T \int_\Omega h_m^\alpha \Pi_{\D_p}^m \varphi ~d\x \d t + \int_0^T \int_\Gamma h_f^\alpha \Pi_{\D_p}^f \varphi ~\d\sigma(\x)\d t, 
  \end{array}\right.
$$
the gradient scheme variational formulation \eqref{GD_hydro} states that
$$
T_1 + T_2 + T_3 + T_4 = T_5. 
$$
For $\omega\in C^\infty_c([0,T))$ and a smooth function ${\bf w}:  \Omega\setminus\overline\Gamma \rightarrow \R^d$ vanishing on $\partial\Omega$ and admitting finite limits on each side of $\Gamma$, let us set 
$$
\bv = (\bv^1,\ldots,\bv^{N}) \in (X^0_{\D_\bu})^{N} \mbox{ with } \bv^i = \omega(t_{i-1}) (P_{\D_\bu}{\bf w})
$$
where $P_{\D_\bu}\mathbf{w}$ realizes the minimum in the definition \eqref{eq:def.SDu} of $\mathcal S_{\D_\bu}(\textbf{w})$.
From the consistency properties of $(\D_\bu^l)_{l\in \N}$, we deduce that 
\begin{equation}
   \label{conv_omegaw}
\left.
\begin{array}{lll}  
  & \Pi_{\D_\bu}\bv  \rightarrow \omega\psi & \mbox{in } L^\infty(0,T;L^2(\Omega)^d),\\ 
  & \bbeps_{\D_\bu}(\bv)  \rightarrow \omega \bbeps({\bf w}) & \mbox{in } L^\infty(0,T;L^2(\Omega,\S_d(\R))),\\
  & \jump{\bv}_{\D_\bu}  \rightarrow \omega \jump{{\bf w}} & \mbox{in } L^\infty(0,T;L^2(\Gamma)). 
\end{array}
\right.
\end{equation}
Setting 
$$
  \left.\begin{array}{lll}
    && T_6 =  \dsp \int_0^T \int_\Omega \( \bbsig_{\D_u}(\bu) : \bbeps_{\D_\bu}(\bv)
    -   b(\Pi_{\D_p}^m p_m^E)  \div_{\D_\bu}(\bv)\)d\x \d t,\\[2ex]
    && T_7 =   \dsp \int_0^T \int_\Gamma (\Pi_{\D_p}^f p_f^E)  \jump{\bv}_{\D_\bu} \d\sigma(\x)\d t,\\[2ex]
    && T_8 = \dsp \int_0^T \int_\Omega \mathbf{f} \cdot \Pi_{\D_\bu} \bv ~\d\x \d t.  
  \end{array}\right.
$$
the gradient scheme variational formulation \eqref{GD_meca} states that
$$
T_6 + T_7 = T_8. 
$$
Using a discrete integration by part \cite[Section D.1.7]{gdm}, we have $T_1 = T_{11} + T_{12}$ with 
\begin{align*}
& T_{11} = \dsp - \int_0^T \int_\Omega  \phi_\D (\Pi_{\D_p}^m s^\alpha_m) (\Pi_{\D_p}^m P_{\D_p} \psi) \theta'(t) ~\d\x \d t, \\
& T_{12} = - \int_\Omega  (\Pi^m_{\D_p} I_{\D_p}^m \bar \phi^0)  (\Pi_{\D_p}^m S^\alpha_m(I_{\D_p}\bar p^\alpha_0))  (\Pi_{\D_p}^m P_{\D_p}\psi) \theta(0) ~\d\x.  
\end{align*}
Using \eqref{conv_thetapDpsi} and \eqref{conv_phipEm}, and that $\Pi_{\D_p}^m s^\alpha_m \in [0,1]$ converges to $S^\alpha_m(\bar p_c)$ a.e.\ in $(0,T)\times\Omega$ (this follows from \eqref{conv_sm}), it holds that
$$
T_{11} \rightarrow - \int_0^T \int_\Omega  \bar \phi_m S^\alpha_m(\bar p_c)  \psi \theta'(t) ~\d\x \d t. 
$$  
Using  \eqref{conv_thetapDpsi}, that $\Pi^m_{\D_p} I_{\D_p}^m \bar \phi^0$ converges in $L^2(\Omega)$ to $\bar \phi^0$ and that 
$\Pi_{\D_p}^m S^\alpha_m(P_{\D_p}\bar p^\alpha_0) \in [0,1]$ converges a.e.\ in $\Omega$ to $S^\alpha_m(\bar p^\alpha_0)$, we deduce that 
$$
T_{12} \rightarrow - \int_\Omega  \bar \phi^0  S^\alpha_m(\bar p^\alpha_0)  \psi \theta(0) ~\d\x.  
$$
Writing $T_3 = T_{31} + T_{32}$ with 
\begin{align*}
& T_{31} = \dsp - \int_0^T \int_\Gamma  d_{f,\D_\bu} (\Pi_{\D_p}^f s^\alpha_f) (\Pi_{\D_p}^f P_{\D_p} \psi) \theta'(t) ~\d\sigma(\x) \d t, \\
& T_{32} = \int_\Gamma  \jump{\bu^0}_{\D_\bu}  (\Pi_{\D_p}^f S^\alpha_f( I_{\D_p} \bar p^\alpha_0))  (\Pi_{\D_p}^f P_{\D_p}\psi) \theta(0) ~\d\sigma(\x),   
\end{align*}
we obtain, using similar arguments and \eqref{conv_u0}, that
$$
T_{31} \rightarrow \dsp - \int_0^T \int_\Gamma  \bar d_f S^\alpha_f(\gamma\bar p_c) \gamma\psi \theta'(t) ~\d\sigma(\x) \d t,
$$
and
$$
T_{32} \rightarrow - \int_\Gamma  \bar d_f^0  S^\alpha_f(\gamma \bar p^\alpha_0) \gamma \psi \theta(0) ~\d\sigma(\x).  
$$
Using that $0\le \eta^\alpha_m(\Pi_{\D_p}^m s_m^\alpha)\leq \eta^\alpha_{\rm m,max}$, the continuity of $\eta_m^\alpha$,
the convergence of $\Pi_{\D_p}^m s_m^\alpha$ a.e.\ in $(0,T)\times \Omega$ to $S^\alpha_m(\bar p_c)$, 
\eqref{conv_pnablap} and \eqref{conv_thetapDpsi}, it holds that
$$
T_2 \rightarrow \dsp \int_0^T \int_\Omega \eta_m^\alpha(S_m^\alpha(\bar p_c)) \K_m \nabla \bar p^\alpha \cdot   \theta \nabla \psi  ~\d\x \d t.
$$
The convergence 
$$
T_4 \rightarrow   \int_0^T \int_\Gamma \eta^\a_f ( S^{\a}_f(\gamma \bar p_c)) {\bar d_f^{\;3}\over 12}  \nabla_\tau \gamma \bar p^\a \cdot \theta \nabla_\tau \gamma \psi~  \d\sigma(\x) \d t
$$
is established using $0\le \eta^\alpha_f(\Pi_{\D_p}^f s_f^\alpha)\leq \eta^\alpha_{\rm f,max}$, the continuity of $\eta_f^\alpha$, the convergence of $\Pi_{\D_p}^f s_f^\alpha$ a.e.\ in $(0,T)\times \Gamma$ to $S^\alpha_f(\gamma\bar p_c)$, combined with the weak convergence of $d_{f,\D_\bu}^{\nf 3 2} \nabla_{\D_p}^f p^\alpha$ to $\bar d_{f}^{\nf 3 2} \nabla_\tau\gamma \bar p^\alpha$ in $L^2((0,T)\times\Gamma)^{d-1}$, the strong convergence of $d_{f,\D_\bu}^{\nf 3 2}$ to $\bar d_{f}^{\nf 3 2}$ in
$L^s((0,T)\times\Gamma)$ for all $2\leq s < {8\over 3}$ (resulting from \eqref{conv_dfsf}), and the strong convergence \eqref{conv_thetapDpsi} of $\nabla_{\D_p}^f \varphi$ to $\theta \nabla_\tau\gamma \psi$ in $L^\infty(0,T;L^r(\Gamma))$ with $r>8$. 

The convergence
$$
T_5 \rightarrow \dsp \int_0^T \int_\O h_m^\a  ~\theta \psi  ~\d\x \d t + \int_0^T \int_\G h_f^\a ~\theta (\gamma  \psi)~  \d\sigma(\x) \d t
$$
is readily obtained from \eqref{conv_thetapDpsi}. The following convergences of $T_6$, $T_7$, $T_8$
$$
\begin{aligned}
T_6 & \rightarrow \dsp \int_0^T \int_\O \( \bbsig(\bar \bu): \bbeps({\bf w})\omega - b \bar p_m^E \div( {\bf w}) \omega\) ~\d\x \d t,\\
T_7 & \rightarrow \int_0^T \int_\G \bar p_f^E ~\jump{ {\bf w}}\omega  ~\d\sigma(\x) \d t, \\
T_8 & \rightarrow \int_0^T \int_\Omega \mathbf{f}\cdot {\bf w} \omega~ \d\x \d t  
\end{aligned}
$$
classically result from the strong convergences \eqref{conv_omegaw} combined with the weak convergences \eqref{conv_unablau}.

Using the above limits in $T_1+T_2+T_3+T_4=T_5$ and $T_6+T_7=T_8$ concludes the proof that $\bar p^\alpha$, $\alpha\in \{\g,\l\}$, and $\bar\bu$ satisfy the variational formulation \eqref{eq_var_hydro}--\eqref{eq_var_meca}.
\end{proof}

\section{Two-dimensional numerical example}\label{sec:numerical.example}

The objective of this section is to numerically investigate the convergence of the discrete solutions on a simple geometrical configuration based on a cross-shaped fracture network. We refer to \cite{bonaldi2020} for the presentation of a more advanced application to the desaturation by suction at the interface between a ventilation tunnel and a Callovo-Oxfordian argilite fractured storage rock.

\subsection{Setting}

Let us consider %\kb{\sout{problem} 
the system \eqref{eq_edp_hydromeca}--\eqref{closure_laws}
in the square domain $\Omega = (0,L)^2$, with $L=100\,\text{m}$, lying in the $xy$-plane and containing a cross-shaped fracture network $\Gamma$ made up of four fractures 
 (cf.~Figure~\ref{test_case}), each of length $\frac{L}{8}$, aligned with the coordinate 
 axes and intersecting at the center of the domain $(\frac{L}{2},\frac{L}{2})$. More precisely, the fracture network is defined as follows:
$\overline\Gamma = \bigcup_{i=1}^4 \overline\Gamma_i$, where $\Gamma_1 = (\frac{3}{8}L,\frac{L}{2})\times 
\{\frac{L}{2}\}$, $\Gamma_2 =  (\frac{L}{2},\frac{5}{8}L)\times \{\frac{L}{2}\}$,
$\Gamma_3 =   \{\frac{L}{2}\} \times (\frac{3}{8}L,\frac{L}{2})$, and
$\Gamma_4 =   \{\frac{L}{2}\} \times (\frac{L}{2},\frac{5}{8}L)$.
\begin{figure}
\centering
\includegraphics[scale=.7]{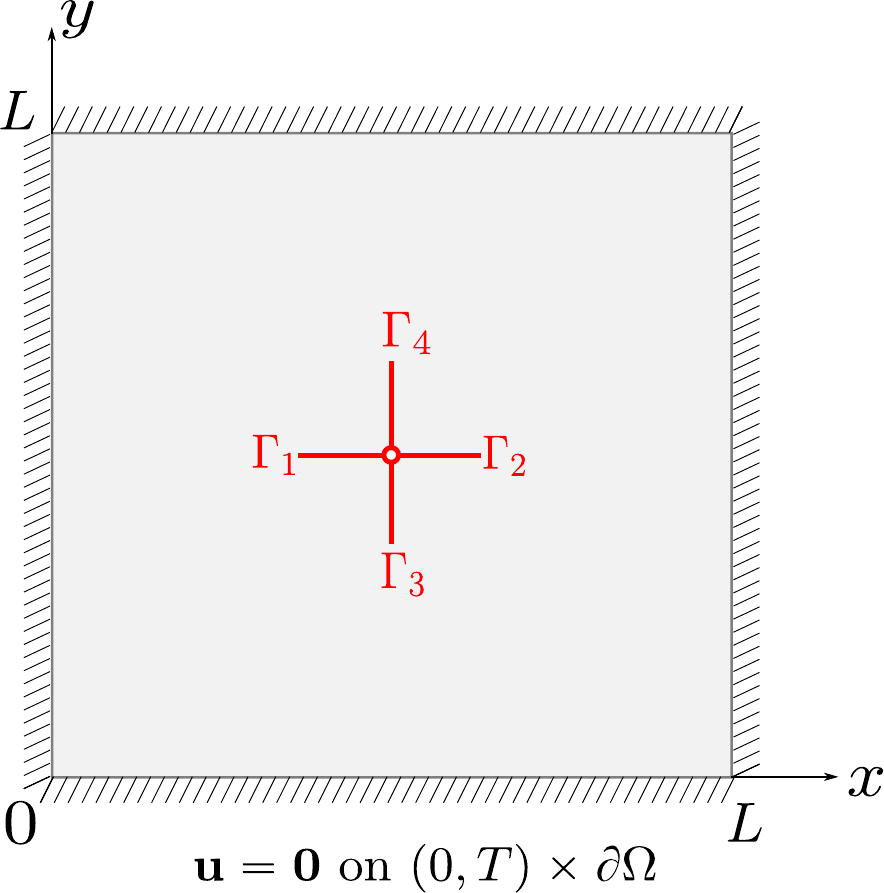}
\caption{Computational domain.}
\label{test_case}
\end{figure}
The data set employed is inspired by~\cite[Section~5.1.2]{dual-porosity}.
The matrix and fracture network have the following mobility laws:
\begin{equation}
  \label{def_mob}
  \eta_{m}^\alpha(s^\alpha) = \frac{(s^\alpha)^2}{\mu^\alpha}, \quad \quad \eta_f^\alpha(s^\alpha) = \frac{s^\alpha}{\mu^\alpha}, \quad \alpha \in \{\l,\g\},
\end{equation}
where $\mu^\l = 10^{-3}\,\rm{Pa{\cdot}s}$ and $\mu^\g = 1.851{\cdot}10^{-5}\,\rm{Pa{\cdot}s}$ are the dynamic viscosities of the wetting and non-wetting phases, respectively. Notice that $\eta_m^\alpha$ and $\eta_f^\alpha$ do not satisfy the assumptions of our analysis, as they are not bounded below by a strictly positive number; these choices are however physically relevant, and as the test shows, do not seem to impair the convergence of the numerical scheme. A non-degenerate regularization of these mobilities is also investigated below. The function yielding the saturation in both rock types in terms of the capillary pressure is provided by
 Corey's law:
$$s_{\rm rt}^\g = S_{\rm rt}^\g (p_c) = \max\(1 - \exp\(-\frac{p_c}{R_{\rm rt}}\),0\),
\quad {\rm rt}\in \{m,f\},$$
with $R_m=10^4\,\rm{Pa}$ and $R_f = 10\,\rm{Pa}$. The matrix is homogeneous and isotropic,
i.e.~$\mathbb K_m = \Lambda_m \mathbb I$, characterized by a permeability $\Lambda_m = 3{\cdot}10^{-15}\,\rm m^2$, an initial porosity $\phi_m^0 = 0.2$, effective Lam\'e parameters $\lambda = 833\,{\rm MPa}$, $\mu = 1250\,{\rm MPa}$, effective (drained) bulk modulus\footnote{%
In general, $K_{\rm dr} = \lambda + {2\mu}/{d}$, where $d\in\{2,3\}$ is the space dimension.}
$K_{\rm dr} = \lambda + \mu = 2083\,{\rm MPa}$, and solid grain bulk modulus $K_{\rm s} = 11244\,{\rm MPa}$. From these, one can infer the values of the Biot coefficient $b=1-\frac{K_{\rm dr}}{K_{\rm s}}\simeq0.81$, and of the Biot modulus $M=\frac{K_{\rm s}}{b-\phi^0_m}\simeq18.4\,{\rm GPa}$.
%The densities of the fluids are $\rho^\l = 1000\,{\rm kg/m^3}$ and $\rho^\g = 800\,{\rm kg/m^3}$.
Since we consider a horizontal domain with no gravity effect, we set ${\bf f} = {\bf 0}$ in $\Omega$ and no gravity term appears in the Darcy laws as in \eqref{eq_edp_hydromeca}. The domain is assumed to be clamped all over its boundary, i.e.~${\bf u} = {\bf 0}$ on $(0,T)\times\del\Omega$; for the flows, we impose a wetting saturation $s_m^\l = 1$ on the north side of the boundary $(0,T)\times((0,L)\times\{L\}) $, whereas the remaining part of the boundary is considered as impervious (${\bf q}_m^\alpha\cdot\bf n = 0, \alpha\in\{\g,\l\}$).
%The system is subject to an initial pressure $p_0^\l = 10^5\,{\rm Pa}$ and an initial saturation $s_{0,{\rm rt}}^\l=1$, ${\rm rt}\in \{m,f\}$, for the wetting phase, which in turn results in an initial pressure $p_{0,{\rm rt}}^\g = p_{0,{\rm rt}}^\l + P_{c,{\rm rt}}(s_{0,{\rm rt}}^\g) =  10^5\,{\rm Pa}$ for the wetting phase. 
The system is subject to the initial conditions $p_{0}^\g = p_{0}^\l= 10^5\,{\rm Pa}$, which in turn results in an initial  saturation $s_{0,{\rm rt}}^\g=0$,  ${\rm rt}\in \{m,f\}$. 
The final time is set to $T=1000\,{\rm days}=8.64{\cdot}10^7\,\rm{s}$. The system is excited by the following source term, representing injection of non-wetting fluid at the center of the fracture network:
$$h_f^\g(t,\x) = \frac{g(\x)}{\dsp \int_\Gamma g(\x)\, \d\sigma(\x)}\frac{V_{\rm por}}{T/5},\quad (t,\x)\in
(0,T)\times \Gamma,$$
where $V_{\rm por} =  \int_\Omega \phi^0_m(\x)\,\d\x$ is the initial porous volume and 
$g(\x) = e^{-\beta |(\x-\x_0)/L|^2},\ \x_0 = (\frac{L}{2},\frac{L}{2})$, with $\beta=1000$ and $|{\cdot}|$ the Euclidean norm.
The remaining source terms $h^\l_f$ and $h^\alpha_m$, $\alpha\in\{\l,\g\}$, are all set to zero.

%For the space discretization of \eqref{eq_edp_hydromeca}, we employ a Two-Point Flux Approximation (TPFA) finite volume scheme for the fluid flows, based on the $mf$-linear $m$-upwind model (cf.~\cite{gem.aghili}) and second-order finite elements ($\P_2$) for the displacement field in the matrix (see e.g.~\cite{daim.et.al,jeannin.et.al}), adding supplementary unknowns on the fracture faces to take into account discontinuities.
\begin{figure}
\centering
\includegraphics[scale=1.5]{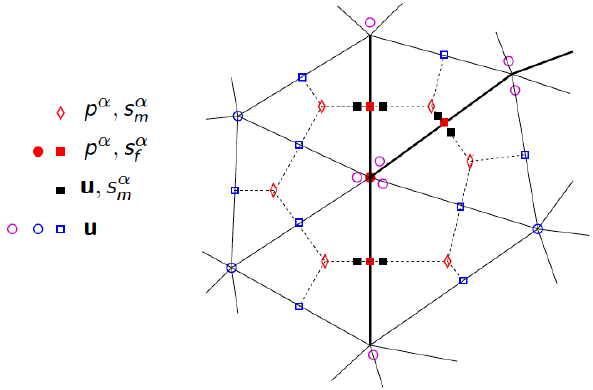}
\caption{Example of admissible triangular mesh with three fracture edges in bold. The dot lines joining each cell center to the center of each of its edges are assumed orthogonal to the edge. The discrete unknowns are presented for the two-phase flow and the mechanics. Note that the discontinuities of the saturations and of the displacement are captured at matrix fracture interfaces. The matrix and fracture saturations $s^\alpha_m$, $s^\alpha_f$ at matrix--fracture interfaces are computed using a single primary unknown parametrizing the capillary pressure graphs (cf.~\cite{gem.aghili}).  Note also that additional nodal unknowns are defined at intersections of at least three fractures.}
\label{incognite}
\end{figure}

As mentioned in the introduction, the GDM framework covers many possible schemes for both the flow and mechanical components of the model. For the flow, one would typically consider finite volume methods (or mixed finite elements), such as the low-order two- and multi-point flux approximations or hybrid mimetic mixed schemes \cite[Chapters 12, 13]{gdm}; even though high-order finite volume methods such as the hybrid high-order scheme \cite{hho-book} or non-conforming virtual elements \cite{Ayuso-de-Dios.Lipnikov.ea:16} also fit into the GDM \cite{Di-Pietro.Droniou.ea:18}, their usage in the current context seem less justified given the expected low regularity of the solution. Given that our simulations are done on triangular meshes, we opted to discretise the flow part using the cheap and robust Two-Point Flux Approximation (TPFA).
For the elasticity equation in~\eqref{eq_edp_hydromeca}, standard conforming finite element methods in standard displacement formulation as well as other more advanced techniques (such as stabilised nodal strain formulation or Hu-Washizu-based formulations) are known to fit in the GDM  \cite{DL14}. Our choice was on the second order $\P_2$ finite element in displacement formulation in the matrix~\cite{daim.et.al,jeannin.et.al}, adding supplementary unknowns on the fracture faces to account for the discontinuities. It provides a better accuracy than $\P_1$ finite element especially on the normal stresses at fracture tips and intersections.

The adaptation of the TPFA discretization to the hybrid-dimensional two-phase Darcy flow model follows \cite{gem.aghili} using  \mbox{$mf$-linear} $m$-upwind model for matrix-fracture interactions. However, unlike \cite{gem.aghili}, we consider here a centered approximation of the mobilities, and the scheme used in the test can therefore be written as a gradient scheme \eqref{GD_hydro}--\eqref{GD_meca}.
The~discrete unknowns for the phase pressures, the phase saturations and for the displacement field are shown in Figure~\ref{incognite}.
The computational domain $\Omega$ is decomposed using \emph{admissible} triangular meshes for the TPFA scheme (cf.~\cite[Section~3.1.2]{finite.vol} and the example Figure~\ref{incognite}). 
Let $n\in\N^\star$ denote the time step index. The time stepping is adaptive, defined as
$$\dtn = \min \{ {\varrho}\dtnmun , \Delta t^{\max} \},$$
where $\dtzero=0.025\,$days is the initial time step, $\Delta t^{\max} = 5$ days is the maximal time step (except for the finest mesh for which it is set to $\Delta t^{\max} = 2$ days), and ${\varrho} = 1.1$. At each time step, the flow unknowns
are computed by a Newton-Raphson algorithm. At each Newton-Raphson iteration, the Jacobian matrix is computed analytically and the linear system is solved using a GMRes iterative solver.
The time step is reduced by a factor 2 whenever the Newton-Raphson algorithm does not converge within 50 iterations, with the stopping criteria defined by the relative residual norm lower than $10^{-5}$ or a maximum normalized variation of the primary unknowns lower than $10^{-4}$. 
On the other hand, given the matrix and fracture equivalent pressures $p_m^E$ and $p_f^E$, the displacement field $\bu$ is computed using the direct solver MA48 (see~\cite{ma48}). Following \cite{BM2000,daim.et.al,jeannin.et.al,KTJ11,wheeler.mikelic,GKW16}, the coupling between the two-phase Darcy flow and the mechanical deformation is solved by means of a \emph{fixed-point} algorithm. This algorithm computes the matrix porosity and the fracture aperture, using discrete versions of the coupling laws \eqref{closure_laws}, at each time step and fixed-point iteration. The algorithm is summarized in the following scheme, where $k$ denotes the current fixed-point iteration and $n$ the current time step.%\\[2ex]
 \begin{tcolorbox}
	\textbf{Iterative coupling algorithm}
	\smallskip
	\\
	At each time step $n$, for $k=1,\dots$, until convergence, solve the following Darcy and mechanical subproblems: 
	\\ [-7.5pt]
%	\STATE
	\begin{enumerate}[label=(\roman*),leftmargin=15pt]
	\item \textbf{Compute} $p_{\rm rt}^{\a,n,k}$, $s_{\rm rt}^{\a,n,k}$, $\a \in \{\l,\g\}$, ${\rm rt}\in \{m,f\}$, solving
	the Darcy flow model using $d_f^{n,k-1}$ in the fracture conductivity and the following porosity and fracture aperture in the accumulation term:
	\\[1ex]
	 $\left\{
	\begin{aligned}
	\phi_m^{n,k} - \phi_m^{n-1} & = C_{r,m} (p_m^{E,n,k}-p_m^{E,n,k-1}) + b\,\div(\bu^{n,k-1}-\bu^{n-1})
	+ \frac{1}{M}(p_m^{E,n,k} - p_m^{E,n-1}),\\
	d_f^{n,k} - d_f^{n-1} & = C_{r,f} (p_f^{E,n,k}-p_f^{E,n,k-1}) - \jump{\bu^{n,k-1}-\bu^{n-1}}.
	\end{aligned}
	\right.
	$ \\[1ex]
	\label{itm:1}
	\item \textbf{Compute} the displacement field $\bu^{n,k}$ using the equivalent pressures $p_m^{E,n,k}$ and
	$p_f^{E,n,k}$ computed at step~\ref{itm:1}.
	\end{enumerate}
	\bigskip
	%\vspace{2cm}
	{\bf Initialization} \smallskip \\
	For given $n > 1$, set
	$$\left\{
	\begin{aligned}
	\frac{p_{\rm rt}^{E,n,0}-p_{\rm rt}^{E,n-1}}{\dtnmun} & = \frac{p_{\rm rt}^{E,n-1,0}-p_{\rm rt}^{E,n-2}}{\dtnmdeux},
	\ \ {\rm rt}\in\{m,f\},\\
	\frac{\bu^{n,0}-\bu^{n-1}}{\dtnmun} & = \frac{\bu^{n-1}-\bu^{n-2}}{\dtnmdeux};
	\end{aligned}
	\right.
	$$
	For $n=1$, set
	$$\left\{
	\begin{aligned}
	p_{\rm rt}^{E,-1} & = p_{\rm rt}^{E,0},
	\ \ {\rm rt}\in\{m,f\},\\
	\bu^{-1} &  = \bu^0.
	\end{aligned}
	\right.
	$$
\end{tcolorbox}
Here, $C_{r,m}$ and $C_{r,f}$ are positive relaxation parameters mimicking the rock compressibility (see e.g.~\cite{BM2000,daim.et.al,jeannin.et.al,KTJ11,wheeler.mikelic,GKW16}). For our numerical simulations, we choose $C_{r,m}=\frac{16 b^2}{2\mu + 2\lambda}$ (cf.~\cite{wheeler.mikelic}), and $C_{r,f}=\widetilde d_f C_{r,m}$ with $\widetilde d_f=10^{-3}\,{\rm m}$. The convergence of this fixed-point algorithm is achieved if the relative norm of the displacement field increment between two successive iterations is lower than $10^{-5}$. 

\subsection{Numerical convergence}

To verify the convergence of the method, we take into account six refined admissible triangular grids with $N = N_0$, $4N_0$, $16N_0$, $64N_0$, $256N_0$, $1024N_0$ cells, $N_0=224$.
All the numerical experiments of this subsection consider the centered approximation of the degenerate mobilities \eqref{def_mob}.
  The non-degenerate regularization consisting in replacing the mobilities with
  $$ {\mu^\alpha \eta_\rt^\alpha(s^\alpha) + \epsilon  \over \mu^\alpha(1 + \epsilon)} \quad \mbox{(for $\rt\in \{m,f\}$ and $\alpha\in \{\l,\g\}$)}$$
 has also been investigated, and found to exhibit significant differences, compared to the degenerate case, mainly on the matrix saturations and only for $\epsilon\geq 10^{-3}$; the differences are small for $\epsilon=10^{-4}$ and not observable for $\epsilon\le 10^{-5}$.

Figure~\ref{conv_u_sg} shows the convergence of the displacement field and gas saturation profiles along the line $y=55\,$m, intersecting the vertical fracture, computed at the final time for the first five grids. In addition, we consider a reference solution (denoted with the subscript ${\rm ref}$) computed on the finest (sixth) grid, made up by $1024N_0=229376$ cells, and used to showcase the time histories of the solution as well as to compute the time histories of the relative errors for each grid. Figure~\ref{df_time} shows the variation with respect to the curvilinear abscissa ($x$ or $y$, depending on the orientation) of the initial and final apertures for the fractures in the cross-shaped network, based on the reference solution. Note that the non-symmetry of the $y$ plot results from the output boundary condition on the north-side. At time $t=0$, the widths of both $x$- and $y$-oriented fractures coincide. Figure~\ref{Ufield} displays the final non-wetting matrix pressure and saturation computed on the fifth grid; as expected, the non-wetting fluid accumulates at the tips, flows through the fracture network and is attracted towards the upper open boundary.
Figure~\ref{ref_sols} showcases the time histories of the average of some relevant physical quantities computed based on the reference solution (the average of $a$ is denoted by $a^\star$). In particular, we notice the increase in width for the fracture network as a result of the gas injection, followed by a decrease after attaining a maximum due to an increasing gas matrix mobility in the neighborhood of the fractures. The same remark holds for the equivalent pressure $p^E_m$. The mean saturation in the matrix, as expected, grows linearly with time until the gas front reaches the upper boundary.
To illustrate the spatial convergence of the scheme, Figure \ref{conv_u_sg} plots on 4 meshes the cuts at $y=55$ m of both components of the displacement field and of the matrix non-wetting saturation. The non-monotone profile of the saturation cut results from the fronts propagating from the different tips of the fractures.
Figure~\ref{relative_errors} shows the convergence of the errors 
$$\left({\int_0^T (a_N^\star(t)-a^\star_{\rm ref}(t))^2dt \over \int_0^T a^\star_{\rm ref}(t)^2dt}\right)^{1\over 2},
$$  as a function of the mesh step for the first five meshes and $a=d_f, s^{\g}_m$, $p^E_m$. Computations are carried out, again, using averaged quantities ($a^\star_N$ denotes the spatial average of quantity $a$ computed using $N$ triangular elements). A similar convergence rate is observed for all quantities. 
Finally, we give an insight into the performance of our method in Table~\ref{perfs}, where
%\newpage
%\vspace{-.1cm}
\begin{itemize}
\item NbCells is the number of cells of the mesh,
\item N$_{\Delta t}$ is the number of successful time steps,
\item N$_{\text{Newton}}$ is the total number of Newton-Raphson iterations,
\item N$_{\text{GMRes}}$ is the total number of GMRes iterations,
\item N$_{\text{FixedPoint}}$ is the total number of fixed point iterations,
\item CPU (s) is the CPU time of the simulation in seconds.   
\end{itemize}
The iterative coupling algorithm exhibits good robustness with respect to the mesh size, and has a linear convergence behaviour as proved in \cite{GKW16} in the linear case.  As expected for incompressible fluids in fractured porous media \cite{GKW16}, the convergence rate is however very sensitive to small initial time steps. This issue is shown in \cite{BBDMEcmor20} to be efficiently solved by using a Newton Krylov acceleration of the fixed-point algorithm.
\begin{table}
\centering
{ 
  {
    \begin{tabular}{|c|c|c|c|c|c|}
      \hline
      NbCells & {N}$_{\Delta t}$  &  N$_{\text{Newton}}$  & N$_{\text{GMRes}}$ & N$_{\text{FixedPoint}}$ & CPU (s)\\ \hline
     $N_0$ &  246    & 11902  & 81037  & 11163  & 23  \\ \hline      
     4$N_0$ &  246    & 4685  & 47352  & 4234  & 35 \\ \hline
     16$N_0$ &  246    & 4626  & 53870 & 4138 & 130\\ \hline
     64$N_0$ & 246    & 4713  & 65293 & 4063 & 500\\ \hline
     256$N_0$ & 246   & 4951  & 88106 & 4062 & 2600\\ \hline
     1024$N_0$ & 537   & 5788  & 125495  & 4147 & 14500\\ \hline
    \end{tabular}
    }
  }
    \caption{Performance of the method with the centered scheme in terms of the number of mesh elements, the number of successful time steps, the total number of Newton-Raphson iterations, the total number of GMRes iterations, the total number of fixed-point iterations, and the CPU time.}
      \label{perfs}
\end{table}
%\vspace{-1cm}
%%%%%%%%%
%%%%%%%%%
%%%%%%%%%
%%%%%%%%%%%%%%%%%%%%%%%%%%%%%%%
%%%%%%%%%%%%%%%%%%%%%%%%%%%%%%%
%%%%%%%%%%%%%%%%%
%%%%
\begin{figure}
\centering
\subfloat[$p_m^\g(T;x,y)$]{\includegraphics[scale=.15]{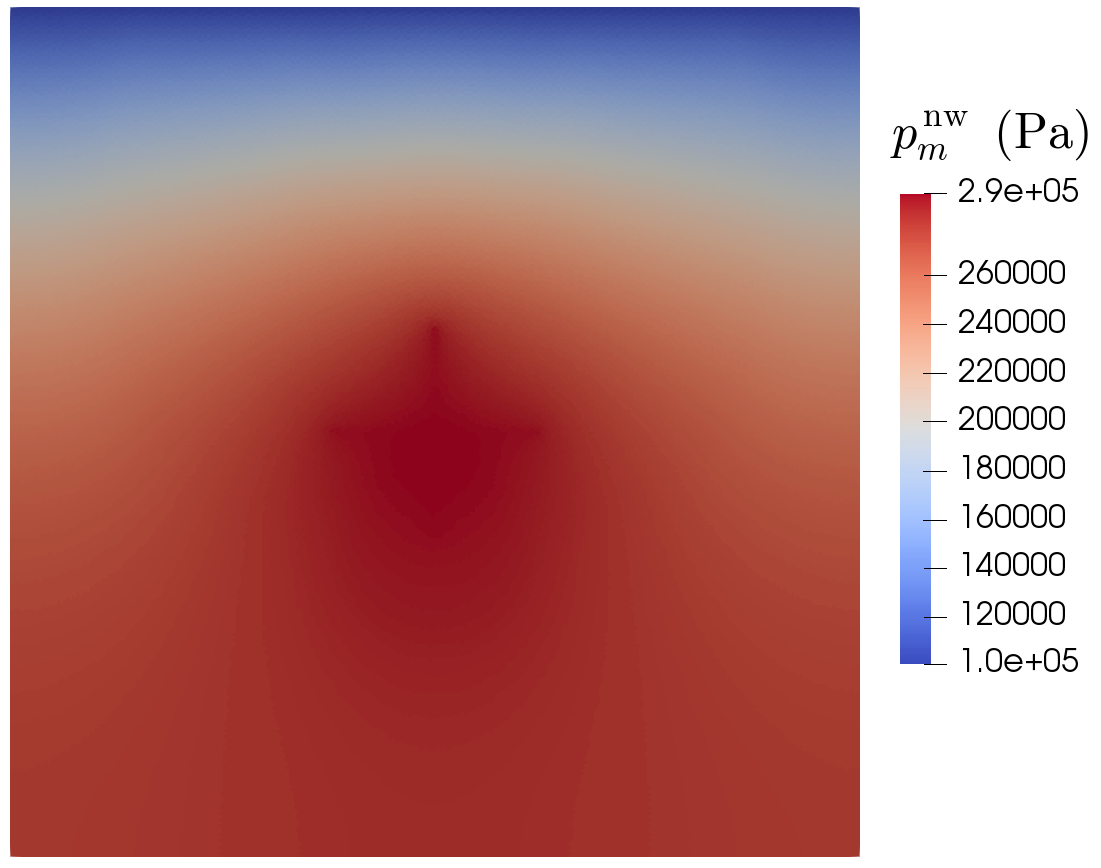}} \  \
\subfloat[$s_m^\g(T;x,y)$]
{{{\includegraphics[scale=.15]{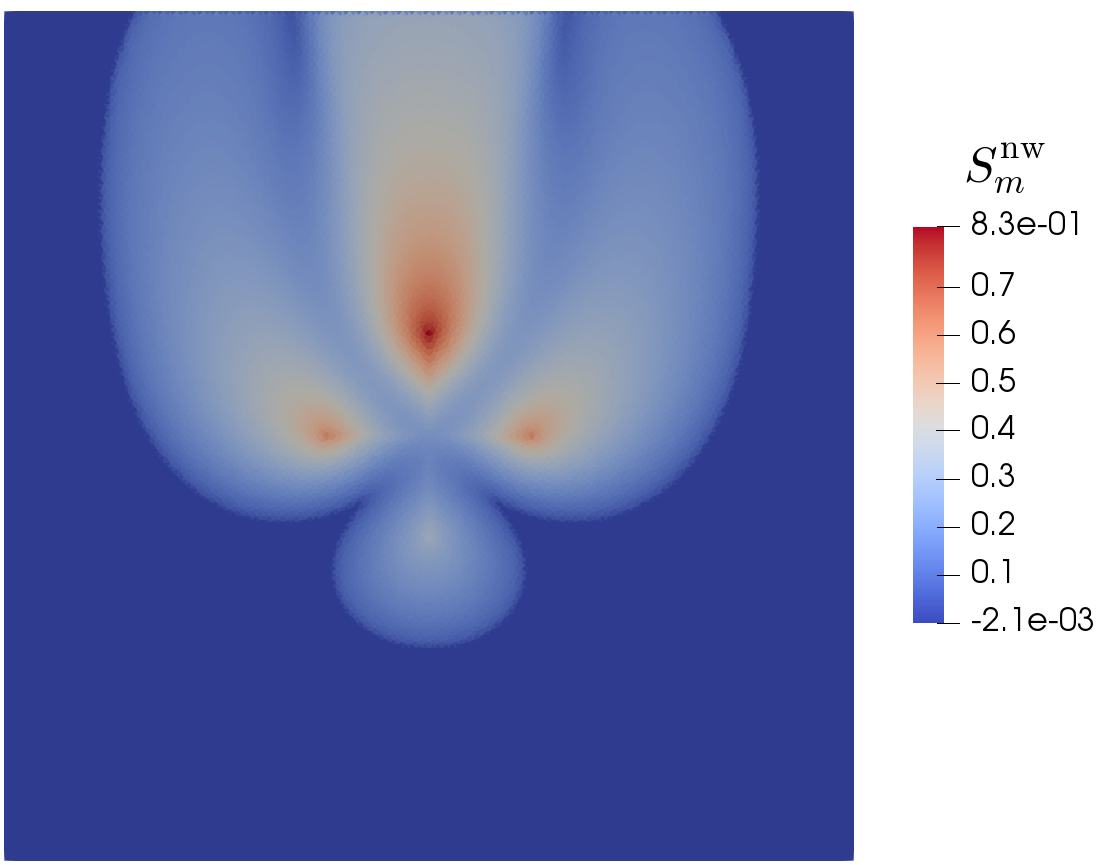}}}}\\
\subfloat[$u_1(T;x,y)$]{{\ \ \includegraphics[scale=.15]{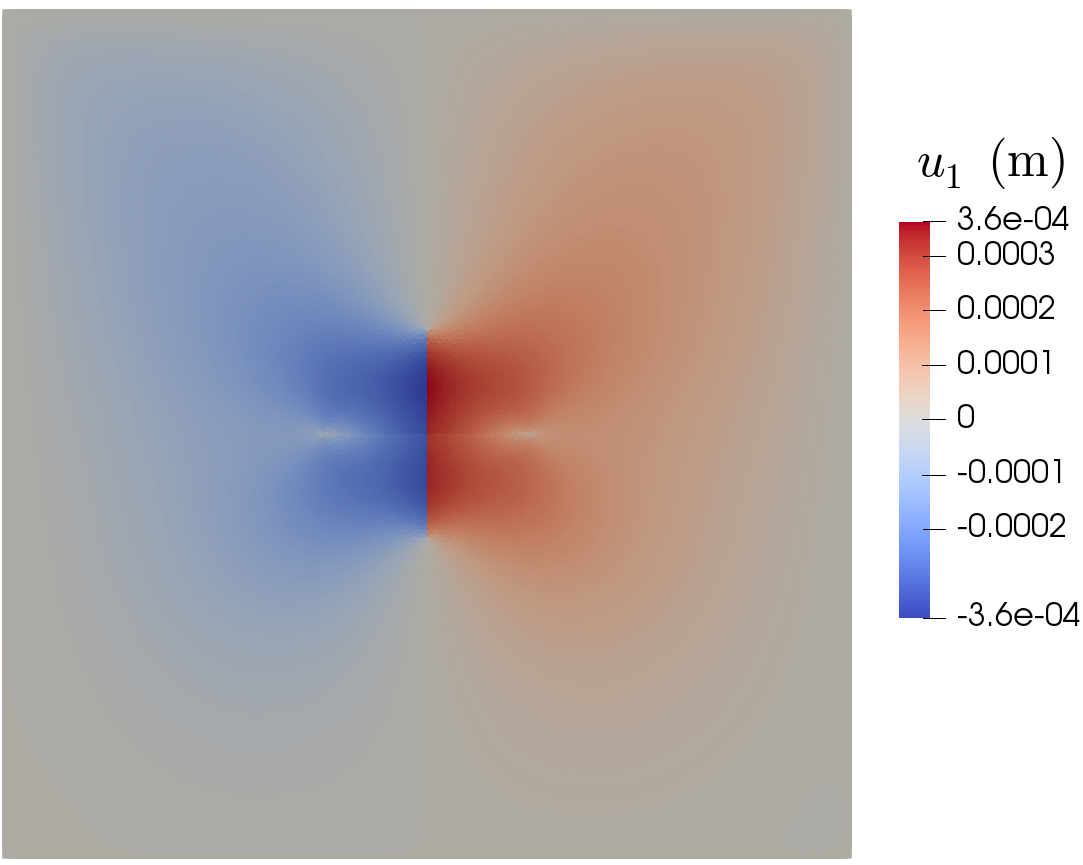}}} \ \ 
\subfloat[$u_2(T;x,y)$]{\includegraphics[scale=.15]{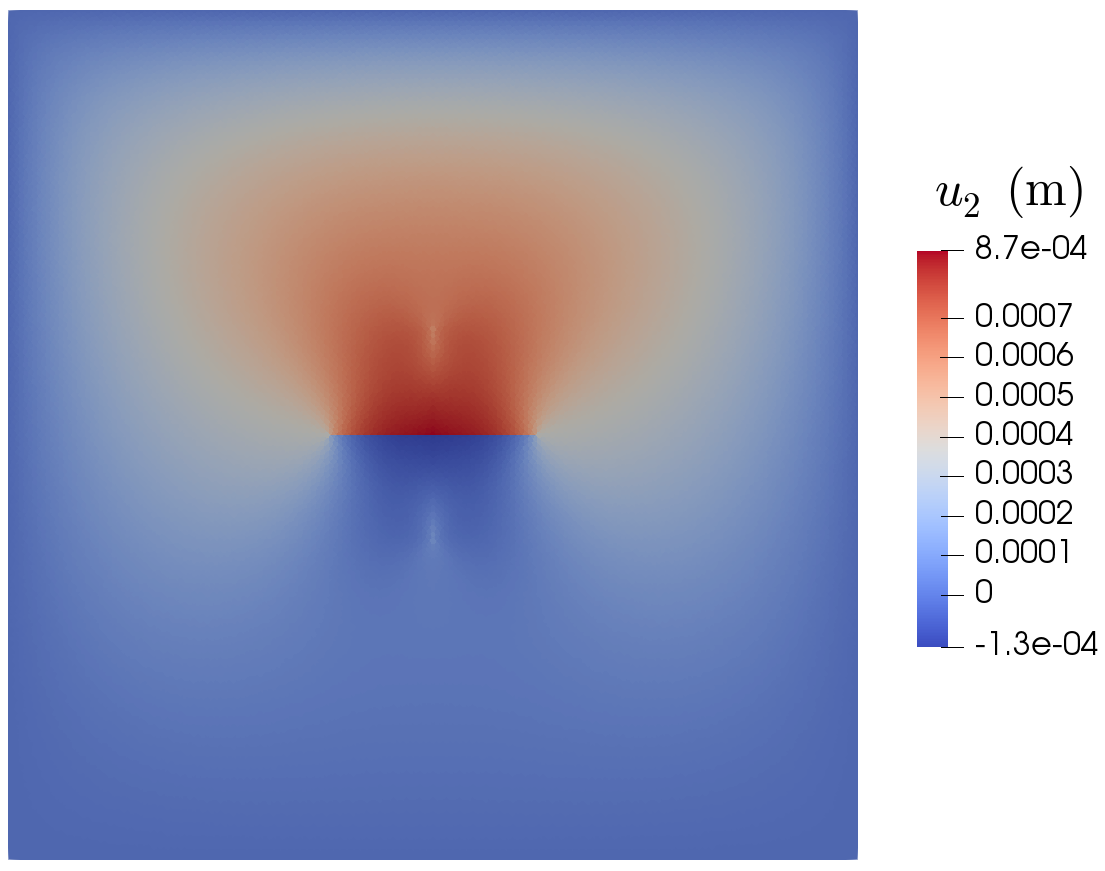}}
\caption{Final non-wetting matrix pressure and saturation (a)-(b), and final (c)-(d) displacement field on the fifth mesh of size 57344 cells.}
\label{Ufield}
\end{figure}
%%%%%
\begin{figure}
\hspace*{-2cm} 
\centering
\includegraphics[scale=.75]{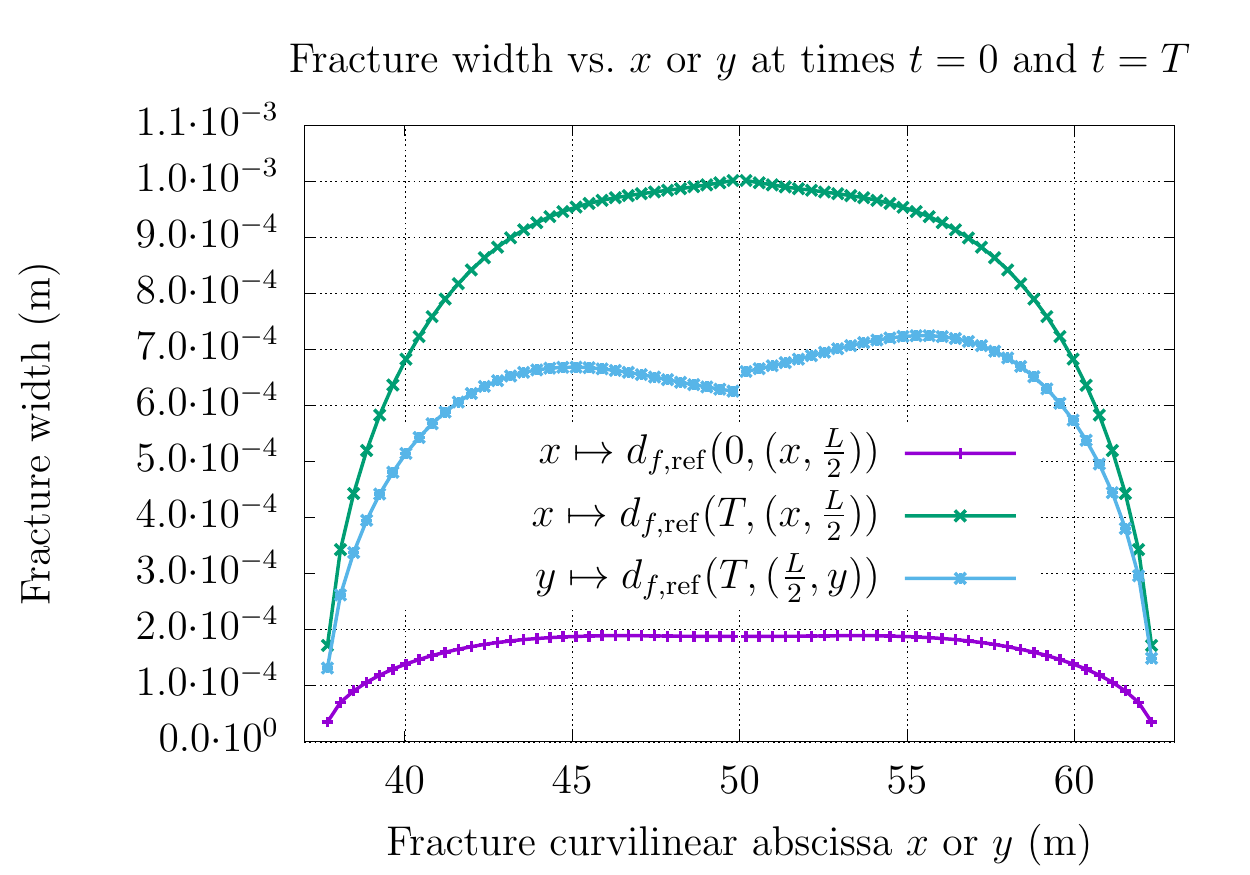}
\caption{Initial and final widths of the $x$- and $y$-oriented fractures vs.~corresponding curvilinear abscissae, computed using the finest grid (reference solution). The initial width for both the $x$- and $y$-oriented fractures is the same.}
\label{df_time}
\end{figure}
%%%%
\begin{figure}
%\captionsetup[subfigure]{labelformat=empty}
\centering
\subfloat[]{\includegraphics[keepaspectratio=true,scale=.165]{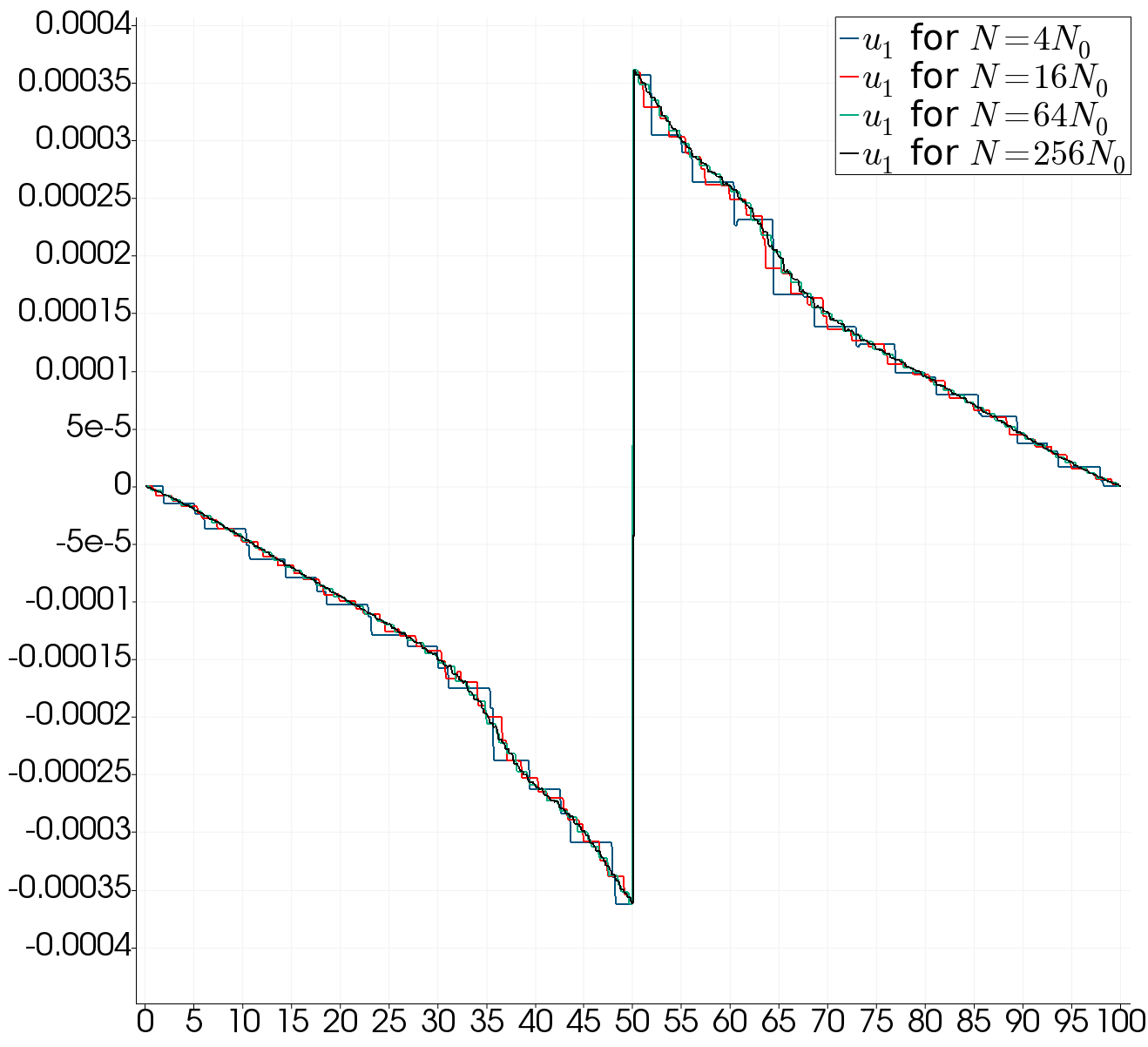}}
%\hspace{.01cm}
\subfloat[]{\includegraphics[keepaspectratio=true,scale=.165]{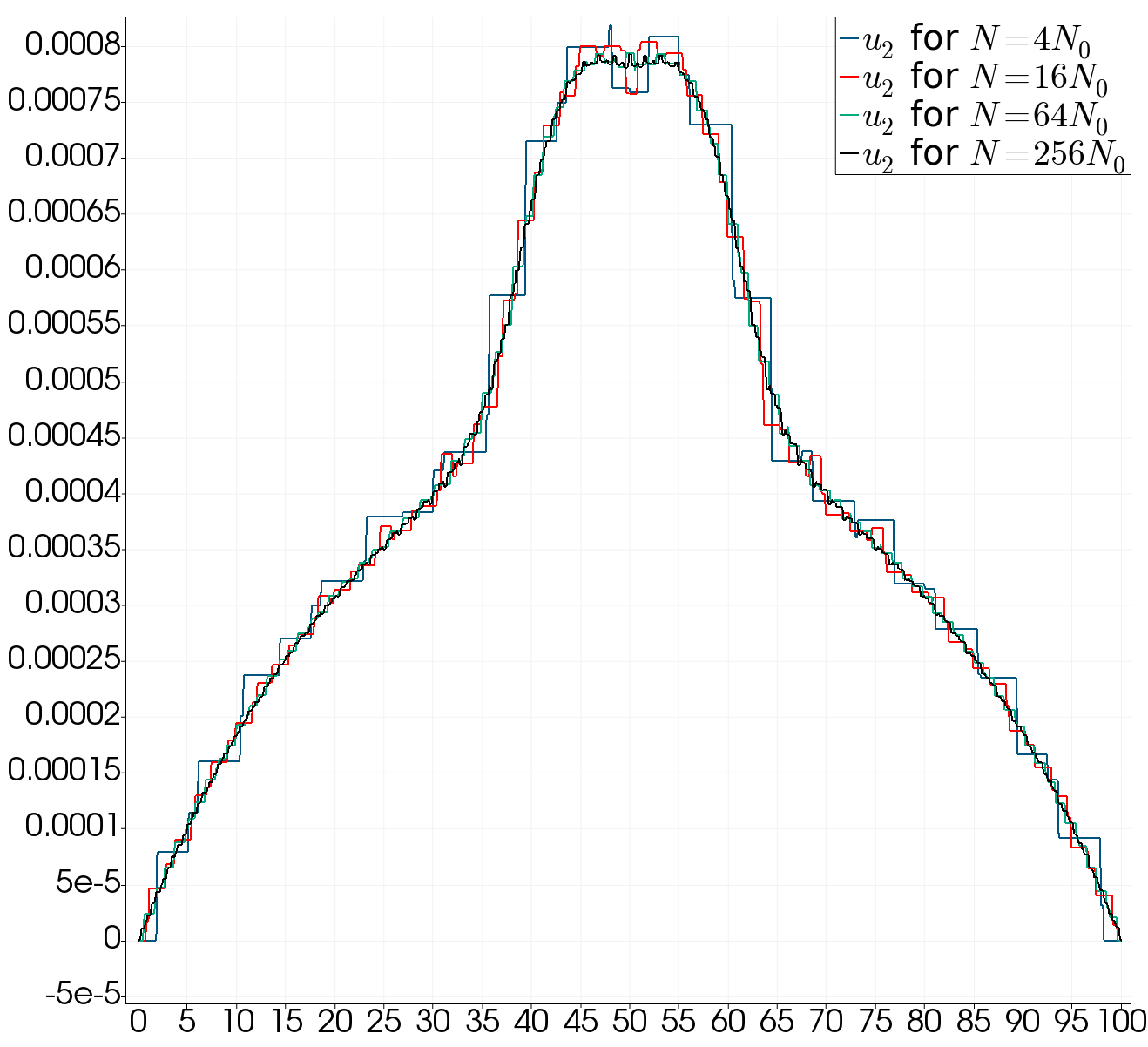}}\\
%%%%
\subfloat[]{\includegraphics[keepaspectratio=true,scale=.2]{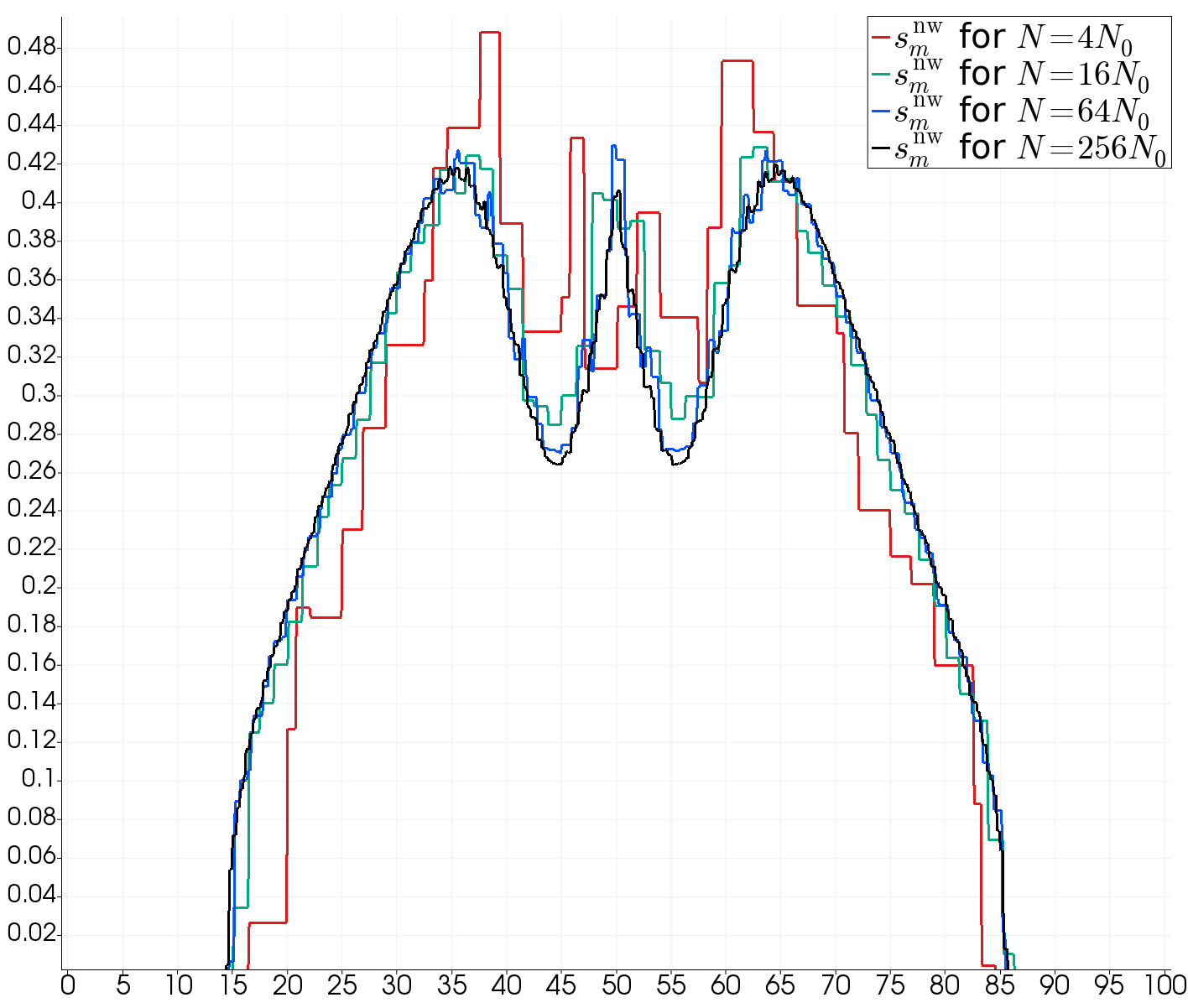}}
%\hspace{.01cm}
\caption{Convergence of the profiles of the displacement field components (m) $u_1$ (a), $u_2$ (b), and gas saturation $s_m^{\rm nw}$ (c) at the final time, along the line $y=55$\,m intersecting the vertical fracture, for four grids, with $N$ triangular elements, and $N_0=224$.}
\label{conv_u_sg}
\end{figure}

%%%%%%%%%%%%%%%%%%%%%%%%%%%%
\begin{figure}
%\captionsetup[subfigure]{labelformat=empty}
\centering
%\subfloat[]{\includegraphics[scale=.625]{phim_ref}}
%\hspace{.01cm}
\subfloat[]{\includegraphics[scale=.55]{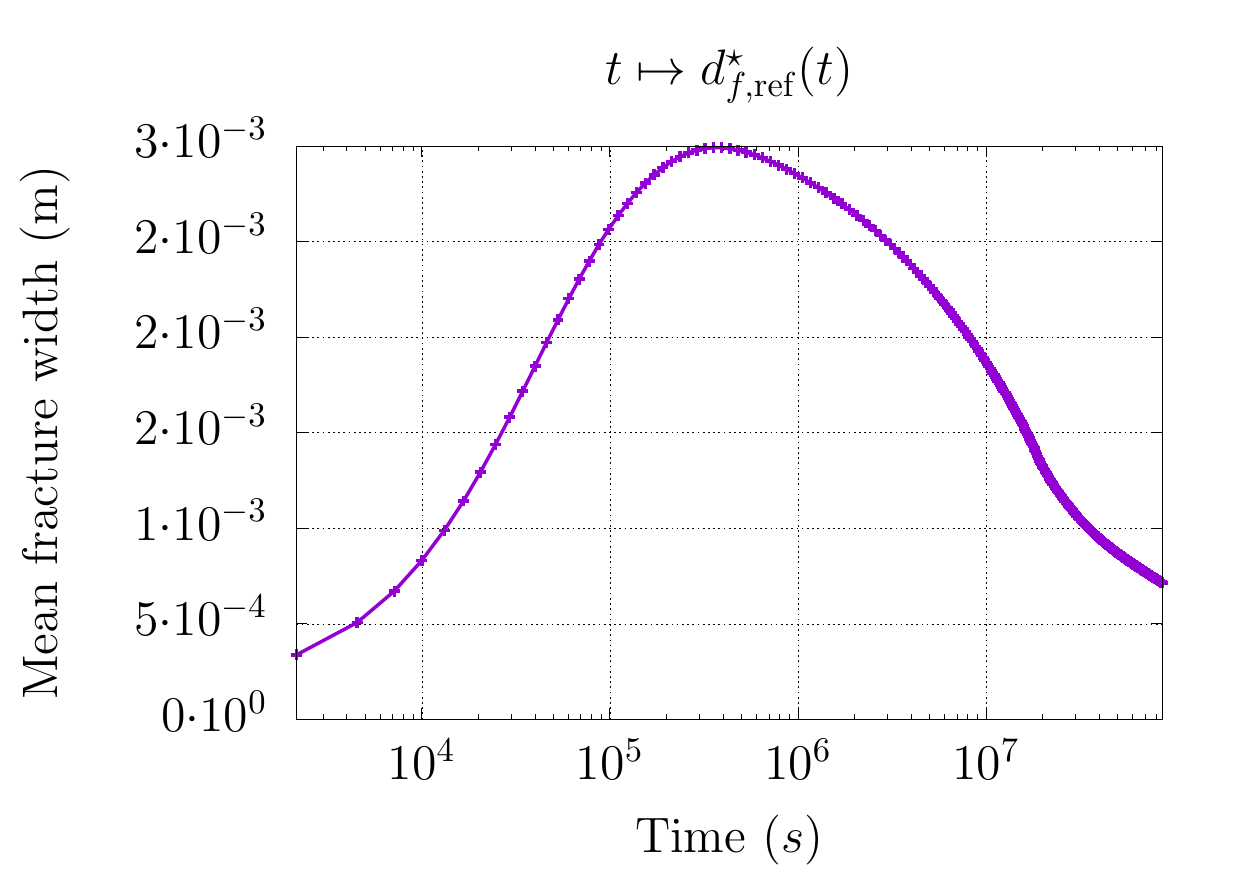}}\
%%%%
\subfloat[]{\includegraphics[scale=.55]{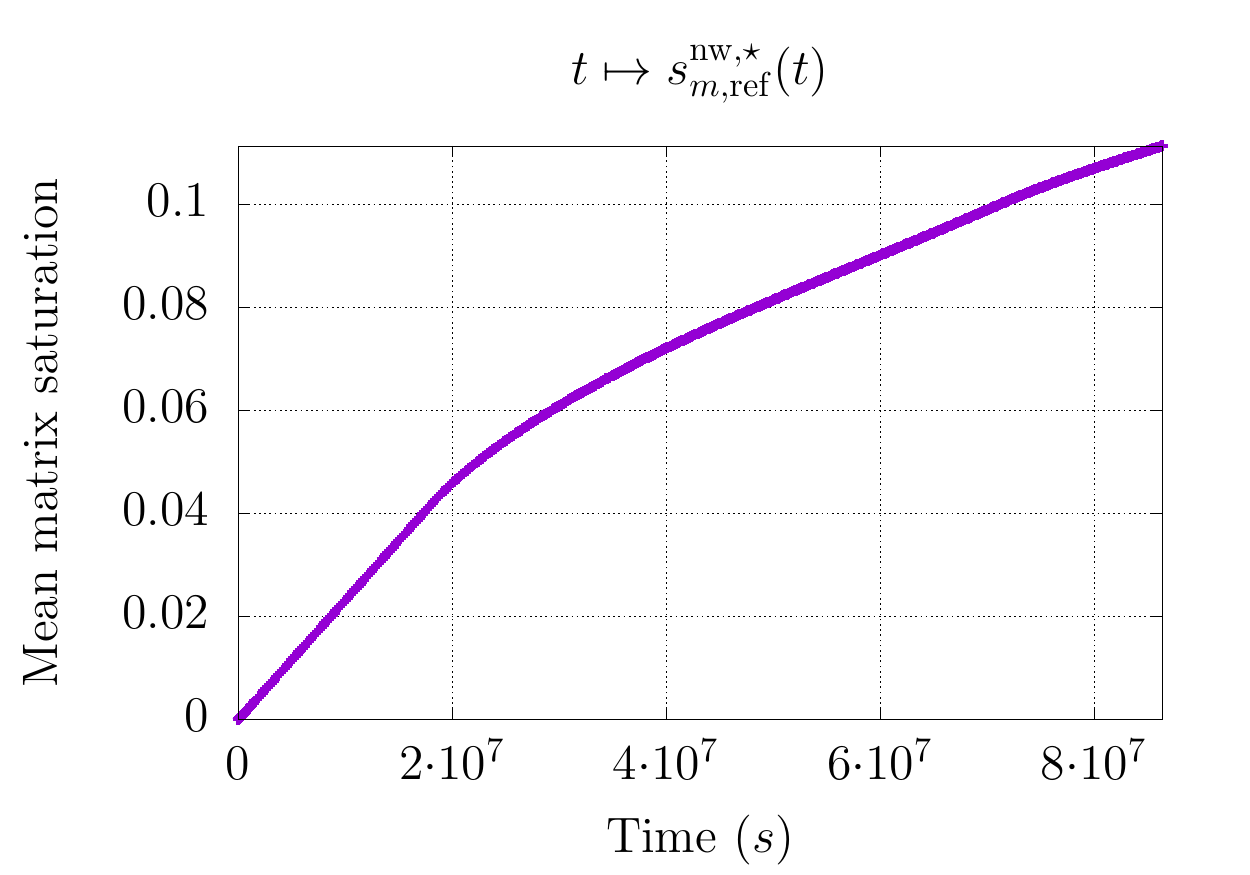}}\
%\hspace{.01cm}
%\subfloat[]{\includegraphics[scale=.7]{sf_ref}}
%%%%
\subfloat[]{\includegraphics[scale=.55]{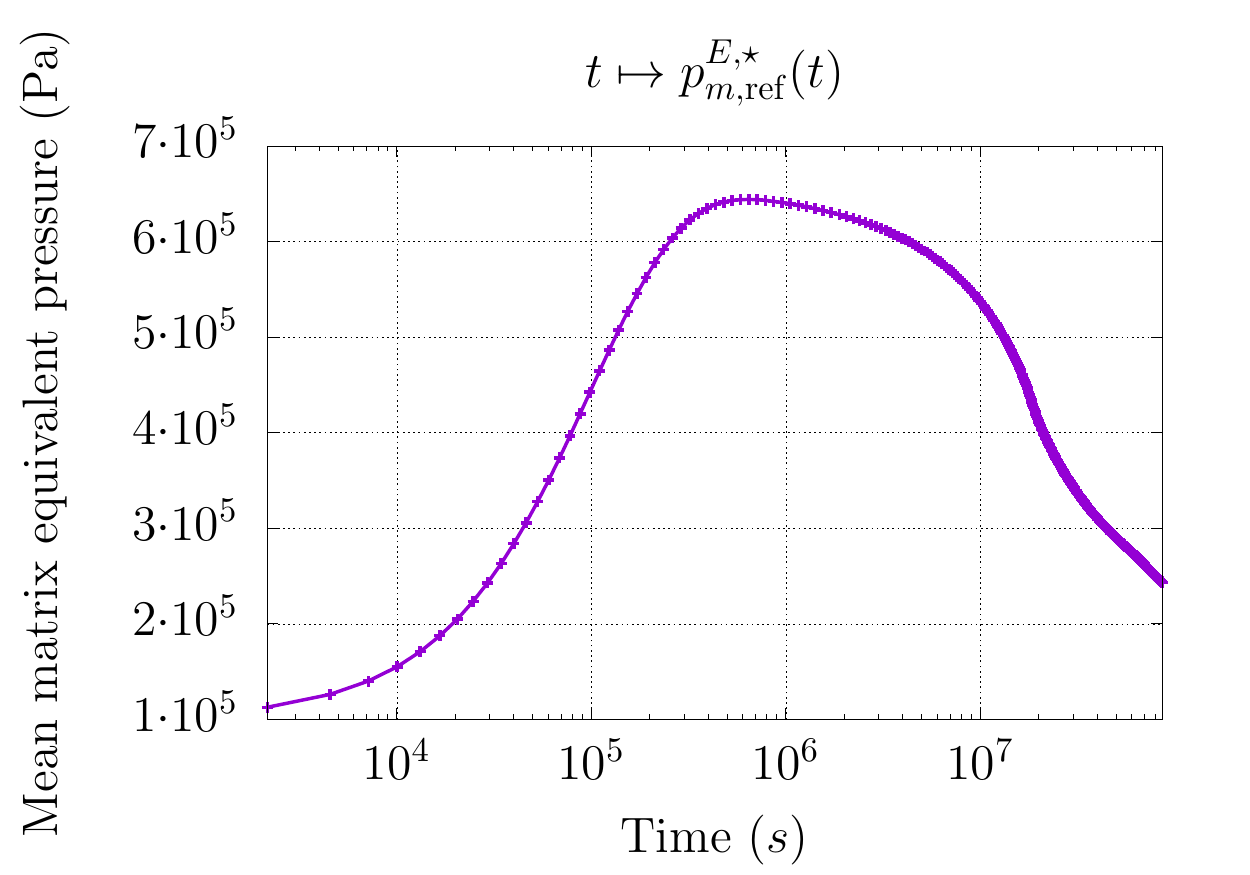}}
%\hspace{.01cm}
%\subfloat[]{\includegraphics[scale=.625]{pEf_ref}}
\caption{Time histories of the average of some physical quantities based on the reference solution ($a^\star$ denotes the spatial average of $a$).}
\label{ref_sols}
\end{figure}
%%%%%%%%%%%%%%%%%%%%%%%%%%%%%
\begin{figure}
\hspace*{-.5cm}
%\captionsetup[subfigure]{labelformat=empty}
\centering
\includegraphics[scale=.75]{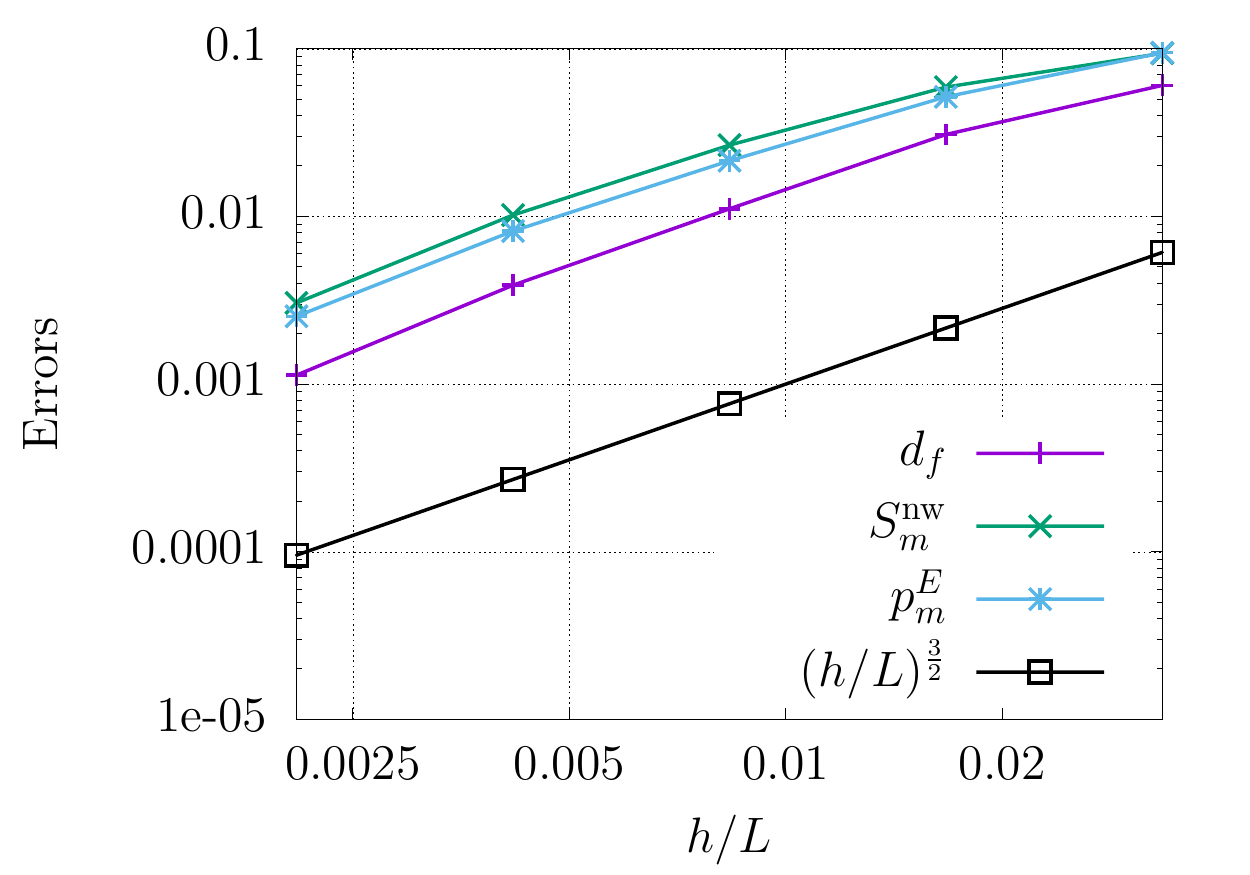}
%\subfloat[]{\includegraphics[scale=.65]{err_phim}}
%\hspace{.01cm}
%\subfloat[]{\includegraphics[scale=.6]{err_df}} \ \
%%%%
%\subfloat[]{\includegraphics[scale=.6]{err_sm}}\\
%\hspace{.01cm}
%\subfloat[]{\includegraphics[scale=.75]{err_sf}}\\
%%%%
%\hspace*{-.5cm}
%\subfloat[]{\includegraphics[scale=.6]{err_pEm}}
%\hspace{.01cm}
%\subfloat[]{\includegraphics[scale=.65]{err_pEf}}
\caption{Relative $L^2$ norm of the error as a function of the mesh step computed on the first five meshes for the time histories of the mean quantities $d^\star_f$, $s^{\g,\star}_m$ and $p^{E,\star}_m$ with respect to the corresponding reference time histories (cf.~Figure~\ref{ref_sols}).}

\label{relative_errors}
\end{figure}

\subsection{Comparison between the upwind and centered schemes}

In this subsection, the solutions obtained by the upwind approximation of the mobilities \cite{gem.aghili} are compared to the solutions obtained in the previous subsection based on the centered approximation.
  Small differences can be noticed on the upper part of the matrix saturation at final time in Figures \ref{fig_upwind_centered_schemes_mesh6} and \ref{fig_upwind_centered_schemes_mesh7}  due to the larger numerical diffusion of the upwind scheme. They clearly reduce with the mesh refinement as can be observed from the comparison between the line cut plots in these figures. The cuts at $y=80$ m in the right Figure \ref{fig_upwind_centered_cut80_mesh67} for the fifth and sixth meshes confirm that, compared to the upwind scheme, the centered scheme converges more quickly with the mesh refinement. This is also checked on the left Figure \ref{fig_upwind_centered_cut80_mesh67} which exhibits the convergence of the error for the average quantities $d^\star_f$, $s^{\g,\star}_m$ and $p^{E,\star}_m$ with respect to the reference mesh solutions. It shows that the upwind scheme is slightly more accurate on the two coarsest meshes but that the centered scheme has a better convergence rate.  It has been checked that the displacement field and fracture aperture for both schemes exhibit very little differences on the finest meshes.
  Table \ref{perfs_upwind} shows the numerical performances of the upwind scheme. Comparing with the results obtained for the centered scheme (Table \ref{perfs}), the upwind scheme offers a slighly more efficient nonlinear convergence. It was also checked that the upwind scheme can accommodate larger time steps with successful nonlinear convergence than the centered scheme.

We note that a GDM discretisation of \eqref{eq_edp_hydromeca}--\eqref{closure_laws} with upwind approximation of the mobilities as in \cite{gem.aghili} for the TPFA scheme requires a scheme-dependent definition of the Darcy fluxes, and hence loses the generality of the GDM. Once such definition is provided, the convergence analysis would entail a specific treatment which needs to be investigated due to the nonlinearity introduced by the upwinding; however, this analysis could benefit from the tools developed here, such as the relative compactness results.

\begin{figure}
\centering
\includegraphics[scale=.5]{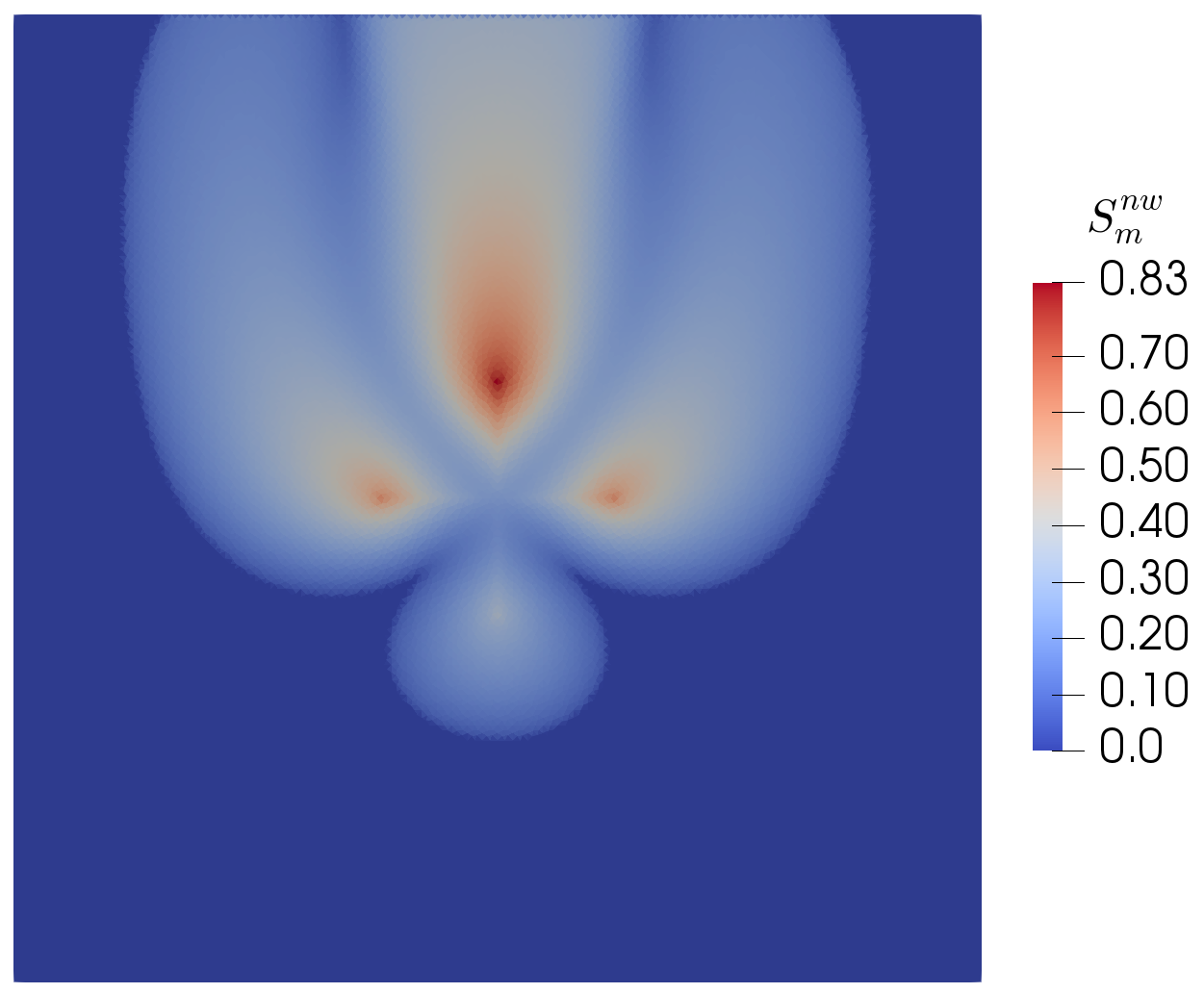}          
\includegraphics[scale=.5]{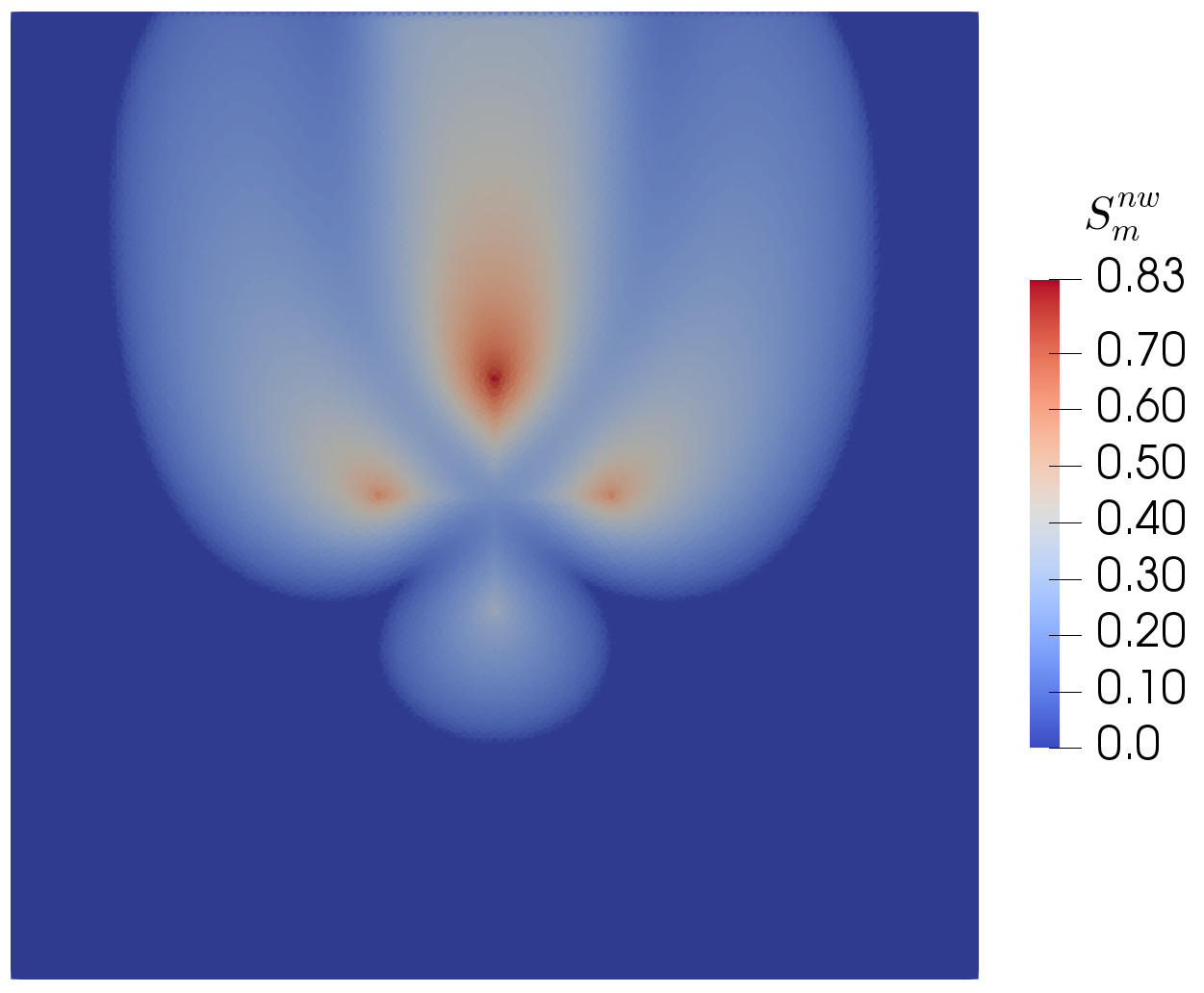}\\
\includegraphics[scale=0.7]{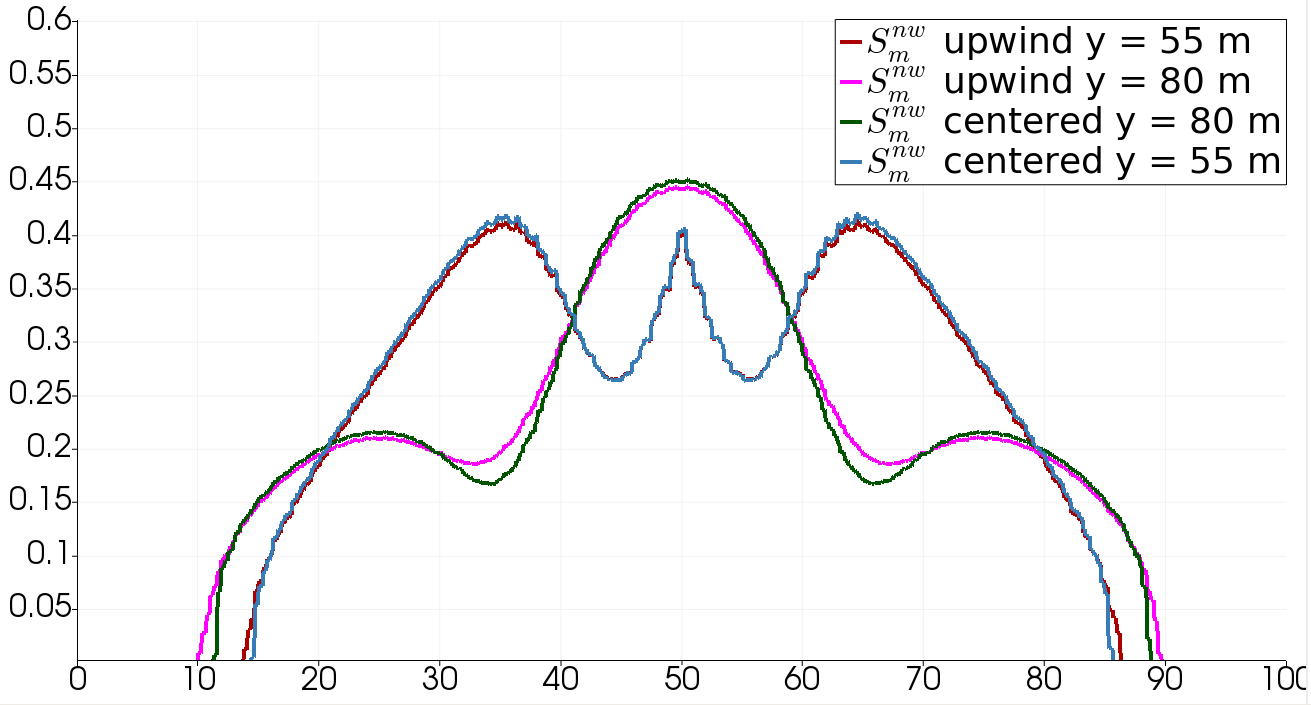}
\caption{On the fifth mesh with 256$N_0$ cells, final non-wetting phase matrix saturations for the centered scheme (top left) and upwind scheme (top right), and line cuts of both solutions at $y=55$ m and $y=80$ m (bottom).}
\label{fig_upwind_centered_schemes_mesh6}
\end{figure}

\begin{figure}
\centering
\includegraphics[scale=.5]{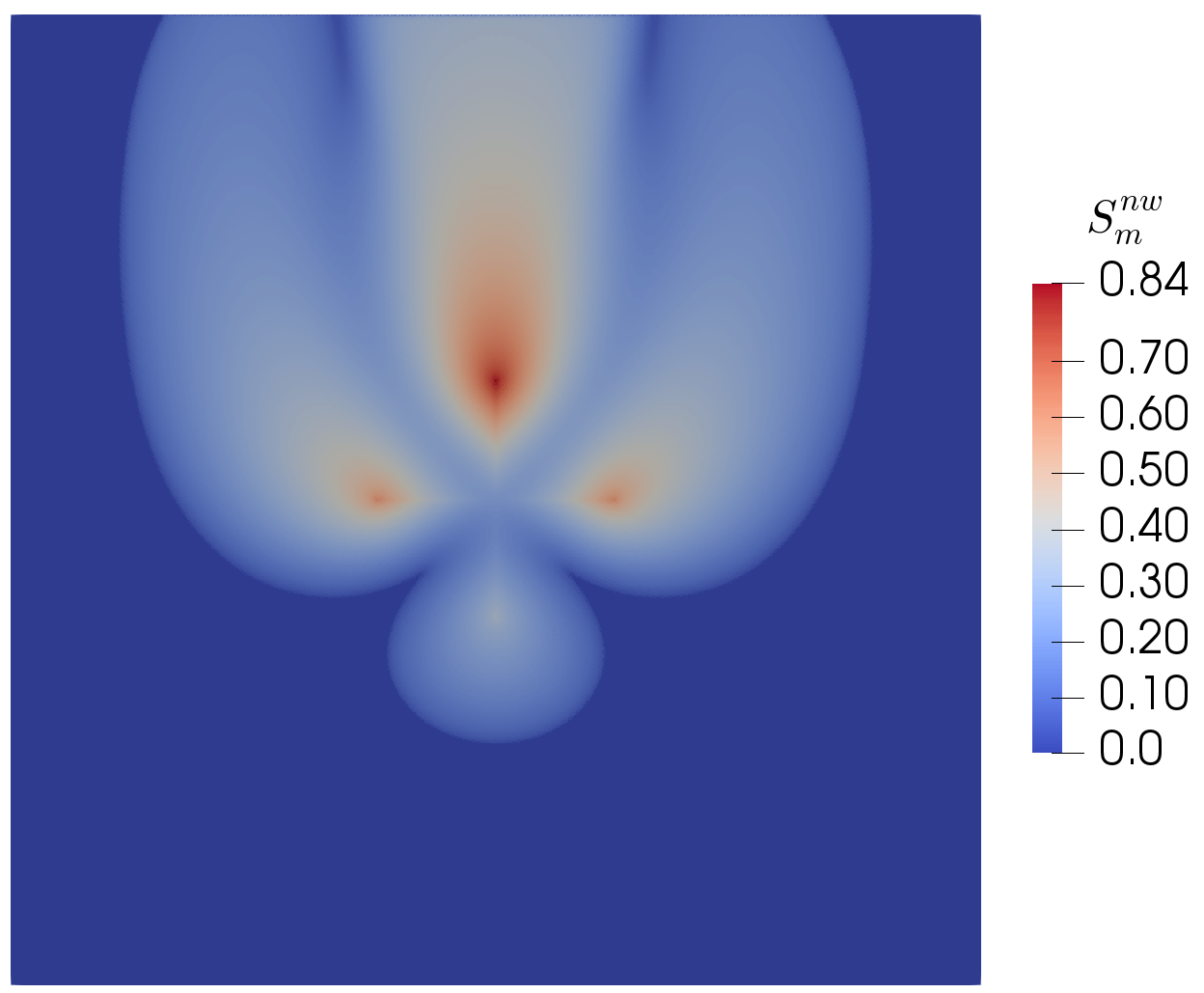}          
\includegraphics[scale=.5]{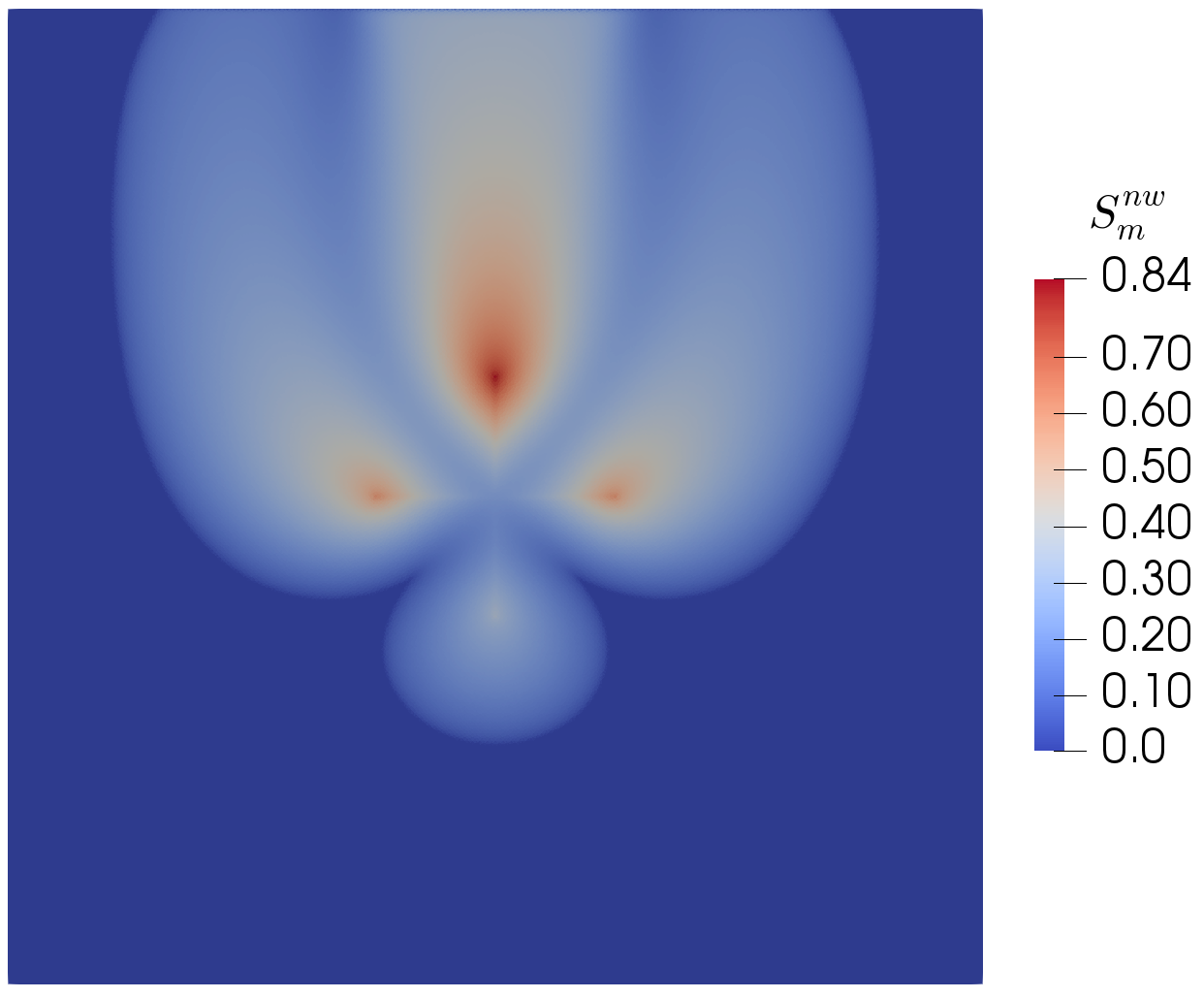}\\
\includegraphics[scale=0.7]{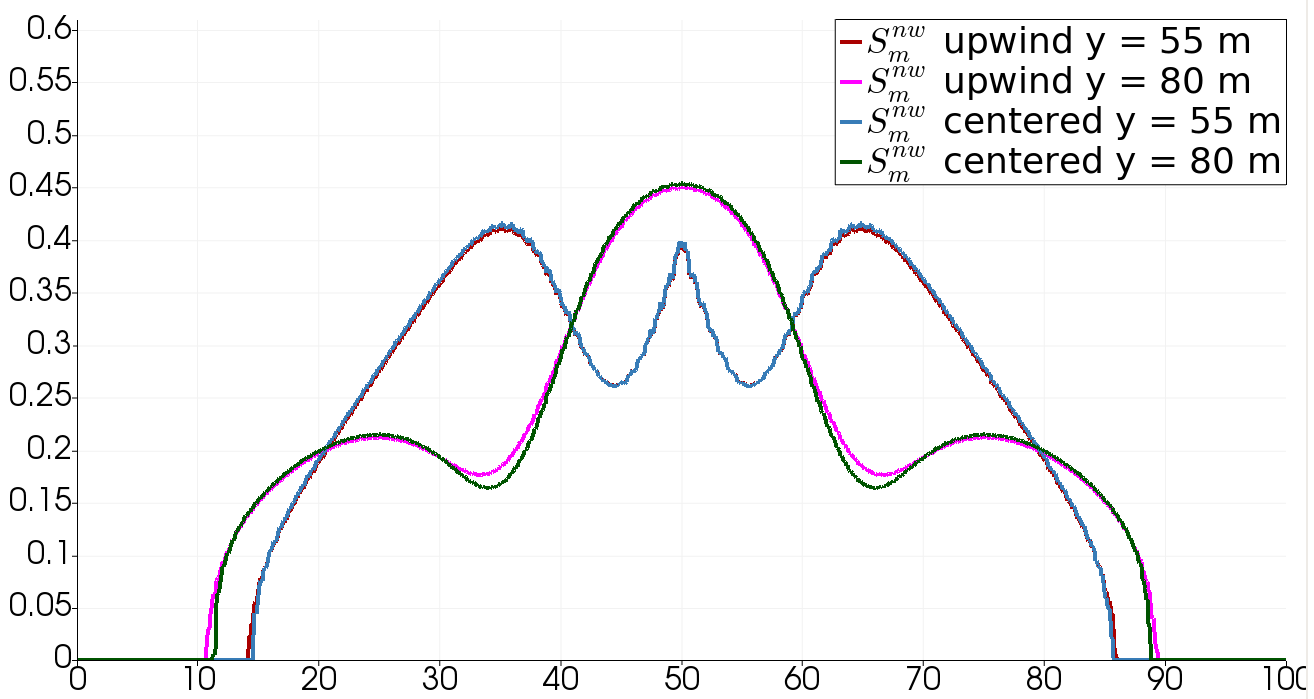}
\caption{On the finest mesh with 1024$N_0$ cells, final non-wetting phase matrix saturations for the centered scheme (top left) and upwind scheme (top right), and line cuts of both solutions at $y=55$ m and $y=80$ m (bottom).}
\label{fig_upwind_centered_schemes_mesh7}
\end{figure}

\begin{figure}
\centering
\includegraphics[scale=.15]{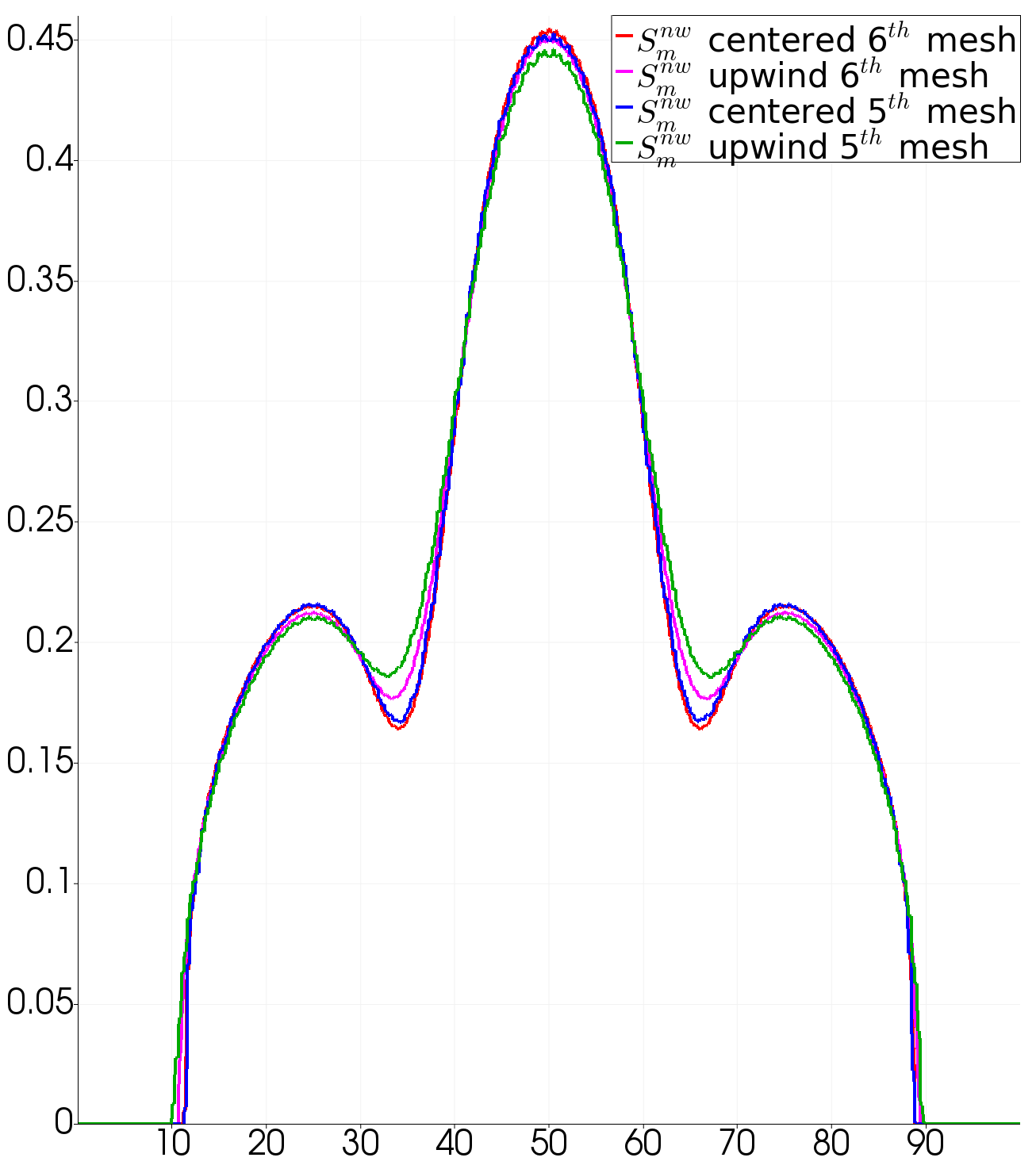}\\
\includegraphics[scale=.55]{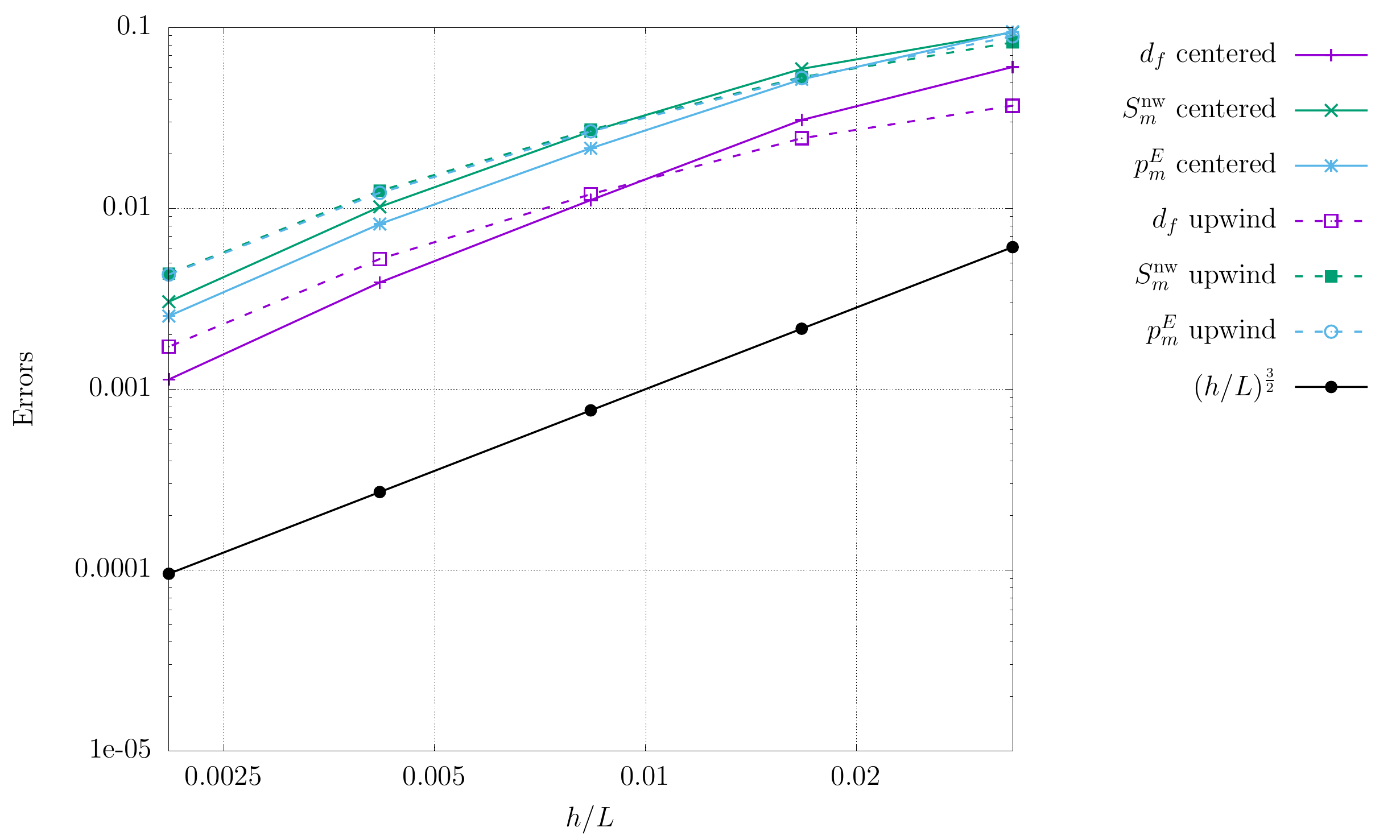}
\caption{(Top): on the fifth (256$N_0$ cells) and sixth (1024 $N_0$ cells) meshes, cuts at $y=80$ m of the non-wetting phase matrix saturations for both the centered  and upwind schemes.
(Bottom): for both the centered and upwind schemes, 
relative $L^2$ norm of the error as a function of the mesh step computed on the first five meshes for the time histories of the mean quantities $d^\star_f$, $s^{\g,\star}_m$ and $p^{E,\star}_m$ with respect to the corresponding reference time histories. 
}
\label{fig_upwind_centered_cut80_mesh67}
\end{figure}

\begin{table}
\centering
{ 
  {
    \begin{tabular}{|c|c|c|c|c|c|}
      \hline
      NbCells & {N}$_{\Delta t}$  &  N$_{\text{Newton}}$  & N$_{\text{GMRes}}$ & N$_{\text{FixedPoint}}$ & CPU (s)\\ \hline
     256$N_0$ & 246   & 4809  & 82085 & 4054 & 2250\\ \hline
     1024$N_0$ & 537   & 5486  & 114763  & 4136 & 13600\\ \hline
    \end{tabular}
    }
  }
    \caption{Performance of the method with the upwind scheme in terms of the number of mesh elements, the number of successful time steps, the total number of Newton-Raphson iterations, the total number of GMRes iterations, the total number of fixed-point iterations, and the CPU time.}
      \label{perfs_upwind}
\end{table}

\section{Conclusions}

We developed, in the framework of the gradient discretization method, the numerical analysis  of a two-phase flow model in deformable and fractured porous media. The model considers a linear elastic mechanical model with open fractures coupled with an hybrid-dimensional two-phase Darcy flow assuming continuity of each phase pressure across the fractures. The model accounts for a general network of planar fractures including immersed, non-immersed fractures and fracture intersections, and considers different rock types in the matrix and fracture network domains. 

It is assumed, for the convergence analysis, that the porosity remains bounded below by a strictly positive constant, and that the fracture aperture remains larger than a fixed non-negative continuous function vanishing only at the tips of the fracture network. These assumptions stem from the limitations of the continuous model itself. In addition, the mobility functions are assumed to be bounded below by strictly positive constants. However, unlike previous works, the fracture conductivity $d_f^3/12$ was not frozen and the complete non-linear coupling between the flow and mechanics equations was considered.

Assuming that the gradient discretization meet generic coercivity, consistency, limit-conformity and compactness properties, we proved the weak convergence of the phase pressures and displacement field to a weak continuous solution, as well as the strong convergence of the fracture aperture and of the matrix and fracture saturations.
Numerical experiments carried out for a cross-shaped fracture network immersed in a two-dimensional porous medium and using a TPFA finite volume scheme for the flow combined with a $\P_2$ finite element method for the mechanics, confirmed the numerical convergence of the scheme. 

\appendix
\section{Appendix}

\subsection{Appendix 1}
\label{appendix.a1}

\begin{proposition}\label{prop_grid_projection}
  Let $X\subset \mathbb{R}^d$ be bounded, $\delta > 0$ and let $\left( A^\delta_m \right)_{m\in M_\delta}$ be a covering of $X$ in disjoint cubes of length $\delta$.
 Let $R^\delta: L^2(\R^d) \to L^2(X)$ be such that, for any $v\in L^2(\R^d)$,
$$
\left( R^\delta v\right)|_{A^\delta_m\cap X} = \frac{1}{\delta^d} \int_{A^\delta_m} v(\x) ~\d\x\quad\forall m\in M_\delta,
$$
Then, we have 
$$
\|R^\delta v - v\|_{L^2(X)} \leq 2^{d/2} \sup_{|z|\leq \delta}\|v(\cdot+z) - v\|_{L^2(X)}. 
$$
\end{proposition}

\begin{proof}
The proof can be found in \cite[p. 756]{droniou.eymard2016}. Note that the assumption, in this reference, that $v$ is zero outside $X$ is actually not useful.
\end{proof}

\begin{lemma}\label{lemme2}
Let $X\subset \mathbb{R}^d$ be bounded, and $U$ be an open subset of $\R^d$ such that $\{\x\in\R^d\,:\,{\rm dist}(\x,X)<\delta_0\} \subset U$ for a given $\delta_0 >0$, where the distance is considered for the supremum norm in $\R^d$.  Let $\left( w_k \right)_{k\in \N}$ be a bounded sequence in $L^\infty(0,T; L^2(U))$ that converges uniformly in time and weakly in $L^2(U)$ to $w \in L^\infty(0,T; L^2( U ))$. Let $p\in [1, +\infty]$ and let us define 
$$
T(\delta) = \sup_k \Big\| \sup_{|z|\leq \delta} \| w_k(\cdot,\cdot+z) - w_k(\cdot,\cdot) \|_{L^2(X)}  \Big\|_{L^p(0,T)}.
$$
If $\lim_{\delta\to 0}T(\delta) = 0$, then the sequence $\left( w_k \right)_{k\in \N}$ converges to $w$ in $L^p(0,T;L^2(X))$.
\end{lemma}

\begin{proof}
 For $0 < \delta < \delta_0$, let $\left( A^\delta_m \right)_{m\in M_\delta}$ be a covering of $X$ in disjoint cubes of length $\delta$ and let $R^\delta$ be the corresponding $L^2$ projection operator as defined in Proposition \ref{prop_grid_projection}. We write
$$
w_k - w = (w_k - R^\delta w_k)  + (R^\delta w_k - R^\delta w) + (R^\delta w - w)
$$
and we establish the convergence to $0$ of each bracketed term in the right-hand side. First, in view of Proposition \ref{prop_grid_projection}
$$
\| w_k(t,\cdot) - R^\delta w_k(t,\cdot) \|_{L^2(X)} \lsim  \sup_{|z|\leq \delta}\|w_k(t,\cdot+z) - w_k(t,\cdot)\|_{L^2(X)}
$$
implying that
$$
\| w_k - R^\delta w_k \|_{L^p(0,T;L^2(X) )} \lsim T(\delta).
$$
Setting $v_k = w_k - R^\delta w_k$, $k\in \N$, we have, if $p = \infty$, $\| v_k(t,\cdot) \|_{L^2(X)} \lsim T(\delta)$ for a.e.\ $t\in (0,T)$. 
Since ${\rm Id}-R^\delta:L^2(X)\to L^2(X)$ is linear, the weak convergence of $w_k(t,\cdot)$ implies that $v_k(t,\cdot)\weakto v(t,\cdot) \coloneq w(t,\cdot)-R^\delta w(t,\cdot)$ weakly in $L^2(X)$, and thus that
$$
\|w(t,\cdot)-R^\delta w(t,\cdot)\|_{L^2(X)} \leq \liminf_{k\to +\infty} \| v_k(t,\cdot) \|_{L^2(X)} \lsim T(\delta). 
$$
For $p<\infty$, we have, using the above weak convergence of $(v_k(t,\cdot))_{k\in\N}$ and Fatou's lemma,
$$
\int_0^T \|v(t,\cdot)\|^p_{L^2(X)} \d t\leq \int^T_0 \liminf_{k\to +\infty} \|v_k(t,\cdot)\|^p_{L^2(X)} \d t\leq \liminf_{k\to +\infty} \int^T_0 \|v_k(t,\cdot)\|^p_{L^2(X)} \d t \lsim T^p(\delta). 
$$
Hence, for any $p$,
$$
\| w - R^\delta w \|_{L^p(0,T;L^2(X) )} \lsim T(\delta).
$$
Finally, 
$$
R^\delta w_k - R^\delta w = \sum_{m\in M_\delta} a^\delta_{km}(t) \mathbbm{1}_{A^\delta_m\cap X}, \quad \mbox{with} \quad 
 a^\delta_{km}(t)= {1\over \delta^d} \int_{A^\delta_m} (w_k(t,\x) - w(t,\x))\d\x,
$$
Since the covering $\left( A^\delta_m \right)_{m\in M_\delta}$ is finite and since, for all $m\in M_\delta$, the term $a^\delta_{km}(t)$ converges uniformly in time to zero, it results that $(R^\delta w_k - R^\delta w)$ converges as $k\to +\infty$ to zero in $L^{\infty}((0,T)\times X)$.

Gathering the estimates, we have that 
$$
\| w_k - w  \|_{L^p(0,T; L^2(X))} \lsim 2 T( \delta ) + \| R^\delta w_k - R^\delta w  \|_{L^p(0,T; L^2(X))}.  
$$
Passing to the superior limit as $k \to +\infty$, we deduce that $\limsup_{k\to +\infty} \| w_k - w \|_{L^p(0,T; L^2(X))} \lsim 2 T(\delta)$
which yields, letting $\delta\to 0$, $\limsup_{k\to +\infty} \| w_k - w \|_{L^p(0,T; L^2(X))} = 0$.
\end{proof}

\subsection{Appendix 2}
\label{appendix.a2}

\begin{lemma}\label{lemma_limitconformityDu}
 Let $(\D_\bu^l)_{l\in{\mathbb N}}$ be a sequence of GDs assumed to satisfy the coercivity and limit-conformity properties.
 Let $(\bu^l)_{l\in \N}$ be a sequence of vectors with $\bu^l \in X^0_{\D^l_\bu}$ such that there exist $C$ independent of $l\in \N$ with $\|\bu^l\|_{\D_\bu} \leq C$.   Then, there exists $\bar\bu\in \U_0$ such that, up to a subsequence, the following weak limits hold:  
\begin{equation*}
\begin{array}{lll}  
  & \Pi_{\D^l_\bu} \bu^l \weakto \bar \bu & \mbox{in } L^2(\Omega)^d,\\
  & \bbeps_{\D^l_\bu} (\bu^l) \weakto \bbeps(\bar\bu) & \mbox{in } L^2(\Omega,\S_d(\R)),\\
  & \div_{\D^l_\bu} (\bu^l) \weakto \div(\bar\bu) & \mbox{in } L^2(\Omega),\\
  & \jump{\bu^l}_{\D^l_\bu} \weakto \jump{\bar\bu} & \mbox{in } L^2(\Gamma). 
\end{array}
\end{equation*} 
\end{lemma}

\begin{proof} 
By assumption the sequence $(\|\bbeps_{\D^l_\bu}\|_{L^2(\Omega,\S_d(\R))})_{l\in\N}$ is bounded which implies, 
from the coercivity property, that the sequences $(\|\Pi_{\D^l_\bu} \bu^l\|_{L^2(\Omega)})_{l\in\N}$ and $(\|\jump{\bu^l}_{\D^l_\bu}\|_{L^2(\Gamma)})_{l\in \N}$ are also bounded. Hence there exist $\bar\bu \in L^2(\Omega)^d$, $\bar\bbeps\in L^2(\Omega,\S_d(\R))$ and
$\bar g\in L^2(\Gamma)$ such that, up to a subsequence, one has 
\begin{equation*}
\begin{array}{lll}  
  & \Pi_{\D^l_\bu} \bu^l \weakto \bar \bu & \mbox{in } L^2(\Omega)^d,\\
  & \bbeps_{\D^l_\bu} (\bu^l) \weakto \bar\bbeps & \mbox{in } L^2(\Omega,\S_d(\R)),\\[1ex]
  & \jump{\bu^l}_{\D^l_\bu} \weakto \bar g & \mbox{in } L^2(\Gamma). 
\end{array}
\end{equation*} 
Passing to the limit in the definition of the limit-conformity yields, for any $\bbsig \in$ $C_\Gamma^\infty(\Omega\setminus\overline\Gamma,\S_{d}(\R))$, 
$$
\int_\Omega \(\bbsig: \bar \bbeps + \bar\bu \cdot \div(\bbsig)\)\d \x - \int_\Gamma (\bbsig \n^+)\cdot\n^+ \bar g ~\d\sigma(\x) = 0. 
$$
Selecting first $\bbsig$ with a compact support in $\Omega\backslash \Gamma$, and then a generic $\bbsig$, 
it results that $\bar\bu \in \U_0$ with $\bar\bbeps = \bbeps(\bar \bu)$ and $\bar g = \jump{\bar\bu}$. Since $\div_{\D^l_\bu} (\bu^l) = \mbox{\rm Trace}(\bbeps_{\D^l_\bu} (\bu^l))$, it also holds that  $\div_{\D^l_\bu} (\bu^l) \weakto \div(\bar\bu)$  in $L^2(\Omega)$.

\end{proof}

Let us fix $\bar p^\alpha \in V_0$, $\alpha\in \{\g,\l\}$, $\mathbf f \in L^2(\Omega)^d$, and  define
  $$
  \bar p^E_m = \sum_{\alpha\in \{\g,\l\}} \bar p^\alpha S^\alpha_m(\bar p_c) - U_m(\bar p_c) \quad \mbox{ and } \quad \bar p^E_f = \sum_{\alpha\in \{\g,\l\}}  \gamma \bar p^\alpha S^\alpha_f(\gamma \bar p_c) - U_f(\gamma\bar p_c). 
  $$
  with $\bar p_c = \bar p^\g - \bar p^\l$.
  We consider the solution $\bar \bu \in \U_0$ of the following variational formulation
  \begin{equation}
    \label{VF_mechanics}
  \int_\O \( \bbsig(\bar \bu): \bbeps(\bar \bv) - b ~\bar p_m^E \div(\bar \bv)\) \d\x + \int_\G \bar p_f^E ~\jump{\bar \bv}  \d\sigma(\x)  = \int_\Omega \mathbf{f}\cdot\bar\bv ~\d\x,\qquad\forall \bar\bv\in \U_0. 
  \end{equation}

  Let us take $p^\alpha\in \D_p$, $\alpha\in \{\g,\l\}$, $p_c = p^\g - p^\l$ and
  $$
  p^E_m = \sum_{\alpha\in \{\g,\l\}} p^\alpha S^\alpha_m(p_c) - U_m(p_c) \quad \mbox{ and } \quad p^E_f = \sum_{\alpha\in \{\g,\l\}} p^\alpha S^\alpha_f(p_c) - U_f(p_c).
$$
  We consider the following gradient scheme for \eqref{VF_mechanics}: Find $\bu\in X_{\D_\bu}^0$ such that, for all $\bv\in X_{\D_\bu}^0$,
  \begin{equation}
    \label{GD_mechanics}
  \dsp \int_\Omega \( \bbsig_{\D_u}(\bu) : \bbeps_{\D_\bu}(\bv)  
    - b~(\Pi_{\D_p}^m p_m^E)  \div_{\D_\bu}(\bv)\)  \d\x + 
     \int_\Gamma (\Pi_{\D_p}^f p_f^E)  \jump{\bv}_{\D_\bu} \d\sigma(\x) 
    = \int_\Omega \mathbf{f} \cdot \Pi_{\D_\bu} \bv ~\d\x.
  \end{equation}
The Lax-Milgram theorem ensures that the exact solution $\bar\bu$ and approximate solution $\bu$ exist and are unique. The following proposition provides an error estimate. 

\begin{proposition}\label{error_estimate_mechanics}
  Let $\bar\bu \in \U_0$ be the solution of   \eqref{VF_mechanics} and $\bu\in X_{\D_\bu}^0$ the solution
  of the gradient scheme \eqref{GD_mechanics}. Then, there exists a hidden constant depending only on the coercivity constant $C_{\D_\bu}$ and on the physical data such that the following error estimate holds 
  \begin{equation}
    \label{error_estimate_u}
    \begin{array}{ll}
      & \|\bbeps_{\D_\bu}(\bu)  - \bbeps(\bar \bu)\|_{L^2(\Omega,\S_{d}(\R))}
      + \|\Pi_{\D_\bu} \bu  -\bar \bu\|_{L^2(\Omega)} + \|\jump{\bu}_{\D_\bu}  - \jump{\bar \bu}\|_{L^2(\Gamma)} \\[2ex]
      & \lsim {\cal S}_{\D_\bu}(\bar \bu) + {\cal W}_{\D_\bu}(\bbsig(\bar\bu) -b \, \bar p^E_m \mathbb I) + \|\bar p^E_m - \Pi_{\D_p}^m p^E_m \|_{L^2(\Omega)}
      + \|\bar p^E_f - \Pi_{\D_p}^f p^E_f \|_{L^2(\Gamma)}. 
      \end{array}
    \end{equation}
As a consequence, if $(\D_\bu^l)_{l\in\N}$ is a sequence of coercive, consistent and limit-conforming GDs, if $\bu^l$ is the solution of \eqref{GD_mechanics} for $\D_\bu=\D_\bu^l$, if $(\D_p^l)_{l\in\N}$ is a sequence of GDs and $p^{\alpha,l}\in X_{\D_p^l}^0$, $l\in\N$, are such that $\Pi_{\D_p^l}^mp^{E,l}_m\to \bar p^E_m$ in $L^2(\Omega)$ and $\Pi_{\D_p^l}^fp^{E,l}_f\to \bar p^E_f$ in $L^2(\Gamma)$, then, as $l\to+\infty$,
\begin{equation}\label{eq:convergence_u}
\begin{aligned}
\bbeps_{\D_\bu^l}(\bu^l) \to{}&\bbeps(\bar \bu)&&\mbox{ in $L^2(\Omega,\S_d(\R))$},\\
\Pi_{\D_\bu^l} \bu^l \to{}&\bar \bu&&\mbox{ in $L^2(\Omega)^d$},\\
\jump{\bu^l}_{\D_\bu^l}\to{}&\jump{\bar \bu}&&\mbox{ in $L^2(\Gamma)$}.
\end{aligned}
\end{equation}
\end{proposition}

\begin{proof}
  We note that even though ${\cal W}_{\D_\bu}$ was considered, in the definition of limit-conformity of a sequence of GDs, only on $C_\Gamma^\infty(\Omega\setminus\overline\Gamma,\S_{d}(\R))$, it can be defined on
  \begin{align*}
    H_{\rm div, \Gamma}(\Omega\setminus\overline\Gamma;\S_d(\R)) & \coloneq \left\{\bbsig\in L^2(\Omega;\S_d(\R))\,:\,\div(\bbsig)|_{\Omega^\beta}\in L^2(\Omega^\beta)^d, \beta\in \Xi, \right.\\
   & \qquad \left.\bbsig^+ \n^+ + \bbsig^- \n^- = {\mathbf 0} \mbox{ on } \Gamma, \,(\bbsig^+ \n^+) {\times} \n^+ = {\mathbf 0} \mbox{ on } \Gamma\right\}, 
  \end{align*}
  where $(\Omega^\beta)_{\beta\in \Xi}$ are the connected components of $\Omega\setminus\overline\Gamma$.
  Setting $\bbsig = \bbsig(\bar\bu) - b \,\bar p^E_m \mathbb I\in H_{\rm div, \Gamma}(\Omega\setminus\overline\Gamma; \S_d(\R)) $ as an argument of ${\cal W}_{\D_\bu}$ and using
$\div \bbsig = -\mathbf{f}$, we obtain that for all $\bv\in X^0_{\D_\bu}$
$$
\begin{array}{ll}
  & \dsp \left | \int_\Omega \( ( \bbsig(\bar\bu) - \bbsig_{\D_\bu}(\bu) ): \bbeps_{\D_\bu}(\bv)-b (\bar p^E_m - \Pi_{\D_p}^m p^E_m) \div_{\D_\bu}(\bv)) \) \d\x 
   \dsp + \int_\Gamma (\bar p^E_f - \Pi_{\D_p}^f p^E_f) \jump{\bv}_{\D_\bu} \d\sigma(\x) \right | \\ [2ex]
  & \dsp \qquad  \leq \|\bv\|_{\D_\bu} {\cal W}_{\D_\bu}( \bbsig(\bar\bu) -b \,\bar p^E_m \mathbb I ). 
\end{array}  
$$
Setting $\bv = P_{\D_\bu} \bar\bu - \bu$, where $P_{\D_\bu}\bar\bu$ realizes the minimum in $\mathcal S_{\D_\bu}(\bar\bu)$, we infer
$$
 \| P_{\D_\bu} \bar\bu - \bu \|_{\D_\bu}
  \lsim {\cal S}_{\D_\bu}(\bar \bu) + {\cal W}_{\D_\bu}(\bbsig(\bar\bu) -b \,\bar p^E_m \mathbb I) + \|\bar p^E_m - \Pi_{\D_p}^m p^E_m \|_{L^2(\Omega)}
      + \|\bar p^E_f - \Pi_{\D_p}^f p^E_f \|_{L^2(\Gamma)}.
$$
Combined with the definition of $C_{\D_\bu}$, the estimates above establish \eqref{error_estimate_u}.

Under the assumptions in the second part of the proposition, the hidden constant in \eqref{error_estimate_u} is independent of $l$, the last two terms in the left-hand side of this estimate converge to $0$ as $l\to+\infty$, as well as  $S_{\D_\bu^l}(\bu)$ by definition of the consistency of the sequence of GDs. When its argument $\bbsig$ is in the vector space $C_\Gamma^\infty(\Omega\setminus\overline\Gamma,\S_{d}(\R))$, ${\cal W}_{\D_\bu^l}(\bbsig)$ also converges to $0$ by definition of limit-conformity; since this space is dense in $H_{\rm div, \Gamma}(\Omega\setminus\overline\Gamma;\S_d(\R))$,
the arguments in \cite[Lemma 2.17]{gdm} show that this convergence also holds for the argument $\bbsig=\bbsig(\bar\bu) -b \,\bar p^E_m \mathbb I$.
Estimate \eqref{error_estimate_u} therefore yields the convergences \eqref{eq:convergence_u}.
\end{proof}
\begin{acknowledgements}
We are grateful to Andra and to the Australian Research Council's Discovery Projects (project DP170100605) funding scheme for partially supporting this work.
\end{acknowledgements}

%
%%%%%%%%%%%%%%%%%%
%
%%-----------------------------
%%      your bibliography
%%-----------------------------
%\bibliographystyle{abbrv}
%\clearpage
\bibliographystyle{plain}
\bibliography{GDM_Paper}
\end{document}